\pdfoutput=1
\documentclass[twoside]{article}
\usepackage[letterpaper, left=3.5cm, right=3.5cm, top=4cm, bottom=2.5cm]{geometry}
\usepackage{graphicx}
\usepackage{amsmath,amssymb,amsthm}
\usepackage{authblk}
\usepackage[utf8]{inputenc}
\usepackage{microtype}
\usepackage{mathrsfs} 
\usepackage{enumitem}
\usepackage{calc}
\usepackage{esint}
\usepackage{fancyhdr}
\usepackage{xcolor}
\usepackage[labelfont = bf]{caption}
\usepackage{doi}
\usepackage{subfigure}
\usepackage[numbers]{natbib}
\usepackage{float}
\usepackage{stmaryrd}
\usepackage{mathtools}
\usepackage{booktabs}
\usepackage{colortbl}
\usepackage{tabulary}
\usepackage{diagbox}
\usepackage{caption}
\usepackage{cleveref}
\usepackage{commath}
\usepackage{multirow} 

\usepackage{physics}
\usepackage{amsmath}
\usepackage{tikz}
\usepackage{mathdots}
\usepackage{yhmath}
\usepackage{cancel}
\usepackage{color}
\usepackage{siunitx}
\usepackage{array}
\usepackage{multirow}
\usepackage{amssymb}
\usepackage{gensymb}
\usepackage{tabularx}
\usepackage{extarrows}
\usepackage{booktabs}
\usepackage{color,soul}
\usetikzlibrary{fadings}
\usetikzlibrary{patterns}
\usetikzlibrary{shadows.blur}
\usetikzlibrary{shapes}

\newtheorem{theorem}{Theorem}[section]
\numberwithin{equation}{section}

\newtheorem{definition}[theorem]{Definition}
\newtheorem{corollary}[theorem]{Corollary}
\newtheorem{remark}[theorem]{Remark}
\newtheorem{lemma}[theorem]{Lemma}

\usepackage{titlesec}
\titleformat{\section}{\normalfont\scshape\centering}{\thesection.}{0.5em}{}
\titleformat*{\subsection}{\itshape}
\titleformat*{\subsubsection}{\itshape}

\providecommand{\keywords}[1]
{
	{\small\emph{Keywords:} #1}
}

\providecommand{\MSC}[1]
{
	{\small\emph{AMS MSC (2020):~~} #1}
}

\definecolor{denim}{rgb}{0.08, 0.38, 0.74}
\definecolor{byzantium}{rgb}{0.44, 0.16, 0.39} 
\definecolor{shamrockgreen}{rgb}{0.0, 0.62, 0.38} 

\usepackage{hyperref}
\hypersetup{
	colorlinks=true,
	linkcolor=denim,
	citecolor = shamrockgreen,
	filecolor=magenta, 
	urlcolor=byzantium,
} 

\providecommand{\jumptmp}[2]{#1\llbracket{#2}#1\rrbracket}
\providecommand{\jump}[1]{\jumptmp{}{#1}}

\pdfoutput=1
\begin{document}
	\setlength{\abovedisplayskip}{5.5pt}
	\setlength{\belowdisplayskip}{5.5pt}
	\setlength{\abovedisplayshortskip}{5.5pt}
	\setlength{\belowdisplayshortskip}{5.5pt}

	\title{\vspace{-15mm}Convection Effects and Optimal Insulation: \\ Modelling and Analysis\thanks{This work is partially supported by the Office of Naval Research (ONR) under Award NO: N00014-24-1-2147, NSF grant DMS-2408877, the Air Force Office of Scientific Research (AFOSR) under Award NO: FA9550-25-1-0231.}}\vspace{-1.5mm}
	\author[1]{Harbir Antil\thanks{Email: \url{hantil@gmu.edu}}}
	\author[2]{Alex Kaltenbach\thanks{Email: \url{kaltenbach@math.tu-berlin.de}}}
	\author[3]{Keegan L.\ A.\ Kirk\thanks{Email: \url{kkirk6@gmu.edu}}\vspace{-1.5mm}}
	\date{\today\vspace{-2.5mm}} 
	\affil[1,3]{\small{Department of Mathematical Sciences and the Center for Mathematics and Artificial Intelligence (CMAI), George Mason University, Fairfax, VA 22030, USA.}}
	\affil[2]{\small{Institute of Mathematics, Technical University of Berlin, Stra\ss e des 17.\ Juni 136, 10623 Berlin}\vspace{-2.5mm}}
	\maketitle

	\pagestyle{fancy}
	\fancyhf{}
	\fancyheadoffset{0cm}
	\addtolength{\headheight}{-0.25cm}
	\renewcommand{\headrulewidth}{0pt} 
	\renewcommand{\footrulewidth}{0pt}
	\fancyhead[CO]{\textsc{Modelling of an optimal insulation problem}}
	\fancyhead[CE]{\textsc{H. Antil, A. Kaltenbach, and K.  Kirk}}
	\fancyhead[R]{\thepage}
	\fancyfoot[R]{}
	
	\begin{abstract}
		In this paper, we study an insulation problem that seeks to determine the optimal~distri\-bution of a given amount $m>0$ of insulating material coating an insulated boundary part $\Gamma_I\subseteq \partial\Omega$ of a thermally conducting body $\Omega\subseteq \mathbb{R}^d$, $d\in \mathbb{N}$, subject to convective heat transfer.  
        The \textit{`thickness'} of the insulating layer $\Sigma_{I}^{\varepsilon}\subseteq \mathbb{R}^d$ is given locally via $\varepsilon \mathtt{d}$, where $\varepsilon>0$ denotes the (arbitrarily small) conductivity and $\mathtt{d}\colon \Gamma_{I}\to [0,+\infty)$ the (to be determined) distribution of the  insulating material.
        Then, the physical process is modelled by the~stationary heat equation in the insulated thermally conducting body $\Omega_{I}^{\varepsilon}\coloneqq \Omega\cup\Sigma_{I}^{\varepsilon}$ 
        with Robin-type boundary conditions on the interacting insulation boundary $\Gamma_I^{\varepsilon}\subseteq \partial\Omega_{I}^{\varepsilon}$ (reflecting convective heat transfer  between the thermally conducting body $\Omega$ and its surrounding medium)~as~well~as Dirichlet and Neumann boundary conditions at the remaining boundary parts, \textit{i.e.},  $\partial\Omega_{I}^{\varepsilon}\setminus \Gamma_I^{\varepsilon}$.
        
        More \hspace{-0.1mm}precisely, \hspace{-0.1mm}we \hspace{-0.1mm}establish \hspace{-0.1mm}$\Gamma(L^2(\mathbb{R}^d))$-convergence \hspace{-0.1mm}of \hspace{-0.1mm}the \hspace{-0.1mm}heat \hspace{-0.1mm}loss \hspace{-0.1mm}formulation~(as~\hspace{-0.1mm}${\varepsilon\hspace{-0.175em}\to\hspace{-0.175em} 0^+}$), in the case that the thermally conducting body $\Omega$ is a bounded Lipschitz domain having a $C^{1,1}$-regular or piece-wise flat insulated boundary $\Gamma_I$.
	\end{abstract}
	
	
	\MSC{35J25; 35Q93; 49J45; 80A20}
	
	\section{Introduction}\thispagestyle{empty}

    \hspace{5mm}The control of heat exchange between a thermally conducting body and its surrounding medium plays a  critical role in many industrial applications spanning almost all fields~of~engineering. Some examples include the design of energy-efficient buildings,  shielding of sensitive components in electronics and machinery, and  protection of passengers and crew~during~air-~and~spacecraft~travel.
    Often, this control is achieved passively through thermal insulation. If only a limited budget of~\mbox{insulating}~\mbox{material}~is~\mbox{allowable}, as is the case when there are strict size or mass constraints on the design, the problem of its optimal distribution becomes a question of both theoretical and practical significance
    (see \cite{Claesson1980,DellaPietraOliva2025}). In \cite{Buttazzo1988}, such an optimization problem is studied under the assumption that thermal conduction is the only mechanism of heat transfer at the body's surface. 
    However, this precludes many important applications, particularly in aerospace and aeronautic engineering, where the dominant mechanism of heat transfer may be convection, radiation, or some combination thereof. 
    In the case where convection is the dominant heat transfer mechanism, Robin-type boundary conditions provide a natural mathematical model for the underlying physics (\textit{cf}.\ \cite{NASAthermalGuidebook}).

    \subsection{Related contributions}\enlargethispage{5mm}

    \subsubsection{Optimal insulation of thermally conducting body under \underline{conductive} heat transfer}
    
    \hspace{5mm}The first PDE-based shape optimization framework for the optimal insulation of a thermally conducting body, when heat transfer with the environment is governed by conduction (that is, Dirichlet boundary conditions are imposed on the boundary of the insulated body), was proposed by Buttazzo (\textit{cf.}~\cite{Buttazzo1988,Buttazzo1988c}).   In this setting, one considers a bounded domain $\Omega \subseteq \mathbb{R}^d$, $d \in \mathbb{N}$, representing the \emph{thermally conducting body}, with material-specific \emph{thermal conductivity} $\lambda > 0$ and \emph{heat source density} $f \in L^2(\Omega)$. An \emph{insulating layer} $\Sigma_\varepsilon \subseteq \mathbb{R}^d\setminus \Omega$ is placed around the body, satisfying $\partial \Omega \subseteq \partial \Sigma_\varepsilon$. The layer has local thickness $\varepsilon \mathtt{d}$, where $\varepsilon > 0$ is the thermal conductivity of the insulating material and $\mathtt{d} \colon \partial \Omega \to [0,+\infty)$ is a distribution function to be determined. The resulting insulated body is $\Omega_\varepsilon \coloneqq \Omega \cup \Sigma_\varepsilon$.  Then, one seeks to minimize the \emph{heat loss} functional $E_\varepsilon^{\mathtt{d}} \colon H_0^1(\Omega_\varepsilon) \to \mathbb{R}$, for every $v_\varepsilon \in H_0^1(\Omega_\varepsilon)$ defined by  
%
    \begin{align}\label{intro:E_vareps_h_conductive}
        \smash{E_\varepsilon^{\mathtt{d}}(v_\varepsilon)\coloneqq \tfrac{\lambda}{2}\|\nabla v_\varepsilon\|_{\Omega}^2+\tfrac{\varepsilon}{2}\|\nabla v_\varepsilon\|_{\Sigma_\varepsilon}^2-(f,v_\varepsilon)_{\Omega}}\,.
    \end{align} 
    The \hspace{-0.15mm}direct \hspace{-0.15mm}method \hspace{-0.15mm}in \hspace{-0.15mm}the \hspace{-0.15mm}calculus \hspace{-0.15mm}of \hspace{-0.15mm}variations \hspace{-0.15mm}yields 
   \hspace{-0.15mm}the \hspace{-0.15mm}existence \hspace{-0.15mm}of \hspace{-0.15mm}a \hspace{-0.15mm}unique \hspace{-0.15mm}minimizer~\hspace{-0.15mm}${u_\varepsilon^{\mathtt{d}}\hspace{-0.175em}\in \hspace{-0.175em}H^1_0(\Omega_\varepsilon)}$ to the heat loss functional \eqref{intro:E_vareps_h_conductive}, which formally satisfies the Euler--Lagrange equations
    \begin{subequations}\label{intro:ELE_Eepsh_conductive}
    \begin{align}\label{intro:ELE_Eepsh_conductive.1}
            -\lambda\Delta u_\varepsilon^{\mathtt{d}}&=f&&\quad \text{ a.e.\ in }\Omega\,,\\ 
            -\varepsilon\Delta u_\varepsilon^{\mathtt{d}}&=0&&\quad \text{ a.e.\ in }\Sigma_\varepsilon\,,\label{intro:ELE_Eepsh_conductive.2}\\ 
            u_\varepsilon^{\mathtt{d}}&=0&&\quad \text{ a.e.\ on }\partial\Omega_{\varepsilon}\,,\label{intro:ELE_Eepsh_conductive.3}\\ 
            \lambda\nabla (u_\varepsilon^{\mathtt{d}}|_{\Sigma_{\varepsilon}})\cdot n&=\varepsilon\nabla (u_\varepsilon^{\mathtt{d}}|_{\Omega})\cdot n&&\quad \text{ a.e.\ on } \partial\Omega\,,\label{intro:ELE_Eepsh_conductive.4} 
    \end{align}
    \end{subequations}
    where $n\colon \partial\Omega\to \mathbb{S}^{d-1}$ denotes the outward unit normal vector field to $\Omega$. 
    

    From the rich literature on asymptotic analysis (as $\varepsilon\to 0^+$) for the heat loss functional \eqref{intro:E_vareps_h_conductive} 
    (\textit{cf}.\  \cite{BrezisCaffarelliFriedman1980,CaffarelliFriedman1980,AcerbiButtazzo1986,AcerbiButtazzo1986b,ButtazzoKohn1987,ButtazzoDalMasoMosco1989,BoutkridaMossinoMoussa1999,BoutkridaGrenonMossinoMoussa2002,MossinoVanninathan2002,PietraNitschScalaTrombetti2021,AcamporaCristoforoniNitschTrombetti2024,AKK2025_modelling}), 
    we want point out the following two contributions:
    \begin{itemize}[noitemsep,topsep=2pt,leftmargin=!,labelwidth=\widthof{$\bullet$}]
        \item[$\bullet$] If $\partial\Omega\in C^{1,1}$, which is equivalent to 
        $n\in \smash{(C^{0,1}(\partial\Omega))^d}$ (\textit{cf}.\ Remark \ref{rem:examples}(\hyperlink{rem:examples.i}{i})),~given $\mathtt{d}\in  C^{0,1}(\partial\Omega)$ with $\mathtt{d}\ge \smash{\mathtt{d}_{\textup{min}}}$ a.e.\ on $\partial\Omega$, for some~$\smash{\mathtt{d}_{\textup{min}}}>0$, defining the  insulating~layer~via
    \begin{align}\label{intro:sigma_eps_conductive_smooth}
        \Sigma_\varepsilon\coloneqq \smash{\big\{s+ tn(s)\mid s\in \partial\Omega\,,\; t\in [0,\varepsilon \mathtt{d}(s))\big\}}\,,
    \end{align}
    {Acerbi} and {Buttazzo} (\textit{cf}.\  \cite[Thm. II.2]{AcerbiButtazzo1986}) proved that the limit functional (as $\varepsilon\to 0^{+}$)~of~\eqref{intro:E_vareps_h_conductive} (in the sense of $\Gamma(L^2(\mathbb{R}^d))$-convergence) is given via 
    $\smash{E}^\mathtt{d}\colon \hspace{-0.05em}H^1(\Omega)\hspace{-0.05em}\to\hspace{-0.05em} \mathbb{R}$, for every $v\hspace{-0.05em} \in \hspace{-0.05em}H^1(\Omega)$~\mbox{defined}~by
    \begin{align}\label{intro:E_h_smooth}
        \smash{E^{\mathtt{d}}(v)\coloneqq \tfrac{\lambda}{2}\|\nabla v\|_{\Omega}^2+\tfrac{1}{2}\|\smash{\mathtt{d}^{-\smash{\frac{1}{2}}}}v\|_{\partial\Omega}^2-(f,v)_{\Omega}}\,.
    \end{align} 
    The \hspace{-0.1mm}assumption \hspace{-0.1mm}$n\hspace{-0.15em}\in\hspace{-0.15em} \smash{(C^{0,1}(\partial\Omega))^d}$ \hspace{-0.1mm}ensures for sufficiently small $\varepsilon > 0$, the mapping 
    $\Phi_\varepsilon \colon D_\varepsilon\coloneqq\bigcup_{s\in \partial\Omega}{\{s\}\times [0,\varepsilon \mathtt{d}(s))}\to \Sigma_\varepsilon$, defined by $\Phi_\varepsilon(s,t)\coloneqq s+ tn(s)$ for all $(s,t)^\top\in D_\varepsilon$, 
    is bi-Lipschitz continuous. \ As a consequence, there are no gaps 
    (\textit{i.e.}, insulation is applied everywhere)
    ~or~self-intersections  
    (\textit{i.e.}, insulation is applied only once)  
    in the insulating layer $\Sigma_\varepsilon$ (\textit{cf}.\ Figure~\ref{fig:transversality}). 
    %
     \item[$\bullet$] 

      If $\partial\Omega\in C^{0,1}$ is piece-wise flat, given $\mathtt{d}\in  C^{0,1}(\partial\Omega)$ with $\mathtt{d}\ge \smash{\mathtt{d}_{\textup{min}}}$ a.e.\ on $\partial\Omega$,~for~some~$\smash{\mathtt{d}_{\textup{min}}}>0$, and a unit-length (globally) transversal 
      $k\in (C^{0,1}(\partial\Omega))^d$, defining the  insulating~layer~via
     \begin{align}\label{intro:sigma_eps2_conductive}
         \Sigma_\varepsilon\coloneqq \smash{\big\{s+ tk(s)\mid s\in \partial\Omega\,,\; t\in [0,\varepsilon \mathtt{d}(s))\big\}}\,,
     \end{align}
     the~authors~(\textit{cf}.\ \cite[Thm.\ 5.1]{AKK2025_modelling}) proved that the limit functional (as $\varepsilon\to 0^{+}$)~of~\eqref{intro:E_vareps_h_conductive} (in the sense of $\Gamma(L^2(\mathbb{R}^d))$-convergence) is given via 
     $\smash{E}^\mathtt{d}\colon H^1(\Omega)\to \mathbb{R}$, for every $v \in H^1(\Omega)$~defined~by
     \begin{align}\label{intro:E_h}
         \smash{E}^{\mathtt{d}}(v)\coloneqq \tfrac{\lambda}{2}\|\nabla v\|_{\Omega}^2+\tfrac{1}{2}\|((k\cdot n)\smash{\mathtt{d})^{-\smash{\frac{1}{2}}}}v\|_{\partial\Omega}^2-(f,v)_{\Omega}\,.
     \end{align}
    \end{itemize}
    Both if $\partial\Omega\hspace{-0.05em}\in\hspace{-0.05em} C^{1,1}$ (in which case, we set $k\hspace{-0.05em}=\hspace{-0.05em}n\hspace{-0.05em}\in\hspace{-0.05em} (C^{0,1}(\partial\Omega))^d$) and if $\partial\Omega\hspace{-0.05em}\in\hspace{-0.05em} C^{0,1}$~is~\mbox{piece-wise}~flat,
    a unique minimizer $u^{\mathtt{d}}\in H^1(\Omega)$ to the $\Gamma$-limit functional  \eqref{intro:E_h} (which, in the case $\partial\Omega\in C^{1,1}$ and $k=n\in (C^{0,1}(\partial\Omega))^d$, reduces to  \eqref{intro:E_h_smooth}) exists and 
    formally 
    satisfies the Euler--Lagrange~equations 
    \begin{align}\label{intro:ELE_Eh}
        \begin{aligned}
            -\lambda \Delta u^{\mathtt{d}}&=f&&\quad \text{ a.e.\ in }\Omega\,,\\[-0.5mm]
          \lambda (k\cdot n) \mathtt{d}\nabla u^{\mathtt{d}}\cdot n+u^{\mathtt{d}}&=0&&\quad \text{ a.e.\ on }\partial\Omega\,.
        \end{aligned}
    \end{align} \newpage

    \subsubsection{Optimal insulation of thermally conducting body under \underline{convective} heat transfer}\enlargethispage{7mm}\vspace{-0.5mm}
    
    \hspace{5mm}The first contribution proposing a PDE-based shape optimization framework for optimal~insula\-tion of a thermally conducting body, when heat transfer with the environment is dominated by convection (\textit{i.e.}, Robin boundary conditions are imposed at  boundary~of~the~\mbox{insulated}~body), was proposed by Della Pietra \textit{et al}.\ (\textit{cf}.\ \cite{PietraNitschScalaTrombetti2021}). Therein, given the setup~of~the~previous subsection and, in addition,  
    a system-specific \textit{heat transfer coefficient} $\beta>0$,  one seeks to minimize  the   \textit{heat loss} functional $\smash{E}_\varepsilon^\mathtt{d}\colon H^1(\Omega_\varepsilon)\to \mathbb{R}$, for every $v_\varepsilon\in H^1(\Omega_\varepsilon)$ defined by
    \begin{align}\label{intro:E_vareps_h_convective}
        \smash{\smash{E}_\varepsilon^{\mathtt{d}}(v_\varepsilon)\coloneqq \tfrac{\lambda}{2}\|\nabla v_\varepsilon\|_{\Omega}^2+\tfrac{\varepsilon}{2}\|\nabla v_\varepsilon\|_{\Sigma_\varepsilon}^2+\tfrac{\beta}{2}\|v_\varepsilon
        \|_{\partial\Omega_{\varepsilon}}^2-(f,v_\varepsilon)_{\Omega}\,.}
    \end{align} 
    Since the heat loss functional \eqref{intro:E_vareps_h_convective}  is proper, strictly convex, weakly coercive, and lower semi-continuous,  
    the direct method in the calculus of variations yields the existence of a unique minimizer ${u_\varepsilon^{\mathtt{d}}\in H^1_0(\Omega_\varepsilon)}$, which formally satisfies the Euler--Lagrange equations\vspace{-0.5mm}
    \begin{align}\label{intro:ELE_Eepsh_convective}
        \begin{aligned}
            -\lambda\Delta u_\varepsilon^{\mathtt{d}}&=f&&\quad \text{ a.e.\ in }\Omega\,,\\[-0.5mm]
            -\varepsilon\Delta u_\varepsilon^{\mathtt{d}}&=0&&\quad \text{ a.e.\ in }\Sigma_\varepsilon\,,\\[-0.5mm]
            \varepsilon \nabla u_\varepsilon^{\mathtt{d}}\cdot n_\varepsilon^{\mathtt{d}}+\beta
            u_\varepsilon^{\mathtt{d}}
            &=0&&\quad \text{ a.e.\ on }\partial\Omega_\varepsilon\,,\\[-0.5mm]
            \lambda\nabla (u_\varepsilon^{\mathtt{d}}|_{\Sigma_{\varepsilon}})\cdot n&=\varepsilon\nabla (u_\varepsilon^{\mathtt{d}}|_{\Omega})\cdot n&&\quad \text{ a.e.\ on } \partial\Omega\,,
        \end{aligned}
    \end{align}
    where $\smash{n_{\varepsilon}^{\mathtt{d}}\colon \partial\Omega_{\varepsilon}\to \mathbb{S}^{d-1}}$ denotes the outward unit normal vector field to $\Omega_{\varepsilon}$.

    The literature on asymptotic analysis (as $\varepsilon\to 0^+$) for the heat loss functional~\eqref{intro:E_vareps_h_convective}~(or~for~the Euler--Lagrange \hspace{-0.1mm}equations \hspace{-0.1mm}\eqref{intro:ELE_Eepsh_convective}) \hspace{-0.1mm}is \hspace{-0.1mm}less \hspace{-0.1mm}rich; \hspace{-0.1mm}in \hspace{-0.1mm}fact, \hspace{-0.1mm}we \hspace{-0.1mm}are \hspace{-0.1mm}only \hspace{-0.1mm}aware \hspace{-0.1mm}of \hspace{-0.1mm}the \hspace{-0.1mm}following~\hspace{-0.1mm}\mbox{contribution}:  
\begin{itemize}[noitemsep,topsep=2pt,leftmargin=!,labelwidth=\widthof{$\bullet$}]
        \item[$\bullet$] In the case $\partial\Omega\in C^{1,1}$ 
        and given $\mathtt{d}\in  C^{0,1}(\partial\Omega)$ with $\mathtt{d}\ge \smash{\mathtt{d}_{\textup{min}}}$ a.e.\ on $\partial\Omega$, for some~$\smash{\mathtt{d}_{\textup{min}}}>0$, defining the  insulating layer via \eqref{intro:sigma_eps_conductive_smooth}, 
    Della Pietra \textit{et al}.\ (\textit{cf}.\ \cite[Thm.\ 3.1]{Buttazzo1988}) proved 
    that the limit functional (as $\varepsilon\to 0^{+}$)~of~\eqref{intro:E_vareps_h_convective} (in the sense of $\Gamma(L^2(\mathbb{R}^d))$-convergence) is given via 
    $\smash{E}^\mathtt{d}\colon \hspace{-0.05em}H^1(\Omega)\hspace{-0.05em}\to\hspace{-0.05em} \mathbb{R}$, for every $v\hspace{-0.05em} \in \hspace{-0.05em}H^1(\Omega)$~\mbox{defined}~by
    \begin{align}\label{intro:E_h_smooth_convective}
        \smash{\smash{E}^{\mathtt{d}}(v)\coloneqq \tfrac{\lambda}{2}\|\nabla v\|_{\Omega}^2+\tfrac{\beta}{2}\|(1+\beta\mathtt{d})^{-\smash{\frac{1}{2}}}
        v
        \|_{\partial\Omega}^2-(f,v)_{\Omega}\,.}
    \end{align}  
    A unique minimizer $u^{\mathtt{d}}\in H^1(\Omega)$ to the $\Gamma$-limit functional  \eqref{intro:E_h_smooth} exists and 
    formally 
    satisfies~the Euler--Lagrange equations\vspace{-0.5mm}
    \begin{align}\label{intro:ELE_Eh}
        \begin{aligned}
            -\lambda \Delta u^{\mathtt{d}}&=f&&\quad \text{ a.e.\ in }\Omega\,,\\[-0.5mm]
          \lambda (1+\mathtt{d})\nabla u^{\mathtt{d}}\cdot n+\beta
          u^{\mathtt{d}}
          &=0&&\quad \text{ a.e.\ on }\partial\Omega\,.
        \end{aligned}
    \end{align}\vspace{-7mm}
    \end{itemize}

    \subsection{New contributions}\vspace{-0.5mm}

    \hspace{5mm}The contributions of the paper are two-fold:
    \begin{enumerate}[noitemsep,topsep=2pt,leftmargin=!,labelwidth=\widthof{2.},font=\itshape]
    \item \textit{Generalization to partial insulation.} We extend the results of Della Pietra \textit{et al}.~\mbox{\cite[Thm.~3.1]{PietraNitschScalaTrombetti2021}} to the setting, where the insulating material is attached to only a boundary portion $\Gamma_I\subseteq \partial\Omega$. On the remaining boundary parts $\partial\Omega\setminus\Gamma_I$, Dirichlet and Neumann boundary conditions~are~imposed. Moreover, we also allow for a non-trivial ambient temperature (\textit{i.e.}, $u_{\infty}\not\equiv 0$).
    \item \textit{Generalization to piece-wise flat insulated boundaries.} We  extend the results of Della Pietra \textit{et al}.~\mbox{\cite[Thm.~3.1]{PietraNitschScalaTrombetti2021}} to Lipschitz domains with piece-wise flat insulated boundary parts $\Gamma_I\subseteq \partial\Omega$. This is achieved using the authors' techniques (\textit{cf}.\ \cite{AKK2025_modelling}) for 
    non-smooth~geometries. However, beyond the techniques developed in \cite{AKK2025_modelling}, the proof of the existence of a recovery~sequence, in the case of piece-wise flat insulated boundary $\Gamma_I$, requires an elaborate smoothing of the outward unit normal vector field $n\colon \Gamma_I\to \mathbb{S}^{d-1}$ to enable the construction of suitable~cut-off~functions.
    \end{enumerate} 
    
    \textit{This paper is organized as follows:} In Sec.\ \ref{sec:preliminaries}, we introduce the relevant notation. In addition,
    we briefly recall the most important definitions and results about the closest point projection,  the (un-)signed distance function and transversal vector field needed for the
    forthcoming analysis.\linebreak In Sec.\ \ref{sec:modelling}, resorting to the $\Gamma$-convergence results proved in Sec.\ \ref{sec:gamma_convergence}, we
    perform a model reduction (for $\varepsilon\to 0^+$) leading to a non-local and non-smooth convex minimization problem,
    whose minimization enables to compute (via an implicit formula) the optimal distribution~of~the~\mbox{insulating}~material. In Sec.\ \ref{sec:tools}, we prove several auxiliary technical tools needed to establish the
    main result~of~the~paper, \textit{i.e.}, the $\Gamma$-convergence result, in Sec.\ \ref{sec:gamma_convergence}.
    \newpage
	\section{Preliminaries}\label{sec:preliminaries}\vspace{-0.5mm} 

    \hspace{5mm}In \hspace{-0.1mm}this \hspace{-0.1mm}section, \hspace{-0.1mm}we \hspace{-0.1mm}collect \hspace{-0.1mm}basic 
    \hspace{-0.1mm}definitions \hspace{-0.1mm}and 
    \hspace{-0.1mm}results \hspace{-0.1mm}needed \hspace{-0.1mm}for \hspace{-0.1mm}the \hspace{-0.1mm}later \hspace{-0.1mm}$\Gamma$-convergence~\hspace{-0.1mm}\mbox{analysis}.\vspace{-1.5mm}

    \subsection{Assumptions on the thermally conducting body and boundary parts}\vspace{-0.5mm}

    \hspace{5mm}Throughout the paper, if not otherwise specified, we assume that
     the \emph{thermally~conducting~body} $\Omega\subseteq \mathbb{R}^d$, $d\in \mathbb{N}$, is a bounded Lipschitz domain with (topological) boundary $\partial\Omega$ and outward unit normal vector field $n\colon \partial\Omega\to \mathbb{S}^{d-1}\coloneqq\{x\in \mathbb{R}^d\mid \vert x\vert=1\}$. Moreover, we assume that $\partial\Omega$ is disjointly split into three (relatively) open boundary parts: an \emph{insulated boundary part} $\Gamma_I\subseteq \partial\Omega$, a \emph{Dirichlet boundary part} ${\Gamma_D\subseteq \partial\Omega}$, and a \emph{Neumann boundary part} $\Gamma_N\subseteq \partial\Omega$; more precisely, we have that $\partial\Omega=\overline{\Gamma}_I\cup \overline{\Gamma}_D\cup \overline{\Gamma}_N$ (\textit{cf}.\ Figure~\ref{fig:domain}).  In~this~connection, we always assume that $\Gamma_I\neq \emptyset$.\vspace{-2mm}

     \begin{figure}[H]
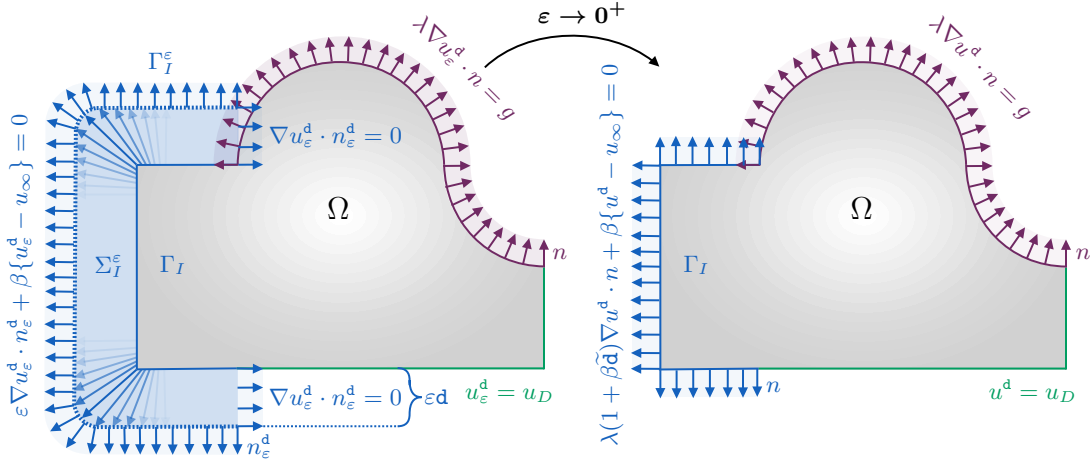

         \centering

  
\tikzset {_ww5bhfpt8/.code = {\pgfsetadditionalshadetransform{ \pgftransformshift{\pgfpoint{0 bp } { 0 bp }  }  \pgftransformscale{1 }  }}}
\pgfdeclareradialshading{_ac4q6gd96}{\pgfpoint{0bp}{0bp}}{rgb(0bp)=(1,1,1);
rgb(0bp)=(1,1,1);
rgb(12.857142857142856bp)=(0.89,0.89,0.89);
rgb(15.982142857142856bp)=(0.86,0.86,0.86);
rgb(25bp)=(0.82,0.82,0.82);
rgb(400bp)=(0.82,0.82,0.82)}

  
\tikzset {_iwybss6hl/.code = {\pgfsetadditionalshadetransform{ \pgftransformshift{\pgfpoint{0 bp } { 0 bp }  }  \pgftransformscale{1 }  }}}
\pgfdeclareradialshading{_hj1qd08up}{\pgfpoint{0bp}{0bp}}{rgb(0bp)=(1,1,1);
rgb(0bp)=(1,1,1);
rgb(12.857142857142856bp)=(0.89,0.89,0.89);
rgb(15.982142857142856bp)=(0.86,0.86,0.86);
rgb(25bp)=(0.82,0.82,0.82);
rgb(400bp)=(0.82,0.82,0.82)}
\tikzset{every picture/.style={line width=0.75pt}} 

\vspace{-1mm}

         \caption{A thermally conducting body $\Omega$ (\textcolor{gray}{gray}) with piece-wise flat insulated boundary $\Gamma_I$ (\textcolor{denim}{blue}) and Lipschitz continuous Dirichlet  $\Gamma_D$ (\textcolor{shamrockgreen}{green}) and Neumann  $\Gamma_N$ (\textcolor{byzantium}{purple})~boundary~part. \textit{Left:} before the  model reduction (as $\varepsilon\to 0^+$), where a Robin boundary condition is imposed at the interacting insulation boundary $\Gamma_{I}^{\varepsilon}$ (\textit{cf}.\ \eqref{def:some_notation.3}); \textit{Right:} after~the~model~reduction~(as~${\varepsilon\to 0^+}$),\linebreak where \hspace{-0.15mm}a \hspace{-0.15mm}Robin  \hspace{-0.15mm}boundary \hspace{-0.15mm}condition \hspace{-0.15mm}with \hspace{-0.15mm}variable \hspace{-0.15mm}coefficient \hspace{-0.15mm}is \hspace{-0.15mm}imposed~\hspace{-0.15mm}at~\hspace{-0.15mm}the~\hspace{-0.15mm}\mbox{insulated}~\hspace{-0.15mm}\mbox{boundary}~\hspace{-0.15mm}$\Gamma_{I}$.}\vspace{-1.5mm}
         \label{fig:domain}
     \end{figure}
    \subsection{Closest point projection and (un-)signed distance function}\vspace{-0.5mm}\enlargethispage{5mm}
    
    \hspace{5mm}The \textit{closest point projection} $\smash{\pi_{\partial\Omega}\colon \mathbb{R}^d\to 2^{\mathbb{R}^d}}$, where $\smash{2^{\mathbb{R}^d}}$ is the power set of $\smash{\mathbb{R}^d}$, for every~$x\in \smash{\mathbb{R}^d}$, is defined by\vspace{-0.5mm}
    \begin{align}\label{def:pi}
        \pi_{\partial\Omega}(x)\coloneqq \underset{y\in \partial\Omega}{\textup{arg\,min}}{\{\vert x-y\vert\}}\,.
    \end{align}
    Denote by $\smash{\mathrm{Med}(\partial\Omega)\hspace{-0.1em}\coloneqq\hspace{-0.1em} \{x\hspace{-0.1em}\in\hspace{-0.1em} \mathbb{R}^d\mid \textrm{card}(\pi_{\partial\Omega}(x))\hspace{-0.1em}>\hspace{-0.1em}1\}}$ the \textit{medial axis} --or \textit{skeleton}-- \textit{of $\Omega$}, \textit{i.e.},~the~set of points in which the closest point projection \eqref{def:pi} is not a singleton;~which~is~closed,~\mbox{$C^2$-rectifiable} (thus, a Lebesgue null set) (\textit{cf}.\ \cite{Alberti1994}), and has the same homotopy type~as~$\Omega$~(\textit{cf}.~\cite[Thm.~4.19]{Lieutier2004}).

 If $\partial\Omega\in C^{1,1}$, there exists $\delta>0$ such that in the \textit{tubular~\mbox{neighborhood}} 
    $\mathcal{N}(\partial\Omega)\coloneqq\partial\Omega+B_\delta^d(0)$, the closest point projection \eqref{def:pi} is single-valued, \textit{i.e.},~${\mathcal{N}(\partial\Omega)\cap \mathrm{Med}(\partial\Omega)=\emptyset}$. For a proof, see \cite[Lem.\ 14.17]{GilbargTrudinger1983} in the~case~${\partial\Omega\in C^2}$, which readily generalizes to the case $\partial\Omega\in  C^{1,1}$.
    

    If only $\Gamma_I\in C^{1,1}$, one can~find  $\delta\in C^{0,1}(\Gamma_I)$~with~${\delta>0}$~on~$\Gamma_I$~and~${\delta=0}$~on~$\partial\Gamma_I$ such~that~in~the \textit{insulated tubular neighborhood} $\mathcal{N}_{\delta}(\Gamma_I)\coloneqq\{ s+t n(s)\mid s\in \Gamma_I\,,\; t\in (-\delta(s),\delta(s))\}$,
    the~closest~point projection \eqref{def:pi} is single-valued, \textit{i.e.},  $\mathcal{N}_\delta(\Gamma_I)\cap \mathrm{Med}(\partial\Omega)=\emptyset$.
    
    If  $\Gamma_I$ is piece-wise flat (\textit{i.e.},  there  exist $L\hspace{-0.15em}\in\hspace{-0.15em}  \mathbb{N}$  boundary parts  $\smash{\Gamma_{I}^{\ell}}\hspace{-0.15em}\subseteq\hspace{-0.15em}\Gamma_I$, $\ell\hspace{-0.15em}=\hspace{-0.15em}1,\ldots,L$,~with~\mbox{constant} outward normal vectors $n_{\ell}\in \mathbb{S}^{d-1}$ such that $\smash{\bigcup_{\ell=1}^L{\Gamma_{I}^{\ell}}}= \Gamma_I$), one can~find $\delta\in C^{0,1}(\Gamma_{I})$ with $\delta\hspace{-0.1em}>\hspace{-0.1em}0$ in $\Gamma_{I}^{\ell} $ and $\delta\hspace{-0.1em}=\hspace{-0.1em}0$ on $\partial\Gamma_{I}^{\ell} $ for all $\ell\hspace{-0.1em}=\hspace{-0.1em}1,\ldots,L$ such that in the \textit{local insulated tubular~neighbor\-hoods} $\mathcal{N}_\delta(\Gamma_I^{\ell})\hspace{-0.1em}\coloneqq\hspace{-0.1em}\{s+t n_{\ell}\mid s\hspace{-0.1em}\in\hspace{-0.1em} \Gamma_I^{\ell}\,,\; t\hspace{-0.1em}\in\hspace{-0.1em} (-\delta(s),\delta(s))\}$,~${\ell\hspace{-0.1em}=\hspace{-0.1em}1,\ldots,L}$,~the~\mbox{closest}~point~\mbox{projection}~\eqref{def:pi} is single-valued, \textit{i.e.},  $\mathcal{N}_\delta(\Gamma_I^{\ell})\cap \mathrm{Med}(\partial\Omega)=\emptyset$ for all $\ell=1,\ldots,L$, as~well~as $\mathcal{N}_\delta(\Gamma_I^{\ell})\cap \mathcal{N}_\delta(\Gamma_I^{\ell'})=\empty$~if~${\ell\neq \ell'}$.\newpage


    In the later $\Gamma$-convergence analyses (especially the proof of the $\limsup$-estimate, \textit{cf}.\ Lemma~\ref{lem:limsup_case2}), it \hspace{-0.1mm}is \hspace{-0.1mm}central \hspace{-0.1mm}to \hspace{-0.1mm}measure \hspace{-0.1mm}distances \hspace{-0.1mm}of \hspace{-0.1mm}exterior \hspace{-0.1mm}(\textit{i.e.}, \hspace{-0.1mm}outside \hspace{-0.1mm}$\Omega$) \hspace{-0.1mm}and \hspace{-0.1mm}interior \hspace{-0.1mm}(\textit{i.e.},~\hspace{-0.1mm}inside~\hspace{-0.1mm}$\Omega$)~\hspace{-0.1mm}points~\hspace{-0.1mm}to~\hspace{-0.1mm}$\partial\Omega$, which is provided by 
    the \textit{unsigned distance function} $\textup{dist}(\cdot,\partial\Omega)\colon \mathbb{R}^d\to [0,+\infty)$, for every $x\in \mathbb{R}^d$ defined by  
     \begin{align}\label{def:dist}
         \textup{dist}(x,\partial\Omega)\coloneqq \min_{y\in\partial\Omega}{\{\vert x-y\vert\}}=\vert x-\pi(x)\vert\,,
     \end{align}
     where the second equality sign exploits that $\vert x-x'\vert\hspace{-0.1em}=\hspace{-0.1em}\vert x-x''\vert$ for all $x',x''\hspace{-0.1em}\in\hspace{-0.1em} \pi_{\partial\Omega}(x)$~and~${x\hspace{-0.1em}\in \hspace{-0.1em}\partial\Omega}$.~By construction, \hspace{-0.1mm}the \hspace{-0.1mm}unsigned \hspace{-0.1mm}distance \hspace{-0.1mm}function \hspace{-0.1mm}\eqref{def:dist} 
   \hspace{-0.1mm}is \hspace{-0.1mm}Lipschitz~\hspace{-0.1mm}\mbox{continuous}~\hspace{-0.1mm}with~\hspace{-0.1mm}\mbox{constant}~\hspace{-0.1mm}$1$~\hspace{-0.1mm}and,~\hspace{-0.1mm}thus,
    by  Rademacher's theorem (\textit{cf}.\ \cite[Thm.\ 2.14]{AFP2000}),  a.e.\ 
    differentiable with $\vert \nabla \textup{dist}(\cdot,\partial\Omega)\vert\leq 1$~a.e.~in~$\mathbb{R}^d$. Beyond that, according to \cite[Cor.\ 3.4.5]{CannarsaSinestrari2004}, 
     it is precisely differentiable in $\mathbb{R}^d\setminus(\textrm{Med}(\partial\Omega)\cup\partial\Omega)$~with
     \begin{align}\label{eq:grad_dist}
         \nabla \textup{dist}(\cdot,\partial\Omega)=\begin{cases}
             n\circ\pi_{\partial\Omega}&\text{ in }\mathbb{R}^d\setminus (\textrm{Med}(\partial\Omega)\cup\overline{\Omega})\,,\\
             -n\circ\pi_{\partial\Omega}&\text{ in }\Omega\setminus \textrm{Med}(\partial\Omega)\,.
         \end{cases}
     \end{align} 

    The change of sign in \eqref{eq:grad_dist} is due to the fact that the unsigned distance function \eqref{def:dist}~does~not take into account whether points lie inside or outside $\Omega$. This additional information is included in the \textit{signed distance function} $\smash{\widehat{\mathrm{dist}}}(\cdot,\partial\Omega)\colon \mathbb{R}^d\to \mathbb{R}$, for every $x\in \mathbb{R}^d$ defined by
     \begin{align}\label{def:dist_hat}
         \smash{\widehat{\mathrm{dist}}}(x,\partial\Omega)\coloneqq\begin{cases}
             \textup{dist}(x,\partial\Omega)&\text{ if }x\in \mathbb{R}^d\setminus\Omega\,,\\
             -\textup{dist}(x,\partial\Omega)&\text{ else}\,.
         \end{cases} 
     \end{align}
     Inherited from the unsigned distance function \eqref{def:dist}, the signed distance function \eqref{def:dist_hat}
     is equally Lipschitz continuous with constant $1$ and, thus, 
    a.e.\ 
    differentiable with $\vert \nabla \smash{\widehat{\mathrm{dist}}}(\cdot,\partial\Omega)\vert\leq 1$~a.e.~in~$\mathbb{R}^d$. Since the signed distance function \eqref{def:dist_hat} takes into account whether points lie inside or outside $\Omega$, it is not only 
    differentiable in $\mathbb{R}^d\setminus(\textrm{Med}(\partial\Omega)\cup\partial\Omega)$, but --instead of \eqref{eq:grad_dist}--~additionally~satisfies
     \begin{align}\label{eq:grad_dist_hat}
         \nabla \smash{\widehat{\mathrm{dist}}}(\cdot,\partial\Omega)= 
             n\circ\pi_{\partial\Omega}\quad\text{ in }\mathbb{R}^d\setminus (\textrm{Med}(\partial\Omega)\cup\partial\Omega)\,.
     \end{align}  
     Thanks to \eqref{eq:grad_dist_hat}, if the insulated boundary $\Gamma_I$ is piece-wise flat,  close to the flat~\mbox{boundary~parts} $\Gamma_{I}^{\ell}$, $\ell=1,\ldots,L$, but away from their boundaries $\partial\Gamma_{I}^{\ell}$, $\ell=1,\ldots,L$, (\textit{cf}.\ Figure \ref{fig:limsup_case2}), %
     the signed distance function \eqref{def:dist_hat} is piece-wise affine and, thus, locally invariant under 
     mollification~across~$\Gamma_I$, which~is~the~striking ingredient in the proof of the $\limsup$-estimate~in~the~case~of~a~\mbox{piece-wise}~flat insulated boundary $\Gamma_I$ (\textit{cf}.\ Lemma~\ref{lem:limsup_case2}).\enlargethispage{3mm}\vspace{-1mm}
     
    \subsection{Function spaces}\vspace{-0.5mm}
    
    \hspace{5mm}Let $\omega \hspace{-0.15em}\subseteq \hspace{-0.15em}\smash{\mathbb{R}^d}$, $d\hspace{-0.15em}\in\hspace{-0.15em} \mathbb{N}$, be a Lebesgue measurable set with Lebesgue measure ${\vert \omega\vert\hspace{-0.15em}\coloneqq\hspace{-0.15em} \int_{\omega}{1\,\mathrm{d}x}\hspace{-0.15em}\in\hspace{-0.15em} [0,+\infty]}$.
    Then, for Lebesgue measurable functions or vector fields $v,w\colon \omega\to \mathbb{R}^{\ell}$, $\ell\in\{1,d\}$, respectively, we employ the inner product $(v,w)_{\omega}\coloneqq \int_{\omega}{v\odot w\,\mathrm{d}x}$,
	whenever the right-hand side is well-defined, where $\smash{\odot\colon \mathbb{R}^{\ell}\times \mathbb{R}^{\ell}\to \mathbb{R}}$ either denotes~scalar~multiplication or the Euclidean inner product. 
    
    For $p\in [1,+\infty]$, we employ standard notation for Lebesgue $\smash{L^p(\omega)}$ and Sobolev $\smash{H^{1,p}(\omega)}$~spaces, where $\omega$ shall be open for  Sobolev spaces. 
    The $L^p(\omega)$- and $H^{1,p}(\omega)$-norm, respectively,~is~defined~by\vspace{-0.5mm}
    \begin{align*}
        \|\cdot\|_{p,\omega}&\coloneqq\begin{cases}
             (\int_\omega{\vert \cdot\vert^p\,\mathrm{d}x})^{\smash{\frac{1}{p}}}&\text{ if }p\in [1,+\infty)\,,\\
    \textup{ess\,sup}_{x\in \omega}{\vert (\cdot)(x)\vert}&\text{ if }p=+\infty\,,
        \end{cases} \\
        \|\cdot\|_{1,p,\omega}&\coloneqq \|\cdot\|_{p,\omega}+\|\nabla\cdot\|_{p,\omega}\,.
    \end{align*}
    The completion of the linear space of smooth and compactly supported functions $C_c^{\infty}(\omega)$~in~$H^{1,p}(\omega)$ is denoted by $H^{1,p}_0(\omega)$.
    We abbreviate $\smash{H^1(\omega)}\coloneqq \smash{H^{1,2}(\omega)}$, $\smash{H^1_0(\omega)}\coloneqq \smash{H^{1,2}_0(\omega)}$, and $\|\cdot\|_{\omega}\coloneqq \|\cdot\|_{2,\omega}$.\linebreak
    Moreover, we employ the same notation in the case that $\omega$ is replaced by a (relatively) open boundary part $\gamma\subseteq \partial\Omega$, in which case the Lebesgue measure $\mathrm{d}x$ is replaced by the~surface~\mbox{measure}~$\mathrm{d}s$. 

    The assumption $\Gamma_I\neq\emptyset$ guarantees the validity of 
    Friedrich's inequality (\textit{cf}.\ \cite[Ex.\ II.5.13]{Galdi}), which states that there exists a constant $c_{\mathrm{F}}>0$ such that for every ${v\in H^1(\Omega)}$,~there~holds
        \begin{align}\label{lem:poin_cont}
            \|v\|_{\Omega}^2\leq \smash{c_{\mathrm{F}}\,\{\|\nabla v\|_{\Omega}^2+\|v\|_{\Gamma_I}^2\}}\,.\\[-6mm]\notag
        \end{align} 

    \subsection{Transversal vector fields}\vspace{-0.5mm}

   \hspace{5mm}The key idea in the generalization of the $\Gamma$-convergence analysis for the case $\Gamma_{I}\in C^{1,1}$~in~\cite{PietraNitschScalaTrombetti2021}~to bounded Lipschitz domains with piece-wise flat $\Gamma_{I}\in C^{0,1}$ is to relax the orthogonality~condition on the outward unit normal field $n\colon \Gamma_{I}\to \mathbb{S}^{d-1}$, preventing the latter to be regular (\textit{cf}.\ Figure~\ref{fig:transversality}\textit{(top)}). More precisely, we replace the outward unit normal field $n\colon \hspace{-0.1em}\Gamma_{I}\hspace{-0.1em}\to\hspace{-0.1em} \mathbb{S}^{d-1}$~by~a~\mbox{unit-length}~vector~field  $k\colon \Gamma_{I}\to \mathbb{S}^{d-1}$ with comparable  properties, but which is allowed to violate the orthogonality condition (on $n$) to a certain extent (\textit{i.e.}, depending on the Lipschitz regularity of $\Gamma_{I}$),~as~a~consequence, is more flexible and can be chosen to be arbitrarily smooth --even if only $\Gamma_{I}\in C^{0,1}$.


   A class of vector fields that precisely meets these requirements are transversal vector fields, for which we employ the following standard definition in this paper (see \cite{HMT07},~for~a~detailed~discussion):\enlargethispage{2.5mm}

	\begin{definition}\label{def:transversal}
		An open set $\Omega\subseteq \mathbb{R}^d$, $d\in  \mathbb{N}$,  of locally finite perimeter, with outward unit normal vector field $  n\colon\partial \Omega\to\mathbb{S}^{d-1}$, has a \emph{continuous (globally) transversal~vector~field} if there exists a  vector field $k \in (C^0(\partial \Omega))^d$ and a constant $\kappa>0$,~the~\emph{transversality constant}~of~$k $,~such~that\vspace{-0.5mm}
			\begin{align}
				k \cdot  n\ge \kappa \quad\text{ a.e.\ on } \partial \Omega\,.\label{eq:transversal} 
			\end{align} 
	\end{definition}

    \begin{remark}[interpretation of transversality]\label{rem:transversality} The condition \eqref{eq:transversal} can be seen as an \emph{`normal angle condition'} as it is equivalent to  
        \begin{align*}
            \smash{\sphericalangle(k ,  n)=\arccos(k \cdot  n)\leq \arccos(\kappa)\quad\text{ a.e.\ on }\partial\Omega\,,}
        \end{align*}
        and, thus,  expresses that 
        the continuous (globally) transversal vector field 
        $k \in (C^0(\partial \Omega))^d$~varies~from the outward unit normal vector field $  n\colon \partial\Omega\to \mathbb{S}^{d-1}$ up to the~maximal~angle~$\arccos(\kappa)$~(\textit{cf}.~\mbox{Figure}~\ref{fig:transversality}). 
    \end{remark}

    \begin{remark}[simple examples for transversal vector fields]\label{rem:examples}
         \begin{itemize}[noitemsep,topsep=2pt,leftmargin=!,labelwidth=\widthof{(ii)}]
            \item[(i)] \hypertarget{rem:examples.i}{} According to \cite[Thm.\ 2.19, (2.74), (2.75)]{HMT07}, if $\Omega\subseteq \mathbb{R}^d$, $d\in \mathbb{N}$,  is a non-empty, bounded open set of locally~finite~perimeter,  
            then, for every $\alpha\in [0,1]$, there holds $n\in (C^{0,\alpha}(\partial\Omega))^d$ if and only if $\Omega$ is a $C^{1,\alpha}$-domain, so that  if $\Omega$ is  a $C^{1,\alpha}$-domain for some $\alpha\in [0,1]$,  a continuous (globally) transversal vector field (with transversality constant $\kappa=1$) is given via $k\coloneqq n\in (C^{0,\alpha}(\partial\Omega))^d$~(\textit{cf}.~Figure~\ref{fig:placeholder});

            \item[(ii)] \hypertarget{rem:examples.ii}{} According to \cite[Cor.~4.21]{HMT07},
            if $\Omega\hspace{-0.05em}\subseteq\hspace{-0.05em} \mathbb{R}^d$, $d\hspace{-0.05em}\in \hspace{-0.05em}\mathbb{N}$, is star-shaped with respect~to~a~ball~${B_r^d(x_0)\hspace{-0.05em}\subseteq \hspace{-0.05em}\Omega}$, where $r>0$ and $x_0\in \Omega$,  
            a smooth (globally) transversal vector field of unit-length is given~via $k\coloneqq \tfrac{\mathrm{id}_{\smash{\mathbb{R}^d}}-x_0}{\vert \mathrm{id}_{\smash{\mathbb{R}^d}}-x_0\vert}\in(C^{\infty}(\partial\Omega))^d$ (\textit{cf}.\ Figure \ref{fig:transversality2D_3D}).\vspace{-1mm} 

            \begin{figure}[H]\vspace{-1.5mm}
                \centering
                \includegraphics[width=0.35\linewidth]{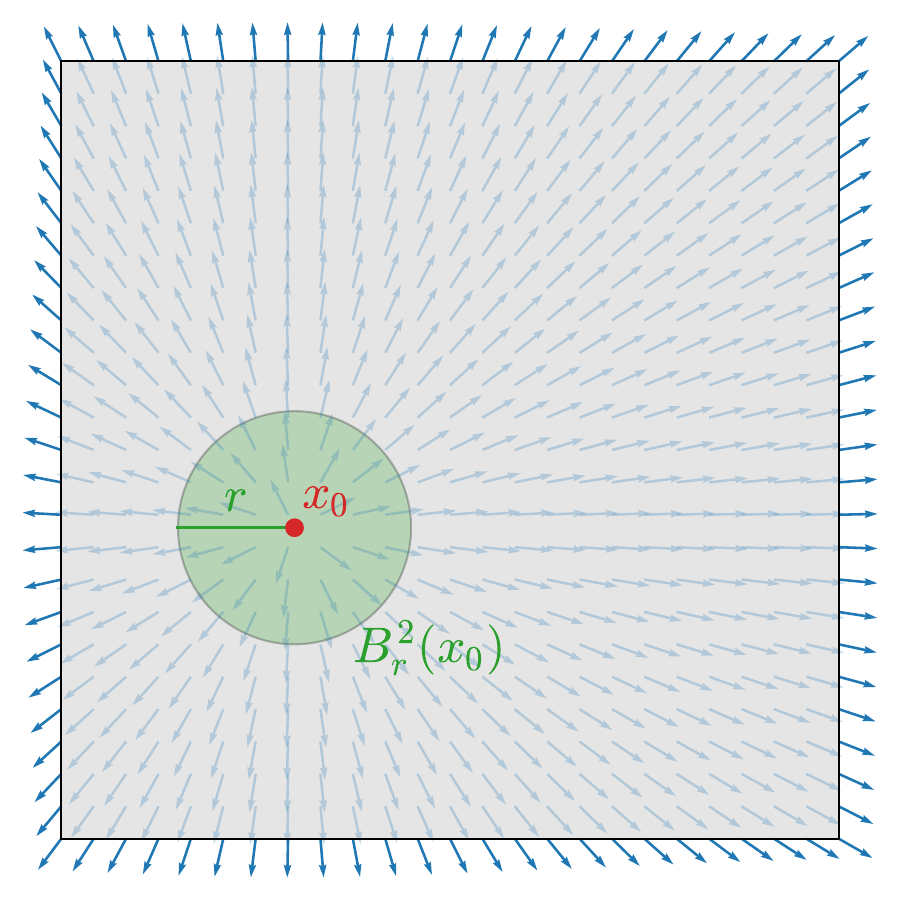}\qquad\includegraphics[width=0.375\linewidth]{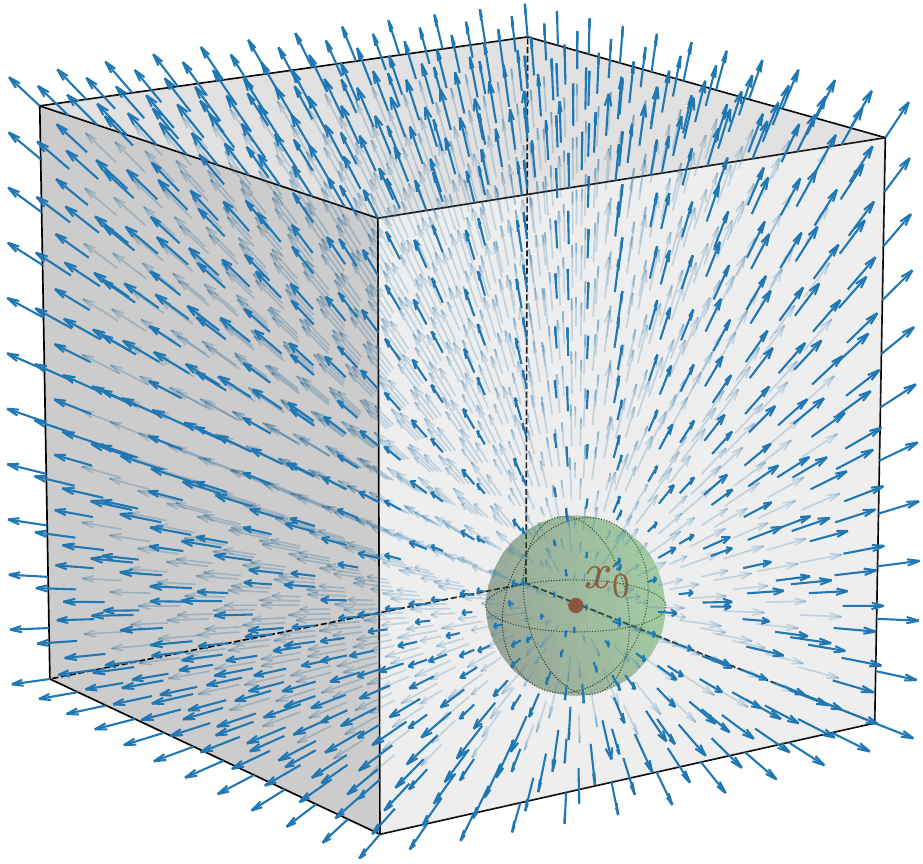}\vspace{-1mm}
                \caption{A domain $\Omega\subseteq \mathbb{R}^d$, $d\in \{2,3\}$, (\textcolor{gray}{gray}) star-shaped with respect to a ball $B_r^d(x_0)$ (\textcolor{shamrockgreen}{green}) and (globally) transversal vector field $k\coloneqq \tfrac{\mathrm{id}_{\smash{\mathbb{R}^d}}-x_0}{\vert \mathrm{id}_{\smash{\mathbb{R}^d}}-x_0\vert}\in(C^{\infty}(\partial\Omega))^d$ (\textcolor{denim}{blue})~\mbox{centred}~at~${x_0\in\Omega}$~(\textcolor{red}{red}).\vspace{-1.5mm}}
                \label{fig:transversality2D_3D}
            \end{figure}
        \end{itemize}
    \end{remark}

    The existence of a continuous (globally) transversal vector field is always ensured~in~this~paper.\vspace{-0.5mm}

    \begin{theorem}\label{thm:ex_transversal}
        Let $\Omega\subseteq \mathbb{R}^d$, $d\in \mathbb{N}$, be a non-empty, bounded Lipschitz domain. Then, there exists a vector field $k\in (C^\infty(\mathbb{R}^d))^d$ whose restriction to $\partial\Omega$ is (globally) transversal for $\Omega$.\vspace{-0.5mm}
    \end{theorem}

    \begin{proof}
        See \cite[Cor.\ 2.13]{HMT07}.
    \end{proof} 

    \begin{figure}[H]
        \centering
        
        \tikzset {_icitj8vmp/.code = {\pgfsetadditionalshadetransform{ \pgftransformshift{\pgfpoint{0 bp } { 0 bp }  }  \pgftransformrotate{-90 }  \pgftransformscale{2 }  }}}
        \pgfdeclarehorizontalshading{_6wypyjxno}{150bp}{rgb(0bp)=(0.89,0.89,0.89);
        	rgb(37.5bp)=(0.89,0.89,0.89);
        	rgb(37.5bp)=(0.86,0.86,0.86);
        	rgb(37.5bp)=(0.82,0.82,0.82);
        	rgb(62.5bp)=(1,1,1);
        	rgb(100bp)=(1,1,1)}
        \tikzset{every picture/.style={line width=0.75pt}} 
	
  
\tikzset{every picture/.style={line width=0.75pt}} 

\begin{tikzpicture}[x=1.34pt,y=1.34pt,yscale=-1,xscale=1]

\draw [color={white}  ,draw opacity=1 ] [dash pattern={on 0.75pt off 0.75pt on 0.75pt off 0.75pt}]  (317.83,77.11) -- (317.83,96.58) ;
\draw  [draw opacity=0][fill={denim}  ,fill opacity=0.2 ] (10.33,27.28) .. controls (10.53,27.28) and (10.73,27.28) .. (10.94,27.28) .. controls (43.25,27.52) and (85.64,90.42) .. (117.72,90.48) .. controls (149.8,90.54) and (191.65,27.34) .. (224.23,27.28) .. controls (255.74,27.22) and (295.74,83.62) .. (319.83,87.55) -- (319.83,139.58) -- (10.33,139.58) -- (10.33,27.28) -- cycle ;
\draw [color={denim}  ,draw opacity=1 ]   (117.72,123.03) -- (117.95,107.25) ;
\draw [shift={(118,104.25)}, rotate = 90.86] [fill={denim}  ,fill opacity=1 ][line width=0.08]  [draw opacity=0] (3.57,-1.72) -- (0,0) -- (3.57,1.72) -- cycle    ;
\draw [color={denim}  ,draw opacity=0.9 ]   (121.75,122.75) -- (119.86,106.98) ;
\draw [shift={(119.5,104)}, rotate = 83.16] [fill={denim}  ,fill opacity=0.9 ][line width=0.08]  [draw opacity=0] (3.57,-1.72) -- (0,0) -- (3.57,1.72) -- cycle    ;
\draw [color={denim}  ,draw opacity=0.8 ]   (125.3,121.78) -- (121.91,106.68) ;
\draw [shift={(121.25,103.75)}, rotate = 77.34] [fill={denim}  ,fill opacity=0.8 ][line width=0.08]  [draw opacity=0] (3.57,-1.72) -- (0,0) -- (3.57,1.72) -- cycle    ;
\draw [color={denim}  ,draw opacity=0.7 ]   (128.8,120.78) -- (123.73,106.09) ;
\draw [shift={(122.75,103.25)}, rotate = 70.96] [fill={denim}  ,fill opacity=0.7 ][line width=0.08]  [draw opacity=0] (3.57,-1.72) -- (0,0) -- (3.57,1.72) -- cycle    ;
\draw [color={denim}  ,draw opacity=0.6 ]   (132,119.75) -- (125.71,105.49) ;
\draw [shift={(124.5,102.75)}, rotate = 66.19] [fill={denim}  ,fill opacity=0.6 ][line width=0.08]  [draw opacity=0] (3.57,-1.72) -- (0,0) -- (3.57,1.72) -- cycle    ;
\draw [color={denim}  ,draw opacity=0.5 ]   (135.25,118.5) -- (128.11,104.43) ;
\draw [shift={(126.75,101.75)}, rotate = 63.09] [fill={denim}  ,fill opacity=0.5 ][line width=0.08]  [draw opacity=0] (3.57,-1.72) -- (0,0) -- (3.57,1.72) -- cycle    ;
\draw [color={denim}  ,draw opacity=0.4 ]   (137.75,116.75) -- (130.42,103.14) ;
\draw [shift={(129,100.5)}, rotate = 61.7] [fill={denim}  ,fill opacity=0.4 ][line width=0.08]  [draw opacity=0] (3.57,-1.72) -- (0,0) -- (3.57,1.72) -- cycle    ;
\draw [color={denim}  ,draw opacity=0.3 ]   (140.25,115) -- (132.74,101.85) ;
\draw [shift={(131.25,99.25)}, rotate = 60.26] [fill={denim}  ,fill opacity=0.3 ][line width=0.08]  [draw opacity=0] (3.57,-1.72) -- (0,0) -- (3.57,1.72) -- cycle    ;
\draw [color={denim}  ,draw opacity=0.2 ]   (142.75,113.5) -- (135.04,100.58) ;
\draw [shift={(133.5,98)}, rotate = 59.17] [fill={denim}  ,fill opacity=0.2 ][line width=0.08]  [draw opacity=0] (3.57,-1.72) -- (0,0) -- (3.57,1.72) -- cycle    ;
\draw [color={denim}  ,draw opacity=0.1 ]   (145.25,111.75) -- (137.36,99.28) ;
\draw [shift={(135.75,96.75)}, rotate = 57.65] [fill={denim}  ,fill opacity=0.1 ][line width=0.08]  [draw opacity=0] (3.57,-1.72) -- (0,0) -- (3.57,1.72) -- cycle    ;
\draw [color={denim}  ,draw opacity=0.9 ]   (113.91,122.75) -- (115.81,106.98) ;
\draw [shift={(116.17,104)}, rotate = 96.89] [fill={denim}  ,fill opacity=0.9 ][line width=0.08]  [draw opacity=0] (3.57,-1.72) -- (0,0) -- (3.57,1.72) -- cycle    ;
\draw [color={denim}  ,draw opacity=0.8 ]   (110.33,121.78) -- (113.75,106.68) ;
\draw [shift={(114.41,103.75)}, rotate = 102.74] [fill={denim}  ,fill opacity=0.8 ][line width=0.08]  [draw opacity=0] (3.57,-1.72) -- (0,0) -- (3.57,1.72) -- cycle    ;
\draw [color={denim}  ,draw opacity=0.7 ]   (106.81,120.78) -- (111.92,106.08) ;
\draw [shift={(112.9,103.25)}, rotate = 109.16] [fill={denim}  ,fill opacity=0.7 ][line width=0.08]  [draw opacity=0] (3.57,-1.72) -- (0,0) -- (3.57,1.72) -- cycle    ;
\draw [color={denim}  ,draw opacity=0.6 ]   (103.59,119.75) -- (109.92,105.49) ;
\draw [shift={(111.14,102.75)}, rotate = 113.95] [fill={denim}  ,fill opacity=0.6 ][line width=0.08]  [draw opacity=0] (3.57,-1.72) -- (0,0) -- (3.57,1.72) -- cycle    ;
\draw [color={denim}  ,draw opacity=0.5 ]   (100.32,118.5) -- (107.51,104.42) ;
\draw [shift={(108.87,101.75)}, rotate = 117.06] [fill={denim}  ,fill opacity=0.5 ][line width=0.08]  [draw opacity=0] (3.57,-1.72) -- (0,0) -- (3.57,1.72) -- cycle    ;
\draw [color={denim}  ,draw opacity=0.4 ]   (97.8,116.75) -- (105.18,103.14) ;
\draw [shift={(106.61,100.5)}, rotate = 118.46] [fill={denim}  ,fill opacity=0.4 ][line width=0.08]  [draw opacity=0] (3.57,-1.72) -- (0,0) -- (3.57,1.72) -- cycle    ;
\draw [color={denim}  ,draw opacity=0.3 ]   (95.28,115) -- (102.85,101.85) ;
\draw [shift={(104.34,99.25)}, rotate = 119.91] [fill={denim}  ,fill opacity=0.3 ][line width=0.08]  [draw opacity=0] (3.57,-1.72) -- (0,0) -- (3.57,1.72) -- cycle    ;
\draw [color={denim}  ,draw opacity=0.2 ]   (92.77,113.5) -- (100.53,100.57) ;
\draw [shift={(102.08,98)}, rotate = 121] [fill={denim}  ,fill opacity=0.2 ][line width=0.08]  [draw opacity=0] (3.57,-1.72) -- (0,0) -- (3.57,1.72) -- cycle    ;
\draw [color={denim}  ,draw opacity=0.1 ]   (90.25,111.75) -- (98.2,99.28) ;
\draw [shift={(99.81,96.75)}, rotate = 122.52] [fill={denim}  ,fill opacity=0.1 ][line width=0.08]  [draw opacity=0] (3.57,-1.72) -- (0,0) -- (3.57,1.72) -- cycle    ;
\draw [color={denim}  ,draw opacity=1 ]   (224.23,59.84) -- (224.47,44.06) ;
\draw [shift={(224.52,41.06)}, rotate = 90.86] [fill={denim}  ,fill opacity=1 ][line width=0.08]  [draw opacity=0] (3.57,-1.72) -- (0,0) -- (3.57,1.72) -- cycle    ;
\draw [color={denim}  ,draw opacity=0.9 ]   (226.41,60) -- (227.74,44.49) ;
\draw [shift={(228,41.5)}, rotate = 94.92] [fill={denim}  ,fill opacity=0.9 ][line width=0.08]  [draw opacity=0] (3.57,-1.72) -- (0,0) -- (3.57,1.72) -- cycle    ;
\draw [color={denim}  ,draw opacity=0.8 ]   (228.33,60.03) -- (231.19,44.95) ;
\draw [shift={(231.75,42)}, rotate = 100.73] [fill={denim}  ,fill opacity=0.8 ][line width=0.08]  [draw opacity=0] (3.57,-1.72) -- (0,0) -- (3.57,1.72) -- cycle    ;
\draw [color={denim}  ,draw opacity=0.7 ]   (230.31,60.53) -- (234.65,45.88) ;
\draw [shift={(235.5,43)}, rotate = 106.49] [fill={denim}  ,fill opacity=0.7 ][line width=0.08]  [draw opacity=0] (3.57,-1.72) -- (0,0) -- (3.57,1.72) -- cycle    ;
\draw [color={denim}  ,draw opacity=0.6 ]   (232.31,61.28) -- (237.48,47.07) ;
\draw [shift={(238.5,44.25)}, rotate = 109.97] [fill={denim}  ,fill opacity=0.6 ][line width=0.08]  [draw opacity=0] (3.57,-1.72) -- (0,0) -- (3.57,1.72) -- cycle    ;
\draw [color={denim}  ,draw opacity=0.5 ]   (234.81,62.03) -- (240.37,48.28) ;
\draw [shift={(241.5,45.5)}, rotate = 112.03] [fill={denim}  ,fill opacity=0.5 ][line width=0.08]  [draw opacity=0] (3.57,-1.72) -- (0,0) -- (3.57,1.72) -- cycle    ;
\draw [color={denim}  ,draw opacity=0.4 ]   (237.06,63.03) -- (243.04,49.49) ;
\draw [shift={(244.25,46.75)}, rotate = 113.83] [fill={denim}  ,fill opacity=0.4 ][line width=0.08]  [draw opacity=0] (3.57,-1.72) -- (0,0) -- (3.57,1.72) -- cycle    ;
\draw [color={denim}  ,draw opacity=0.3 ]   (239.56,63.78) -- (245.72,50.71) ;
\draw [shift={(247,48)}, rotate = 115.24] [fill={denim}  ,fill opacity=0.3 ][line width=0.08]  [draw opacity=0] (3.57,-1.72) -- (0,0) -- (3.57,1.72) -- cycle    ;
\draw [color={denim}  ,draw opacity=0.2 ]   (241.56,64.78) -- (248.12,52.16) ;
\draw [shift={(249.5,49.5)}, rotate = 117.45] [fill={denim}  ,fill opacity=0.2 ][line width=0.08]  [draw opacity=0] (3.57,-1.72) -- (0,0) -- (3.57,1.72) -- cycle    ;
\draw [color={denim}  ,draw opacity=0.1 ]   (243.81,65.78) -- (250.37,53.16) ;
\draw [shift={(251.75,50.5)}, rotate = 117.45] [fill={denim}  ,fill opacity=0.1 ][line width=0.08]  [draw opacity=0] (3.57,-1.72) -- (0,0) -- (3.57,1.72) -- cycle    ;
\draw [color={denim}  ,draw opacity=0.9 ]   (226.41,60.25) -- (227.74,44.74) ;
\draw [shift={(228,41.75)}, rotate = 94.92] [fill={denim}  ,fill opacity=0.9 ][line width=0.08]  [draw opacity=0] (3.57,-1.72) -- (0,0) -- (3.57,1.72) -- cycle    ;
\draw [color={denim}  ,draw opacity=0.8 ]   (228.33,60.28) -- (231.19,45.2) ;
\draw [shift={(231.75,42.25)}, rotate = 100.73] [fill={denim}  ,fill opacity=0.8 ][line width=0.08]  [draw opacity=0] (3.57,-1.72) -- (0,0) -- (3.57,1.72) -- cycle    ;
\draw [color={denim}  ,draw opacity=0.7 ]   (230.31,60.78) -- (234.65,46.13) ;
\draw [shift={(235.5,43.25)}, rotate = 106.49] [fill={denim}  ,fill opacity=0.7 ][line width=0.08]  [draw opacity=0] (3.57,-1.72) -- (0,0) -- (3.57,1.72) -- cycle    ;
\draw [color={denim}  ,draw opacity=0.6 ]   (232.31,61.53) -- (237.48,47.32) ;
\draw [shift={(238.5,44.5)}, rotate = 109.97] [fill={denim}  ,fill opacity=0.6 ][line width=0.08]  [draw opacity=0] (3.57,-1.72) -- (0,0) -- (3.57,1.72) -- cycle    ;
\draw [color={denim}  ,draw opacity=0.5 ]   (234.81,62.28) -- (240.37,48.53) ;
\draw [shift={(241.5,45.75)}, rotate = 112.03] [fill={denim}  ,fill opacity=0.5 ][line width=0.08]  [draw opacity=0] (3.57,-1.72) -- (0,0) -- (3.57,1.72) -- cycle    ;
\draw [color={denim}  ,draw opacity=0.4 ]   (237.06,63.28) -- (243.04,49.74) ;
\draw [shift={(244.25,47)}, rotate = 113.83] [fill={denim}  ,fill opacity=0.4 ][line width=0.08]  [draw opacity=0] (3.57,-1.72) -- (0,0) -- (3.57,1.72) -- cycle    ;
\draw [color={denim}  ,draw opacity=0.3 ]   (239.56,64.03) -- (245.72,50.96) ;
\draw [shift={(247,48.25)}, rotate = 115.24] [fill={denim}  ,fill opacity=0.3 ][line width=0.08]  [draw opacity=0] (3.57,-1.72) -- (0,0) -- (3.57,1.72) -- cycle    ;
\draw [color={denim}  ,draw opacity=0.2 ]   (241.56,65.03) -- (248.12,52.41) ;
\draw [shift={(249.5,49.75)}, rotate = 117.45] [fill={denim}  ,fill opacity=0.2 ][line width=0.08]  [draw opacity=0] (3.57,-1.72) -- (0,0) -- (3.57,1.72) -- cycle    ;
\draw [color={denim}  ,draw opacity=0.1 ]   (243.81,66.03) -- (250.37,53.41) ;
\draw [shift={(251.75,50.75)}, rotate = 117.45] [fill={denim}  ,fill opacity=0.1 ][line width=0.08]  [draw opacity=0] (3.57,-1.72) -- (0,0) -- (3.57,1.72) -- cycle    ;
\draw [color={denim}  ,draw opacity=0.9 ]   (222.33,60) -- (221.01,44.49) ;
\draw [shift={(220.75,41.5)}, rotate = 85.13] [fill={denim}  ,fill opacity=0.9 ][line width=0.08]  [draw opacity=0] (3.57,-1.72) -- (0,0) -- (3.57,1.72) -- cycle    ;
\draw [color={denim}  ,draw opacity=0.8 ]   (220.42,60.03) -- (217.6,44.95) ;
\draw [shift={(217.04,42)}, rotate = 79.38] [fill={denim}  ,fill opacity=0.8 ][line width=0.08]  [draw opacity=0] (3.57,-1.72) -- (0,0) -- (3.57,1.72) -- cycle    ;
\draw [color={denim}  ,draw opacity=0.7 ]   (218.47,60.53) -- (214.18,45.88) ;
\draw [shift={(213.33,43)}, rotate = 73.67] [fill={denim}  ,fill opacity=0.7 ][line width=0.08]  [draw opacity=0] (3.57,-1.72) -- (0,0) -- (3.57,1.72) -- cycle    ;
\draw [color={denim}  ,draw opacity=0.6 ]   (216.49,61.28) -- (211.38,47.07) ;
\draw [shift={(210.36,44.25)}, rotate = 70.22] [fill={denim}  ,fill opacity=0.6 ][line width=0.08]  [draw opacity=0] (3.57,-1.72) -- (0,0) -- (3.57,1.72) -- cycle    ;
\draw [color={denim}  ,draw opacity=0.5 ]   (214.02,62.03) -- (208.51,48.28) ;
\draw [shift={(207.39,45.5)}, rotate = 68.17] [fill={denim}  ,fill opacity=0.5 ][line width=0.08]  [draw opacity=0] (3.57,-1.72) -- (0,0) -- (3.57,1.72) -- cycle    ;
\draw [color={denim}  ,draw opacity=0.4 ]   (211.79,63.03) -- (205.87,49.5) ;
\draw [shift={(204.67,46.75)}, rotate = 66.39] [fill={denim}  ,fill opacity=0.4 ][line width=0.08]  [draw opacity=0] (3.57,-1.72) -- (0,0) -- (3.57,1.72) -- cycle    ;
\draw [color={denim}  ,draw opacity=0.3 ]   (209.31,63.78) -- (203.22,50.72) ;
\draw [shift={(201.95,48)}, rotate = 64.99] [fill={denim}  ,fill opacity=0.3 ][line width=0.08]  [draw opacity=0] (3.57,-1.72) -- (0,0) -- (3.57,1.72) -- cycle    ;
\draw [color={denim}  ,draw opacity=0.2 ]   (207.33,64.78) -- (200.85,52.17) ;
\draw [shift={(199.48,49.5)}, rotate = 62.79] [fill={denim}  ,fill opacity=0.2 ][line width=0.08]  [draw opacity=0] (3.57,-1.72) -- (0,0) -- (3.57,1.72) -- cycle    ;
\draw [color={denim}  ,draw opacity=0.1 ]   (204.86,65.78) -- (198.37,53.17) ;
\draw [shift={(197,50.5)}, rotate = 62.79] [fill={denim}  ,fill opacity=0.1 ][line width=0.08]  [draw opacity=0] (3.57,-1.72) -- (0,0) -- (3.57,1.72) -- cycle    ;
\draw [color={denim}  ,draw opacity=0.9 ]   (222.33,60.25) -- (221.01,44.74) ;
\draw [shift={(220.75,41.75)}, rotate = 85.13] [fill={denim}  ,fill opacity=0.9 ][line width=0.08]  [draw opacity=0] (3.57,-1.72) -- (0,0) -- (3.57,1.72) -- cycle    ;
\draw [color={denim}  ,draw opacity=0.8 ]   (220.42,60.28) -- (217.6,45.2) ;
\draw [shift={(217.04,42.25)}, rotate = 79.38] [fill={denim}  ,fill opacity=0.8 ][line width=0.08]  [draw opacity=0] (3.57,-1.72) -- (0,0) -- (3.57,1.72) -- cycle    ;
\draw [color={denim}  ,draw opacity=0.7 ]   (218.47,60.78) -- (214.18,46.13) ;
\draw [shift={(213.33,43.25)}, rotate = 73.67] [fill={denim}  ,fill opacity=0.7 ][line width=0.08]  [draw opacity=0] (3.57,-1.72) -- (0,0) -- (3.57,1.72) -- cycle    ;
\draw [color={denim}  ,draw opacity=0.6 ]   (216.49,61.53) -- (211.38,47.32) ;
\draw [shift={(210.36,44.5)}, rotate = 70.22] [fill={denim}  ,fill opacity=0.6 ][line width=0.08]  [draw opacity=0] (3.57,-1.72) -- (0,0) -- (3.57,1.72) -- cycle    ;
\draw [color={denim}  ,draw opacity=0.5 ]   (214.02,62.28) -- (208.51,48.53) ;
\draw [shift={(207.39,45.75)}, rotate = 68.17] [fill={denim}  ,fill opacity=0.5 ][line width=0.08]  [draw opacity=0] (3.57,-1.72) -- (0,0) -- (3.57,1.72) -- cycle    ;
\draw [color={denim}  ,draw opacity=0.4 ]   (211.79,63.28) -- (205.87,49.75) ;
\draw [shift={(204.67,47)}, rotate = 66.39] [fill={denim}  ,fill opacity=0.4 ][line width=0.08]  [draw opacity=0] (3.57,-1.72) -- (0,0) -- (3.57,1.72) -- cycle    ;
\draw [color={denim}  ,draw opacity=0.3 ]   (209.31,64.03) -- (203.22,50.97) ;
\draw [shift={(201.95,48.25)}, rotate = 64.99] [fill={denim}  ,fill opacity=0.3 ][line width=0.08]  [draw opacity=0] (3.57,-1.72) -- (0,0) -- (3.57,1.72) -- cycle    ;
\draw [color={denim}  ,draw opacity=0.2 ]   (207.33,65.03) -- (200.85,52.42) ;
\draw [shift={(199.48,49.75)}, rotate = 62.79] [fill={denim}  ,fill opacity=0.2 ][line width=0.08]  [draw opacity=0] (3.57,-1.72) -- (0,0) -- (3.57,1.72) -- cycle    ;
\draw [color={denim}  ,draw opacity=0.1 ]   (204.86,66.03) -- (198.37,53.42) ;
\draw [shift={(197,50.75)}, rotate = 62.79] [fill={denim}  ,fill opacity=0.1 ][line width=0.08]  [draw opacity=0] (3.57,-1.72) -- (0,0) -- (3.57,1.72) -- cycle    ;
\draw  [draw opacity=0][shading=_6wypyjxno,_icitj8vmp] (10.33,139.58) -- (319.83,139.58) -- (319.83,120.11) .. controls (295.74,116.18) and (255.74,59.78) .. (224.23,59.84) .. controls (191.65,59.9) and (149.8,123.09) .. (117.72,123.03) .. controls (85.64,122.97) and (43.25,60.08) .. (10.94,59.84) .. controls (10.73,59.84) and (10.53,59.84) .. (10.33,59.83) -- (10.33,139.58) -- cycle ;

\draw [color={shamrockgreen}  ,draw opacity=1 ]   (117.72,123.03) .. controls (149.87,123.19) and (192.08,59.67) .. (224.23,59.84) ;
\draw [color={shamrockgreen}  ,draw opacity=1 ]   (224.23,59.84) .. controls (255.75,60) and (294.42,115) .. (319.83,120.11) ;
\draw [color={shamrockgreen}  ,draw opacity=1 ]   (10.33,59.83) .. controls (42.48,59.99) and (86.08,123) .. (117.72,123.03) ;
\draw [color={white}  ,draw opacity=1 ] [dash pattern={on 0.75pt off 0.75pt on 0.75pt off 0.75pt}]  (10.33,59.83) -- (10.33,139.58) ;
\draw [color={white}  ,draw opacity=1 ] [dash pattern={on 0.75pt off 0.75pt on 0.75pt off 0.75pt}]  (319.83,120.11) -- (319.83,139.58) ;
\draw [color={white}  ,draw opacity=1 ]   (10.33,139.58) -- (319.83,139.58) ;

\draw (220,99.4) node [anchor=north west][inner sep=0.75pt]   [font=\Large] {$\Omega $};
\draw (41.08,88.07) node [anchor=north west][inner sep=0.75pt]  [color={shamrockgreen}  ,opacity=1 ] [font=\large] {$\Gamma _{I}$};
\draw (17.5,42.5) node [anchor=north west][inner sep=0.75pt]  [color={denim}  ,opacity=1 ] [font=\large] {$\Sigma _{I}^{\varepsilon }$};
\draw (115,97.5) node [anchor=north west][inner sep=0.75pt]  [color={denim}  ,opacity=1 ]  {$n$};
\draw (221.5,35) node [anchor=north west][inner sep=0.75pt]  [color={denim}  ,opacity=1 ]  {$n$};
\draw  [fill={white}  ,fill opacity=1 ]  (138.5,24.5) -- (173,24.5) -- (173,35) -- (138.5,35) -- cycle  ;
\draw (140,25) node [anchor=north west][inner sep=0.75pt]  [font=\normalsize]  {$\Gamma _{I} \in C^{1,1}$};

\end{tikzpicture}\vspace{-1mm}

        \caption{An insulating layer in the case of a $C^{1,1}$-regular insulated boundary $\Gamma_I$~is~depicted. Gaps and self-intersections in $\Sigma_{I}^{\varepsilon}\coloneqq \{s+tn(s)\mid s\in  \Gamma_I\,,\;t\in (0,\varepsilon \mathtt{d}(s)]\}$ are precluded due to the Lipschitz regularity of the outward unit normal vector field $n\colon \Gamma_I\to \mathbb{S}^{d-1}$ (\textit{cf}.\ Remark \ref{rem:examples}(\protect\hyperlink{rem:examples.i}{i})).}
        \label{fig:placeholder}
    \end{figure}
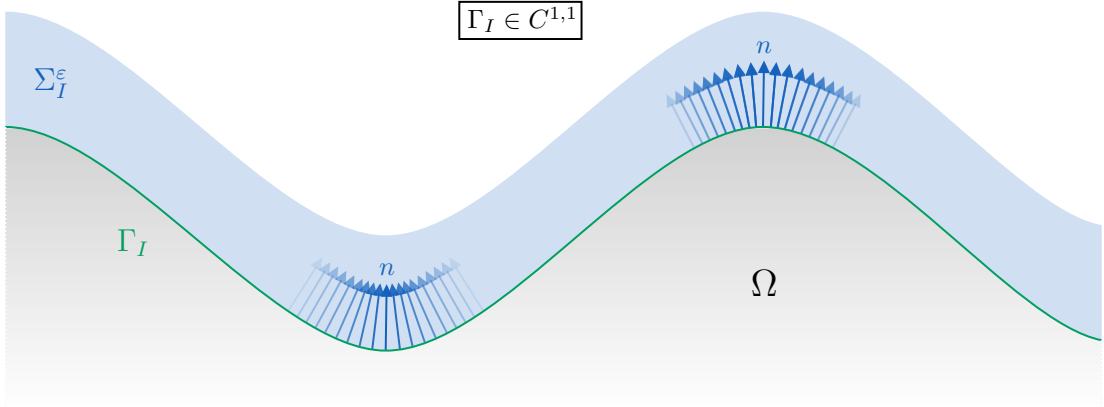\vspace{-5mm}
    
    \begin{figure}[H]
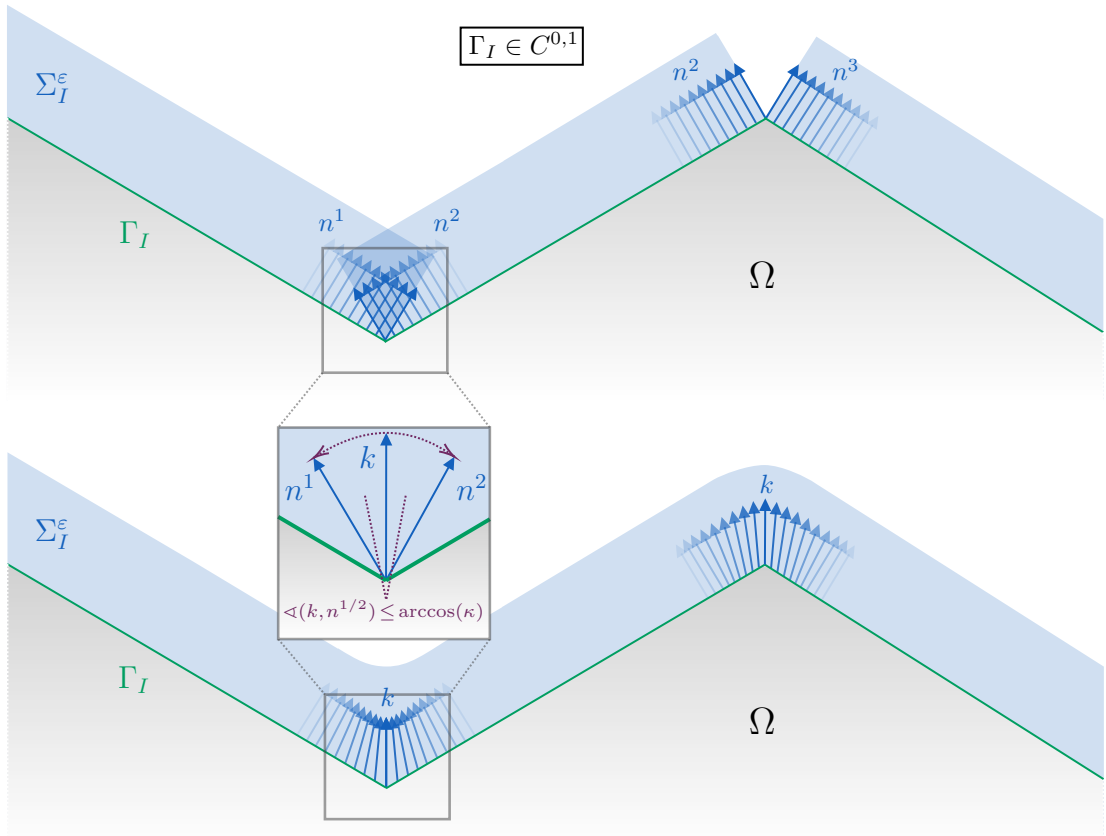

        \centering
        
        \tikzset{every picture/.style={line width=0.75pt}}

         \tikzset {_cbfb50p4s/.code = {\pgfsetadditionalshadetransform{ \pgftransformshift{\pgfpoint{0 bp } { 0 bp }  }  \pgftransformrotate{-90 }  \pgftransformscale{2 }  }}}
        \pgfdeclarehorizontalshading{_00xpoda8h}{150bp}{rgb(0bp)=(0.89,0.89,0.89);
        	rgb(37.5bp)=(0.89,0.89,0.89);
        	rgb(37.5bp)=(0.86,0.86,0.86);
        	rgb(37.5bp)=(0.82,0.82,0.82);
        	rgb(62.5bp)=(1,1,1);
        	rgb(100bp)=(1,1,1)}
        \tikzset{every picture/.style={line width=0.75pt}}



 \tikzset {_x35ti359s/.code = {\pgfsetadditionalshadetransform{ \pgftransformshift{\pgfpoint{0 bp } { 0 bp }  }  \pgftransformrotate{-90 }  \pgftransformscale{2 }  }}}
\pgfdeclarehorizontalshading{_k6x3mygon}{150bp}{rgb(0bp)=(0.89,0.89,0.89);
	rgb(37.5bp)=(0.89,0.89,0.89);
	rgb(37.5bp)=(0.86,0.86,0.86);
	rgb(37.5bp)=(0.82,0.82,0.82);
	rgb(62.5bp)=(1,1,1);
	rgb(100bp)=(1,1,1)}
  


 \tikzset {_4hujh30ho/.code = {\pgfsetadditionalshadetransform{ \pgftransformshift{\pgfpoint{0 bp } { 0 bp }  }  \pgftransformrotate{-90 }  \pgftransformscale{2 }  }}}
\pgfdeclarehorizontalshading{_68sxcjr5p}{150bp}{rgb(0bp)=(0.89,0.89,0.89);
	rgb(37.5bp)=(0.89,0.89,0.89);
	rgb(37.5bp)=(0.86,0.86,0.86);
	rgb(37.5bp)=(0.82,0.82,0.82);
	rgb(62.5bp)=(1,1,1);
	rgb(100bp)=(1,1,1)}
\tikzset{every picture/.style={line width=0.75pt}}


\vspace{-1mm}

        \caption{Two insulating layers   in the case of a piece-wise flat insulated boundary~$\Gamma_I$~are~\mbox{depicted}: 
        \textit{top:} discontinuities of $n\colon \Gamma_{I}\to  \mathbb{S}^{d-1}$ lead to gaps (\textit{i.e.}, no insulating material is applied)~or~self-intersections (\textit{i.e.}, insulating material is applied twice) in  $\widetilde{\Sigma}_{I}^{\varepsilon} \coloneqq \{s+t n(s)\mid s\in \Gamma_I\,,\;t\in (0,\varepsilon \mathtt{d}(s)]\}$; \textit{bottom:} gaps and self-intersections in $\Sigma_{I}^{\varepsilon}\coloneqq \{s+tk(s)\mid s\in  \Gamma_I\,,\;t\in (0,\varepsilon \mathtt{d}(s)]\}$ are precluded by replacing $n\colon\hspace{-0.1em} \Gamma_{I}\hspace{-0.1em}\to\hspace{-0.1em} \mathbb{S}^{d-1}$ by a unit-length continuous (globally) transversal~vector~field~${k\colon \hspace{-0.1em}\Gamma_{I}\hspace{-0.1em}\to\hspace{-0.1em}\mathbb{S}^{d-1}}$, which varies to $n\colon \Gamma_{I}\to \mathbb{S}^{d-1}$~up~to~a~maximal~angle~of~$\arccos(\kappa)$ (\textit{cf}.\ \eqref{eq:transversal}).}
        \label{fig:transversality}
    \end{figure}\newpage

      \section{Model reduction for the thickness of the insulating layer}\label{sec:modelling}

    \hspace{5mm}Let $k\in \smash{(C^0(\Gamma_I))^d}$ be a continuous (globally) transversal vector field of $\Omega$ with transversality constant $\kappa\hspace{-0.175em}\in\hspace{-0.175em} (0,1]$, the existence of which is ensured by Theorem \ref{thm:ex_transversal} and which,~in~the~case~${\Gamma_I\hspace{-0.175em}\in \hspace{-0.175em} C^{1,1}}$, is always fixed as $k\hspace{-0.175em}=\hspace{-0.175em}n\hspace{-0.175em}\in\hspace{-0.175em} \smash{(C^0(\Gamma_I))^d}$ (\textit{cf}.\ Remark \ref{rem:examples}(\hyperlink{rem:examples.i}{i})) (so that the~\mbox{transversality}~\mbox{constant}~is~${\kappa\hspace{-0.15em}=\hspace{-0.15em}1}$).
    Denote by $\mathtt{d}\hspace{-0.05em}\in\hspace{-0.05em} L^\infty(\Gamma_I)$ the (to be determined) non-negative  distribution function~(in~\mbox{direction}~of~$k$).\medskip
    
    Then, for a fixed, but arbitrarily small parameter
     $\varepsilon\hspace{-0.1em}\in\hspace{-0.1em} (0,\varepsilon_0)$, we define the \textit{\mbox{insulating}~layer} (in direction of $k$ and of local \textit{`thickness'} $\varepsilon \mathtt{d}$) $\Sigma_{I}^{\varepsilon}\subseteq \mathbb{R}^d$, the \textit{interacting insulation boundary} $\Gamma_{I}^{\varepsilon}\subseteq \partial\Sigma_{I}^{\varepsilon}$,  and the \textit{insulated conducting body} $\Omega_{I}^{\varepsilon}\subseteq \mathbb{R}^d$, respectively, via
     \begin{subequations}\label{def:some_notation}
     \begin{align}\label{def:some_notation.1}
        \Sigma^{\varepsilon}_{I}&\coloneqq\Sigma^{\varepsilon}_{I}(\mathtt{d})\coloneqq \big\{s+tk(s)\mid s\in \Gamma_I\,,\;t\in [0,\varepsilon \mathtt{d}(s))\big\}\,,\\[-0.25mm]\label{def:some_notation.2}
        \Gamma^{\varepsilon}_{I}&\coloneqq \Gamma^{\varepsilon}_{I}(\mathtt{d})\hspace{0.5mm}\coloneqq \big\{s+\varepsilon \mathtt{d}(s)k(s)\mid s\in \Gamma_{I}\big\}\,,\\\label{def:some_notation.3}
        \Omega^{\varepsilon}_{I}&\coloneqq \Omega^{\varepsilon}_{I}(\mathtt{d})\coloneqq \overline{\Omega}\cup\Sigma^{\varepsilon}_{I}\,.
    \end{align}
    \end{subequations} 

    Furthermore, let $f\in L^2(\Omega)$ be a given \emph{heat source density} (located in the thermally conducting body $\Omega$), $g\in H^{-\smash{\frac{1}{2}}}(\Gamma_N)$ a given \emph{heat flux} (across the Neumann boundary $\Gamma_N$), 
$u_D\in H^{\smash{\frac{1}{2}}}(\Gamma_D)$ a given \emph{temperature distribution} (at the Dirichlet boundary $\Gamma_D$), $u_{\infty}\in H^1(\mathbb{R}^d\setminus \overline{\Omega})$ a given \emph{ambient temperature} (of the surrounding medium in $\mathbb{R}^d\setminus \overline{\Omega}$), $\lambda>0$ the material-specific \emph{thermal conductivity} of the conducting body, and $\beta>0$ a given system-specific \emph{heat transfer coefficient}.\medskip

    Then, we consider the \emph{heat loss} functional $\smash{E}_\varepsilon^{\mathtt{d}}\colon 
    H^1(\Omega^{\varepsilon}_{I})\to \mathbb{R}\cup\{+\infty\}$, for every $v_\varepsilon\in H^1(\Omega^{\varepsilon}_{I})$ defined by 
    \begin{align}\label{eq:Evarh}
\smash{E}_\varepsilon^{\mathtt{d}}(v_\varepsilon)\coloneqq 
\left\{\begin{aligned}
    &\tfrac{\lambda}{2}\|\nabla v_\varepsilon\|_{\Omega}^2+\tfrac{\varepsilon}{2}\|\nabla v_\varepsilon\|_{\smash{\Sigma^{\varepsilon}_{I}}}^2+\tfrac{\beta}{2}\|v_\varepsilon-u_{\infty}\|_{\Gamma_{I}^{\varepsilon}}^2\\&-(f,v_\varepsilon)_{\Omega}-\langle g,v_{\varepsilon}\rangle_{\smash{H^{\smash{\frac{1}{2}}}(\Gamma_N)}}+I_{\{u_D\}}^{\Gamma_D}(v_{\varepsilon})\,,
\end{aligned}\right.
    \end{align} 
    where the indicator functional $I_{\{u_D\}}^{\Gamma_D}\colon H^{\frac{1}{2}}(\partial\Omega)\to \mathbb{R}\cup\{+\infty\}$, for every $\widehat{v}\in H^{\frac{1}{2}}(\partial\Omega)$, is defined~by
    \begin{align*}
        I_{\{u_D\}}^{\Gamma_D}(\widehat{v})\coloneqq\begin{cases}
            0&\text{ if }\widehat{v}=u_D\text{ a.e.\ on }\Gamma_D\,,\\
            +\infty&\text{ else}\,.
        \end{cases}
    \end{align*}
    
    The fixed, but arbitrarily small parameter $\varepsilon\in (0,\varepsilon_0)$ in  \eqref{def:some_notation} and  \eqref{eq:Evarh} plays two~\mbox{different}~roles: 
    \begin{itemize}[noitemsep,topsep=2pt,leftmargin=!,labelwidth=\widthof{(a)},font=\itshape]
        \item[(a)] In the definition of the insulating layer $\Sigma_{I}^{\varepsilon}$ (\textit{cf}.\ \eqref{def:some_notation}) together with the (to be determined) distribution function $\mathtt{d}\in L^\infty(\Gamma_I)$, it influences the local \textit{`thickness'} $\varepsilon\mathtt{d}$ of the~insulating~layer~$\Sigma_{I}^{\varepsilon}$;
        \item[(b)] In the heat loss functional \eqref{eq:Evarh}, it represents the thermal conductivity of the insulating material, which~--in this idealised situation-- is assumed to be arbitrarily small (\textit{i.e.}, ${\varepsilon\ll 1}$).\medskip
    \end{itemize}
    
    Since the heat loss functional \eqref{eq:Evarh} is proper, strictly convex, weakly coercive, and lower~semi-continuous, for given parameter $\varepsilon\hspace{-0.1em}\in\hspace{-0.1em} (0,\varepsilon_0)$ and distribution function $\mathtt{d}\hspace{-0.1em}\in \hspace{-0.1em}L^\infty(\Gamma_I)$,~the~\mbox{direct}~method in the calculus of variations 
    yields the existence of a unique temperature distribution~$u_\varepsilon^{\mathtt{d}}\in H^1(\Omega^{\varepsilon}_{I})$ minimizing \eqref{eq:Evarh}, 
    which formally satisfies the  Euler--Lagrange~equations~(\textit{cf}.~Figure~\ref{fig:domain}(\textit{left})) 
\begin{subequations}\label{eq:ELE_Eepsh}
        \begin{alignat}{2}
            -\lambda \Delta u_\varepsilon^{\mathtt{d}}&=f&&\quad \text{ a.e.\ in }\Omega\,,\label{eq:ELE_Eepsh.a}\\
             u_\varepsilon^{\mathtt{d}}&=u_D&&\quad \text{ a.e.\ on }\Gamma_D\,,\label{eq:ELE_Eepsh.b}\\
           \lambda \nabla u_\varepsilon^{\mathtt{d}}\cdot n&=g&&\quad \text{ a.e.\ on }\Gamma_N\,,\label{eq:ELE_Eepsh.c}\\
             -\varepsilon \Delta u_\varepsilon^{\mathtt{d}}&=0&&\quad \text{ a.e.\ in }\Sigma^{\varepsilon}_{I}\,,\label{eq:ELE_Eepsh.d}\\
            \varepsilon \nabla u_\varepsilon^{\mathtt{d}}\cdot n_\varepsilon^{\mathtt{d}}+\beta\{u_\varepsilon^{\mathtt{d}}-u_{\infty}\}  &= 0&&\quad \text{ a.e.\ on }\Gamma^{\varepsilon}_{I}\,,\label{eq:ELE_Eepsh.e}\\
           \nabla u_{\varepsilon}^{\mathtt{d}}\cdot n_{\varepsilon}^{\mathtt{d}}&=0&&\quad \text{ a.e.\ on } \partial\Sigma_{I}^{\varepsilon}\setminus(\Gamma_I\cup \Gamma^{\varepsilon}_{I})\,,\label{eq:ELE_Eepsh.f}\\ 
           \lambda \nabla (u^{\mathtt{d}}_\varepsilon|_{\Omega})\cdot n&=-\varepsilon \nabla (u^{\mathtt{d}}_\varepsilon|_{\Sigma_{I}^{\varepsilon}})\cdot  n_{\varepsilon}^{\mathtt{d}}&&\quad \text{ a.e.\ on } \Gamma_I\,,\label{eq:ELE_Eepsh.g}
        \end{alignat}
    \end{subequations}
    where $n_{\varepsilon}^{\mathtt{d}}\colon \partial\Sigma_{I}^{\varepsilon}\to \mathbb{S}^{d-1}$ denotes the outward unit normal vector field of the insulating layer $\Sigma_{I}^{\varepsilon}$.\pagebreak  

    From a physical perspective, the Euler--Lagrange equations \eqref{eq:ELE_Eepsh}  have the following interpretation (\textit{cf}.\ Figure~\ref{fig:domain}(\textit{left})):
    \begin{itemize}[noitemsep,topsep=2pt,leftmargin=!,labelwidth=\widthof{$\bullet$}]
        \item[$\bullet$] The steady-state heat conduction equations
        \eqref{eq:ELE_Eepsh.a} and \eqref{eq:ELE_Eepsh.d} express that the thermally~conducting body $\Omega$ and the insulating layer $\Sigma_{I}^{\varepsilon}$ have different conductivities as well as heat sources: in the thermally conducting body $\Omega$, the material-dependent conductivity $\lambda>0$ is fixed;  in the insulating layer $\Sigma_{I}^{\varepsilon}$, the conductivity $\varepsilon>0$ --in this idealised situation-- is~arbitrarily~small. 
        Moreover, there is no heat source present in the insulating layer $\Sigma_{I}^{\varepsilon}$;

        \item[$\bullet$] On the (possibly empty) (non-)insulated boundary parts $\Gamma_D$ and $\Gamma_N$, we impose the Dirichlet boundary condition  \eqref{eq:ELE_Eepsh.b} (\textit{i.e.}, the temperature distribution $u_{\varepsilon}^{\mathtt{d}}$ at $\Gamma_D$ is fixed to  $u_D$)~or~the  Neumann boundary condition \eqref{eq:ELE_Eepsh.c} (\textit{i.e.}, the heat flux $\lambda\nabla u_{\varepsilon}^{\mathtt{d}}\cdot n$ across $\Gamma_N$~is~fixed~to~$g$),~\mbox{respectively};

        \item[$\bullet$] On the (non-empty) insulated boundary part $\Gamma_I$, we impose the Robin boundary~\mbox{condition}~\eqref{eq:ELE_Eepsh.e}, which states that conductive heat flux $\varepsilon \nabla u_{\varepsilon}^{\mathtt{d}}\cdot n_{\varepsilon}^{\mathtt{d}}$ across the interacting insulation boundary $\Gamma_{I}^{\varepsilon}$ (from $\Sigma_{I}^{\varepsilon}$ to $\mathbb{R}^d\setminus \Omega_{I}^{\varepsilon}$)
        is proportional to the difference between the temperature distribution $u_{\varepsilon}^{\mathtt{d}}$ and the ambient temperature $u_\infty$ at the interacting insulation boundary $\Gamma_{I}^{\varepsilon}$. The proportionality constant is given via the system-specific heat transfer coefficient $\beta>0$;

        \item[$\bullet$] On the remaining boundary parts of the insulating layer $\partial\Sigma_{I}^{\varepsilon}\setminus(\Gamma_I\cup \Gamma^{\varepsilon}_{I})$ --interpreted as \textit{`artificial'} boundary parts--  we impose the homogeneous Neumann boundary condition \eqref{eq:ELE_Eepsh.f},  which models that heat flux $\lambda\nabla u^{\mathtt{d}}_\varepsilon\cdot n$ across these boundary parts is zero and, as a consequence, that the  heat flow can only 
       be transverse to these boundary parts; modelling \textit{`perfect insulation'}.

        \item[$\bullet$] The transmission condition \eqref{eq:ELE_Eepsh.g} imposes that the heat flux $\lambda\nabla (u_\varepsilon^{\mathtt{d}}|_{\Omega})\cdot n$ out of the thermally conducting body $\Omega$ has to be the same  as the heat flux $-\varepsilon \nabla (u^{\mathtt{d}}_\varepsilon|_{\smash{\Sigma_{I}^{\varepsilon}}})\cdot n^{\mathtt{d}}_\varepsilon$
        into the insulating~layer $\Sigma_{I}^{\varepsilon}$. Since
        the conductivity of the insulating layer $\Sigma_{I}^{\varepsilon}$ is arbitrarily small (\textit{i.e.}, $\varepsilon\ll 1$), 
        the temperature gradient  $\nabla u^{\mathtt{d}}_\varepsilon\cdot n^{\mathtt{d}}_\varepsilon$ must be proportionally larger than $\nabla u^{\mathtt{d}}_\varepsilon\cdot n$ to carry~the~same~flux.\medskip
    \end{itemize}

    In the case 
    $k\in \smash{(C^{0,1}(\Gamma_I))^d}$ and $\mathtt{d}\in\smash{C^{0,1}(\Gamma_I)}$ with $\mathtt{d}\ge \mathtt{d}_{\textup{min}}$ a.e.\ in $\Gamma_I$, for a constant $\mathtt{d}_{\textup{min}}>0$, if either $\Gamma_I\in C^{1,1}$ with $k=n$ or $\Gamma_I$ is piece-wise flat with $d\leq 4$,
    passing~to~the~limit~(as~${\varepsilon\to 0^{+}}$) with a family of trivial extensions to $L^2(\mathbb{R}^d)$ of the heat loss functionals ${\smash{E}_{\varepsilon}^{\mathtt{d}}\colon H^1(\Omega_{I}^{\varepsilon})\to \mathbb{R}\cup\{+\infty\}}$, $\varepsilon>0$, 
    in the sense of $\Gamma(L^2(\mathbb{R}^d))$-convergence (\textit{cf}.\ Theorem \ref{thm:main}), we arrive at  the $\Gamma$-limit functional $\smash{E}^{\mathtt{d}}\colon \hspace{-0.1em}H^1(\Omega)\to \mathbb{R}\cup\{+\infty\}$, for every $v\in H^1(\Omega)$ defined by
    \begin{align}\label{eq:Eh}
       \smash{E}^{\mathtt{d}}(v)\coloneqq\left\{\begin{aligned}
        &\tfrac{\lambda}{2}\|\nabla v\|_{\Omega}^2+\tfrac{\beta}{2}\|(1+\beta (k\cdot n)\mathtt{d})^{-\smash{\frac{1}{2}}}\{v-u_{\infty}\}\|_{\Gamma_I}^2\\&-(f,v)_{\Omega}-\langle g,v\rangle_{\smash{H^{\smash{\frac{1}{2}}}(\Gamma_N)}}+I_{\{u_D\}}^{\Gamma_D}(v)\,.
       \end{aligned}\right.
    \end{align}
    In the $\Gamma$-limit functional \eqref{eq:Eh}, the second term is the~\textit{`interface'} heat loss,  accounting~for~the~inter-\-action of the system  with the exterior at the insulated boundary $\Gamma_I$,  mediated by the scaled distribution function $(k\cdot n)\mathtt{d}$ (\textit{i.e.}, large local temperature differences between the system~and~the exterior at insulated boundary need to be compensated with a large locally~scaled~\mbox{distribution}~\mbox{function}).\medskip\enlargethispage{4mm} 
   
    Since the $\Gamma$-limit functional \eqref{eq:Eh} is proper, strictly convex, weakly coercive, and lower~semi-continuous, for given distribution function $\mathtt{d}\in L^\infty(\Gamma_I)$, the
    direct method in the calculus of variations yields the existence of a unique temperature distribution $u^{\mathtt{d}}\in H^1(\Omega)$, which  formally satisfies the Euler--Lagrange equations (\textit{cf}.~Figure~\ref{fig:domain}(\textit{right}))
    \begin{subequations}  \label{eq:ELE_Eh}
        \begin{alignat}{2}
            -\kappa \Delta u^{\mathtt{d}}&=f&&\quad \text{ a.e.\ in }\Omega\,, \label{eq:ELE_Eh.a}\\
            \smash{u^{\mathtt{d}}} &=u_D&&\quad \text{ a.e.\ on }\Gamma_D\,, \label{eq:ELE_Eh.b}\\
            \nabla u^{\mathtt{d}}\cdot  n&=g&&\quad \text{ a.e.\ on }\Gamma_N\,, \label{eq:ELE_Eh.c}\\
              \lambda(1+\beta (k\cdot n)\mathtt{d})\nabla u^{\mathtt{d}}\cdot n+\beta \{u^{\mathtt{d}}-u_{\infty}\}&=0&&\quad \text{ a.e.\ on }\Gamma_I\,, \label{eq:ELE_Eh.d}
        \end{alignat}
    \end{subequations}
    where the boundary condition \eqref{eq:ELE_Eh.d} is still of Robin type, but with (distribution-dependent) variable coefficient $\lambda(1+\beta (k\cdot n)\mathtt{d})$.\medskip\enlargethispage{0.5mm}
     
    We are interested in determining the non-negative distribution function $\mathtt{d}\in  L^\infty(\Gamma_I)$ that provides the best insulating performance, once
the total amount of insulating material~is~fixed. Note that  $\mathtt{d}\in L^\infty(\Gamma_I)$ specifies the distribution of the insulating material in~direction~of~${k\in  (C^0(\partial\Omega))^d}$.\newpage \noindent In practice, however, it is often more convenient to describe the distribution of the insulating material in  direction of  $n\in (L^\infty(\partial\Omega))^d$. The distribution of the insulating material in the direction of $n\in (L^\infty(\partial\Omega))^d$, denoted by $\widetilde{\mathtt{d}}\in L^\infty(\Gamma_I)$, can be computed from $\mathtt{d}\in  L^\infty(\Gamma_I)$~via~(\textit{cf}.~Figure~\ref{fig:thickness}) 
\begin{align}\label{eq:relation_h_h_tilde}
    \smash{\widetilde{\mathtt{d}}=(k\cdot n)\mathtt{d}\quad \text{ a.e.\ on }\Gamma_I\,.}
\end{align}
For this reason, the used total amount of the insulating material should be measured~in~the~weighted norm $\|(k\cdot n)(\cdot)\|_{1,\Gamma_I}$ instead of $\|\cdot\|_{1,\Gamma_I}$, that is in terms of $\widetilde{\mathtt{d}}\in L^\infty(\Gamma_I)$ rather than $\mathtt{d}\in  L^\infty(\Gamma_I)$.\vspace{-2mm}
\begin{figure}[H]
        \centering

  
\tikzset {_ilnl6shss/.code = {\pgfsetadditionalshadetransform{ \pgftransformshift{\pgfpoint{0 bp } { 0 bp }  }  \pgftransformrotate{-270 }  \pgftransformscale{2 }  }}}
\pgfdeclarehorizontalshading{_jxwf7ldi5}{150bp}{rgb(0bp)=(0.89,0.89,0.89);
rgb(37.5bp)=(0.89,0.89,0.89);
rgb(37.5bp)=(1,1,1);
rgb(50bp)=(0.86,0.86,0.86);
rgb(62.5bp)=(0.82,0.82,0.82);
rgb(100bp)=(0.82,0.82,0.82)}

  
\tikzset {_clz09natx/.code = {\pgfsetadditionalshadetransform{ \pgftransformshift{\pgfpoint{0 bp } { 0 bp }  }  \pgftransformrotate{-270 }  \pgftransformscale{2 }  }}}
\pgfdeclarehorizontalshading{_tlqcj11p0}{150bp}{rgb(0bp)=(0.89,0.89,0.89);
rgb(37.5bp)=(0.89,0.89,0.89);
rgb(37.5bp)=(1,1,1);
rgb(50bp)=(0.86,0.86,0.86);
rgb(62.5bp)=(0.82,0.82,0.82);
rgb(100bp)=(0.82,0.82,0.82)}

  
\tikzset {_f5mtjkovw/.code = {\pgfsetadditionalshadetransform{ \pgftransformshift{\pgfpoint{0 bp } { 0 bp }  }  \pgftransformrotate{-270 }  \pgftransformscale{2 }  }}}
\pgfdeclarehorizontalshading{_mch5ik1ol}{150bp}{rgb(0bp)=(0.89,0.89,0.89);
rgb(37.5bp)=(0.89,0.89,0.89);
rgb(37.5bp)=(1,1,1);
rgb(50bp)=(0.86,0.86,0.86);
rgb(62.5bp)=(0.82,0.82,0.82);
rgb(100bp)=(0.82,0.82,0.82)}
\tikzset{every picture/.style={line width=0.75pt}} 

\vspace{-1mm}
        \caption{Sketch of relation between a distribution function $\mathtt{d}\colon \Gamma_I\to [0,+\infty)$ (in direction~of~$k$) and the associated distribution function $\widetilde{\mathtt{d}}\coloneqq (k\cdot n)\mathtt{d}\colon \Gamma_I\to [0,+\infty)$ (in direction of $n$).}
        \label{fig:thickness}
    \end{figure}\vspace{-2.5mm}
    
In light of these considerations, for a fixed amount of  the insulating  material $m>0$, we seek a distribution function $\mathtt{d}\in L^\infty(\Gamma_I)$ (in direction~of~$k$) in the class\vspace{-0.25mm} 
    \begin{align*}
        \mathcal{H}_{I}^m\coloneqq \big\{\overline{\mathtt{d}}\in L^1(\Gamma_I) \mid \overline{\mathtt{d}}\ge 0\text{ a.e.\ on }\Gamma_I\,,\; \|(k\cdot n) \overline{\mathtt{d}}\|_{1,\Gamma_I}=m\big\}\,,
    \end{align*}
    or equivalently (since $\smash{\widetilde{(\cdot)}\coloneqq(\mathtt{d}\mapsto \widetilde{\mathtt{d}})\colon \mathcal{H}_{I}^m\to \widetilde{\mathcal{H}}_{I}^m}$ is a bijection),  a distribution function $\smash{\widetilde{\mathtt{d}}\in L^\infty(\Gamma_I)}$ (in direction~of~$n$) in the class\vspace{-0.25mm}
    \begin{align*}
        \widetilde{\mathcal{H}}_{I}^m\coloneqq \big\{\overline{\mathtt{d}}\in L^1(\Gamma_I)\mid \overline{\mathtt{d}}\ge 0\text{ a.e.\ on }\Gamma_I\,,\; \|\overline{\mathtt{d}}\|_{1,\Gamma_I}=m\big\}\,,
    \end{align*} 
    along with a temperature distribution $\smash{u^{\mathtt{d}}\in H^1(\Omega)}$ that jointly minimize the heat loss, \textit{i.e.}, abbreviating $\widetilde{E}^{\overline{d}}\coloneqq \smash{E}^{\smash{\overline{\mathtt{d}}/(k\cdot n)}}$, one has that\vspace{-0.25mm}
    \begin{align}\label{eq:double_min}
        (u^{\mathtt{d}},\mathtt{d})^\top=\big(u^{\mathtt{d}},\tfrac{\widetilde{\mathtt{d}}}{k\cdot n}\big)^\top\in \underset{(v,\overline{\mathtt{d}})^\top\in H^1(\Omega)\times \mathcal{H}_m}{\textup{arg\,min}}
        {\big\{\smash{E}^{\overline{\mathtt{d}}}(v)\big\}}=\underset{(v,\overline{\mathtt{d}})^\top\in H^1(\Omega)\times \smash{\widetilde{\mathcal{H}}_m}}{\textup{arg\,min}}
        {\big\{\widetilde{E}^{\overline{d}}(v)\big\}}\,.\\[-6.5mm]\notag
    \end{align} 
    In the case $\Gamma_D=\emptyset$, for instance, if
     a \textit{non-trivial net heat input}~\mbox{\textit{condition}}~is~met, \textit{i.e.},~we~have~that\enlargethispage{6.5mm}\vspace{-0.25mm}
    \begin{align}\label{noncompatibility}
        \smash{Q_{\texttt{tot}}\coloneqq(f,1)_{\Omega}+\langle g,1\rangle_{H^{\smash{\frac{1}{2}}}(\Gamma_N)}\neq 0\,,}
    \end{align}
    according to \cite{AKKInsulationNumerics2} (or \cite[Thm.\ 4.1]{PietraNitschScalaTrombetti2021}, in the case of pure insulation (\textit{i.e.}, $\Gamma_I=\partial\Omega$) and trivial ambient temperature (\textit{i.e.}, $u_{\infty}=0$)),   
    then a minimizing pair in \eqref{eq:double_min} exists and~meets~the~relation\vspace{-0.25mm}
    \begin{align*}
        \smash{\widetilde{\mathtt{d}}=(k\cdot n)\mathtt{d}=\tfrac{1}{\beta c_{u^{\mathtt{d}}}}\max\{0,\vert u^{\mathtt{d}}-u^{\infty}\vert -c_{u^{\mathtt{d}}}\}\quad\text{ a.e.\ on }\Gamma_I\,,}
    \end{align*}
    where $c_{u^{\mathtt{d}}}>0$ is a constant, which is implicitly, but uniquely determined via (\textit{cf}.\ \cite[Lem.\ 4.1]{PietraNitschScalaTrombetti2021})\vspace{-0.25mm}
    \begin{align*}
        \smash{c_{u^{\mathtt{d}}}=\tfrac{1}{m\beta}\|\max\{0,\vert u^{\mathtt{d}}-u^{\infty}\vert -c_{u^{\mathtt{d}}}\}\|_{1,\Gamma_I}\,.}
    \end{align*}
    In the non-trivial net heat input condition \eqref{noncompatibility}, the volume integral $(f,1)_{\Omega}$ represents~the~\textit{total volumetric heat generation} inside the thermally conducting body $\Omega$ and the (generalized)~surface integral  $\langle g,1\rangle_{H^{\smash{\frac{1}{2}}}(\Gamma_N)}$ the \textit{total prescribed boundary heat flux} through the Neumann~\mbox{boundary}~part~$\Gamma_N$.
    Therefore, the non-trivial net heat input condition \eqref{noncompatibility}~has~the~\mbox{following}~physical~\mbox{implication}:
    By Gauss' theorem and \eqref{eq:ELE_Eh.d}, the net heat input $Q_{\texttt{tot}}$ equals to the \textit{net convective~heat~loss~through}~$\Gamma_I$, \textit{i.e.},  $Q_{\texttt{conv}}\coloneqq(\beta(1+\beta\widetilde{d})^{-1}\{u-u_\infty\},1)_{\Gamma_I}$ and, by the non-trivial net heat input condition~\eqref{noncompatibility},~is non-trivial. On the contrary, if $Q_{\texttt{conv}}=0$, there would be no heat loss for~the~\mbox{insulation}~to~\mbox{reduce}.
    
    \section{Auxiliary technical tools}\label{sec:tools} 

    \hspace{5mm}In this section, we prove auxiliary technical tools that are needed for the $\Gamma$-convergence analysis in Section \ref{sec:gamma_convergence}. To this end, for the remainder of the paper, we assume that $k\in (C^{0,1}(\Gamma_I))^d$ is a Lipschitz continuous (globally) transversal vector field of $\Omega$ with transversality~constant~$k\in(0,1]$, the existence of which is ensured by Theorem \ref{thm:ex_transversal}. Moreover, if not otherwise specified,~let~${\mathtt{d}\in L^\infty(\Gamma_I)}$  be a given distribution function in transversal direction (\textit{i.e.}, in direction of $k$) and $\widetilde{\mathtt{d}}\in L^\infty(\Gamma_I)$ the associated distribution function in normal direction (\textit{i.e.}, in direction of $n$); related via \eqref{eq:relation_h_h_tilde}. Then, for these two distribution functions, we employ the notation~introduced~in~\eqref{def:some_notation}. 

    \subsection{Approximative transformation formula} 
  
    \hspace{5mm}The assumption $k\in (C^{0,1}(\Gamma_I))^d$ along with its (global) transversality property \eqref{eq:transversal} ensures the existence of a constant $\varepsilon_0>0$ such that for every $\varepsilon\in (0,\varepsilon_0)$, the (global) parametrization  $\Phi_\varepsilon \colon  \hspace{-0.05em}D_{I}^{\varepsilon} \hspace{-0.05em}\coloneqq  \hspace{-0.05em} \bigcup_{s\in \partial\Omega}{\{s\}\times   [0,\varepsilon \mathtt{d}(s))} \hspace{-0.05em}\to\hspace{-0.05em} \Sigma_{I}^{\varepsilon}$ of the insulating layer $\Sigma_{I}^{\varepsilon}$, for~every~${(s,t)^\top \hspace{-0.05em}\in \hspace{-0.05em} D_{I}^{\varepsilon}}$~\mbox{defined}~by
    \begin{align}\label{def:parametrization}
        \Phi_\varepsilon (s,t)\coloneqq s+tk(s)\,,
    \end{align}
    is bi-Lipschitz continuous (see \cite[p.\ 633, 634]{HMT07}, for a detailed discussion), \textit{i.e.}, Lipschitz continuous and bijective with Lipschitz continuous inverse. 
    By means of the global~parametrization~\eqref{def:parametrization}, one~can~prove the following \textit{`approximative'} transformation formula; relating volume~integrals~with respect \hspace{-0.1mm}to \hspace{-0.1mm}the \hspace{-0.1mm}insulating \hspace{-0.1mm}layer \hspace{-0.1mm}$\Sigma_{I}^{\varepsilon}$ \hspace{-0.1mm}with 
    \hspace{-0.1mm}boundary \hspace{-0.1mm}integrals \hspace{-0.1mm}with \hspace{-0.1mm}respect~\hspace{-0.1mm}to~\hspace{-0.1mm}the~\hspace{-0.1mm}\mbox{insulated}~\hspace{-0.1mm}\mbox{boundary}~\hspace{-0.1mm}$\Gamma_I$.

    \begin{lemma}\label{lem:approx_trans_formula}
        For every $\varepsilon\in (0,\varepsilon_0)$ and $v_\varepsilon\in L^1(\Sigma_{I}^{\varepsilon})$, there holds
        \begin{align}\label{lem:approx_trans_formula.0}
            \int_{\Sigma_{I}^{\varepsilon}}{v_\varepsilon\,\mathrm{d}x}=\int_{\Gamma_I}{\int_0^{\varepsilon \mathtt{d}(s)}{v_\varepsilon(s+t k(s))\big\{k(s)\cdot n(s)+tR_\varepsilon(s,t)\big\}\,\mathrm{d}t}\,\mathrm{d}s}\,,
        \end{align}
        where the remainders $R_\varepsilon\in  L^\infty(D_{I}^{\varepsilon})$, $\varepsilon\in (0,\varepsilon_0)$, depend only on the Lipschitz characteristics~of~$\Gamma_I$ and satisfy $\smash{\sup_{\varepsilon\in (0,\varepsilon_0)}{\{\| R_\varepsilon\|_{\infty,\smash{D_{I}^{\varepsilon}}}\}}<+\infty}$.
    \end{lemma}

    \begin{proof}
    See \cite[Lem.\ 4.1]{AKK2025_modelling}.
\end{proof} 

    A similar \textit{`approximative'} transformation formula applies for boundary integrals with respect the interacting insulation boundary~$\Gamma_{I}^{\varepsilon}$; relating the latter to boundary integrals with respect to the insulated boundary $\Gamma_I$.

    \begin{lemma}\label{lem:approx_trans_formula2}
        Let $\mathtt{d}\in C^{0,1}(\Gamma_I)$. Then, for every $\varepsilon\in (0,\varepsilon_0)$ and $v_\varepsilon\in L^1(\Gamma_{I}^{\varepsilon})$, there holds
        \begin{align}\label{lem:approx_trans_formula.0}
            \int_{\Gamma_{I}^{\varepsilon}}{v_\varepsilon\,\mathrm{d}s}=\int_{\Gamma_I}{v_\varepsilon(s+\varepsilon \mathtt{d}(s) k(s))\big\{1+\varepsilon^{\frac{1}{2}}r_\varepsilon(s)\big\}\,\mathrm{d}s}\,,
        \end{align}
        where the remainders $r_\varepsilon\in  L^\infty(\Gamma_I)$, $\varepsilon\in (0,\varepsilon_0)$, depend only on the Lipschitz characteristics~of~$\Gamma_I$ and satisfy $\smash{\sup_{\varepsilon\in (0,\varepsilon_0)}{\{\| r_\varepsilon\|_{\infty,\Gamma_I}\}}<+\infty}$.
    \end{lemma}

    As an immediate  consequence of Lemma \ref{lem:approx_trans_formula2}, we obtain the following norm equivalence on $L^p(\Gamma_I^{\varepsilon})$, $p\in [1,+\infty)$.
 
    \begin{corollary}\label{lem:equiv}
        
          Let $\mathtt{d}\in C^{0,1}(\Gamma_I)$. Then, for every $\varepsilon\in (0,\varepsilon_0)$ and $v_\varepsilon\in L^p(\Gamma_{I}^{\varepsilon})$,~${p\in [1,+\infty)}$,~there holds
        \begin{align*}
            (1-\varepsilon^{\smash{\frac{1}{2}}}\| r_\varepsilon\|_{\infty,\Gamma_I})^{-\frac{1}{p}}\|v_\varepsilon(\cdot+\varepsilon\mathtt{d}k)\|_{p,\Gamma_{I}}\leq \|v_\varepsilon\|_{p,\Gamma_{I}^{\varepsilon}}\leq (1+\varepsilon^{\smash{\frac{1}{2}}}\| r_\varepsilon\|_{\infty,\Gamma_I})^{\frac{1}{p}}\|v_\varepsilon(\cdot+\varepsilon\mathtt{d}k)\|_{p,\Gamma_{I}}\,,
        \end{align*}
        where the remainders $r_\varepsilon\in  L^\infty(\Gamma_I)$, $\varepsilon\in (0,\varepsilon_0)$, are as in Lemma \ref{lem:approx_trans_formula2}. 
    \end{corollary}  

    \begin{proof}
        The claimed norm equivalence is a direct consequence of 
        $1-\varepsilon^{\smash{\frac{1}{2}}}\|r_\varepsilon\|_{\infty,\Gamma_I}\leq 1+\varepsilon^{\smash{\frac{1}{2}}}r_\varepsilon(s)\leq 1+\varepsilon^{\smash{\frac{1}{2}}}\|r_\varepsilon\|_{\infty,\Gamma_I} $ for a.e.\ $s\in \Gamma_I$ and all $\varepsilon\in (0,\varepsilon_0)$.
    \end{proof}
    
    \begin{proof}[Proof (of Lemma \ref{lem:approx_trans_formula2})]
    As $\Omega$ is a bounded Lipschitz domain, there exist a radius~${r>0}$~as~well~as~a finite~number $N\hspace{-0.1em}\in\hspace{-0.1em} \mathbb{N}$ of affine isometric mappings $A_i\hspace{-0.1em}\coloneqq \hspace{-0.1em}(x\mapsto O_ix+b_i)\colon  \hspace{-0.1em}\mathbb{R}^d\hspace{-0.1em}\to\hspace{-0.1em}  \mathbb{R}^d$,~where~${O_i\hspace{-0.1em}\in\hspace{-0.1em} \mathrm{O}(d)}$\footnote{$\mathrm{O}(d)\coloneqq \{O\in \mathbb{R}^{d\times d}\mid O^\top=O^{-1}\}$.} and $b_i\hspace{-0.1em}\in\hspace{-0.1em} \mathbb{R}^d$, $i=1,\ldots,N$, and Lipschitz mappings
        $\gamma_i\colon  B_r\hspace{-0.1em}\coloneqq \hspace{-0.1em}B_r^{d-1}(0)\hspace{-0.1em}\to\hspace{-0.1em} \mathbb{R}$,~$i=1,\ldots,N$,~such~that\enlargethispage{3.5mm}
        \begin{align}\label{lem:approx_trans_formula.1}
            \partial\Omega=\bigcup_{i=1}^N{A_i(\mathrm{graph}(\gamma_i))}\,.
        \end{align}
        Moreover, the local parametrizations $s_i\colon B_r\to \partial\Omega\cap s_i(B_r)$, $i=1,\ldots,N$, of the (topological) boundary $\partial\Omega$, for every $i=1,\ldots,N$ and $\overline{x}\in B_r$ defined by 
        \begin{align}\label{lem:approx_trans_formula.2}
            s_i(\overline{x})\coloneqq A_i(\overline{x},\gamma_{i}(\overline{x}))\,,
        \end{align}
        are bi-Lipschtz continuous and their generalized Jacobian determinants $J_{s_i}\colon B_r\to \mathbb{R}$, $i=1,\ldots,N$, 
        for every $i=1,\ldots,N$ and $\overline{x}\in B_r$, are given via
        \begin{align}\label{lem:approx_trans_formula.3}
            J_{s_i}(\overline{x})\coloneqq (1+\vert\nabla \gamma_{i}(\overline{x})\vert^2)^{\frac{1}{2}}\,.
        \end{align}
        Next, let $i\hspace{-0.15em}=\hspace{-0.15em}1,\ldots,N$  be fixed, but arbitrary. 
        Then, the local parametrization ${F_\varepsilon^{i} \colon  \hspace{-0.15em}B_r\hspace{-0.15em}\to\hspace{-0.15em} \Gamma_{I}^{\varepsilon}\hspace{-0.15em}\cap\hspace{-0.15em} F_\varepsilon^{i}(B_r)}$ of the interacting insulation boundary $\Gamma_{I}^{\varepsilon}$, for every $\overline{x} \in B_r$ defined by
        \begin{align}\label{lem:approx_trans_formula.4}
            F_\varepsilon^{i}(\overline{x})\coloneqq \Phi_\varepsilon(s_i(\overline{x}),\varepsilon \mathtt{d}(s_i(\overline{x})))\,,
        \end{align} 
        is bi-Lipschitz continuous and, due to Rademacher's theorem (\textit{cf}.\ \cite[Thm.\ 2.14]{AFP2000}),
        for a.e.\ $\overline{x}\in B_r$, we have that
        \begin{align}\label{lem:approx_trans_formula.5}
            \mathrm{D}F_\varepsilon^{i}(\overline{x})=
            O_i\left[\begin{array}{c}
                \mathrm{I}_{(d-1)\times (d-1)} \\\cmidrule(){1-1}
                \nabla \gamma_{i}(\overline{x})^\top\vphantom{X^{X^X}} 
\end{array}\right]+\varepsilon\{\nabla\mathtt{d}\otimes k+\mathtt{d}\mathrm{D}k \}(s_i(\overline{x}))O_i\left[\begin{array}{c}
                \mathrm{I}_{(d-1)\times (d-1)} \\\cmidrule(){1-1}
                \nabla \gamma_{i}(\overline{x})^\top\vphantom{X^{X^X}} 
            \end{array}\right]\,.
        \end{align} 
       Then, from the representation \eqref{lem:approx_trans_formula.5}, we deduce the existence of  a remainder term $r_{\varepsilon}^{i}\in L^\infty(B_r)$, depending only on the Lipschitz characteristics of $\Gamma_I$ and $\mathtt{d}$, with  ${\sup_{\varepsilon\in (0,\varepsilon_0)}{\{\|r_{\varepsilon}^{i}\|_{\infty,B_r}\}}<+\infty}$, such~that the generalized Jacobian determinant of the local parametrization \eqref{lem:approx_trans_formula.4}, 
        for a.e.\ $\overline{x}\in B_r$, \mbox{using}~\eqref{lem:approx_trans_formula.3},
        can be written as
        \begin{align*}
            J_{F_\varepsilon^{i}}(\overline{x})&=\smash{\textup{det}\big(\mathrm{D}F_\varepsilon^{i}(\overline{x})^\top\mathrm{D}F_\varepsilon^{i}(\overline{x})\big)^{\frac{1}{2}}}\\&=\smash{(1+\vert\nabla \gamma_{i}(\overline{x})\vert^2)^{\frac{1}{2}}
            +\varepsilon^{\frac{1}{2}} r_{\varepsilon}^{i}(\overline{x})}
            \\&= J_{s_i}(\overline{x})+\varepsilon^{\smash{\frac{1}{2}}}r_{\varepsilon}^{i}(\overline{x})\,.
        \end{align*}
        Hence, if 
         $(\eta_i)_{i=1,\ldots,N}\subseteq C_0^\infty(\mathbb{R}^d)$ is a partition of unity subordinate to the open covering of $\Gamma_{I}^{\varepsilon}$~by $(F_\varepsilon^{i}(B_r))_{i=1,\ldots,N}\subseteq \mathbb{R}^d$, \textit{i.e.}, $\sum_{i=1}^N{\eta_i}=1$ in $\Gamma_{I}^{\varepsilon}$ and $\textup{supp}\,\eta_i\subseteq F_\varepsilon^{i}(B_r)$ for all $i=1,\ldots,N$, then, by the definitions of the surface integrals on $\Gamma_{I}^{\varepsilon}$ and $\Gamma_{I}$, respectively, we conclude that
        \begin{align*}
             \int_{\Gamma_{I}^{\varepsilon}}{v\,\mathrm{d}s} 
             &= \sum_{i=1}^N{\int_{B_r}{((\eta_iv)\circ F_\varepsilon^{i})J_{F_\varepsilon^{i}}\,\mathrm{d}\overline{x}}}
             \\&= \sum_{i=1}^N{\int_{B_r}{(\eta_iv)(s_i(\overline{x})+\varepsilon \mathtt{d}(s_i(\overline{x}))k(s_i(\overline{x}))) \big\{J_{s_i}(\overline{x})+\varepsilon^{\frac{1}{2}} \,r_{\varepsilon}^{i}(\overline{x})\big\}\,\mathrm{d}\overline{x}}}
             \\&= \sum_{i=1}^N{\int_{B_r}{(\eta_iv)(s_i(\overline{x})+\varepsilon \mathtt{d}(s_i(\overline{x}))k(s_i(\overline{x}))) \big\{1+\varepsilon^{\frac{1}{2}}\tfrac{r_{\varepsilon}^{i}(\overline{x})}{J_{s_i}(\overline{x})}\big\}J_{s_i}(\overline{x})\,\mathrm{d}\overline{x}}}
             \\&=\int_{\Gamma_I}{v(s+\varepsilon \mathtt{d}(s)k(s))\bigg\{1+\varepsilon^{\frac{1}{2}} \smash{\sum_{i=1}^N}{\tfrac{r_\varepsilon^{i}(s_i^{-1}(s))}{J_{s_i}(s_i^{-1}(s))}\chi_{s_i(B_r)}}(s)\bigg\}\,\mathrm{d}s}
             \\&=\int_{\Gamma_I}{v(\cdot+\varepsilon \mathtt{d}k)\big\{1+\varepsilon^{\smash{\frac{1}{2}}}r_\varepsilon\big\}\,\mathrm{d}s}
             \,,
        \end{align*}
        which is the claimed approximative transformation  formula \eqref{lem:approx_trans_formula.0}. 
\end{proof} 

    \subsection{Lebesgue differentiation theorem with respect to vanishing insulating layers}

    \hspace{5mm}By means of the approximative transformation formula (\textit{cf}.\ Lemma \ref{lem:approx_trans_formula}), one can prove a Lebesgue differentiation theorem with respect to vanishing insulating layers.
    
    \begin{lemma}\label{lem:Lebesgue_boundary_limit}
        Let $a\in L^\infty(\Gamma_I)$  and $v\in H^{1,p}(\Sigma_{I}^{\varepsilon_0})$, $p\in [1,+\infty)$. 
        Then, there holds
        \begin{align}\label{lem:Lebesgue_boundary_limit.0}
            \tfrac{1}{\varepsilon}\smash{\|a^{\smash{\frac{1}{p}}}v\|_{p,\Sigma_{I}^{\varepsilon}}^p}\to \smash{\|(\widetilde{\mathtt{d}}a)^{\smash{\frac{1}{p}}}v\|_{p,\Gamma_I}^p}\quad (\varepsilon\to 0^{+})\,,
        \end{align}
        where $a\hspace{-0.1em}\in \hspace{-0.1em}L^\infty(\Sigma_{I}^{\varepsilon_0})$ denotes the not relabelled extension of $a\hspace{-0.1em}\in \hspace{-0.1em}L^\infty(\Gamma_I)$, for a.e.\  ${x\hspace{-0.1em}=\hspace{-0.1em}s+tk(s)\hspace{-0.1em}\in \hspace{-0.1em}\Sigma_{I}^{\varepsilon_0}}$, where $s\in \Gamma_I$ and $t\in [0,\varepsilon_0 \mathtt{d}(s))$, defined by  $a(x)\coloneqq a(s)$.
    \end{lemma}

    \begin{proof} See \cite[Lem.\ 4.2]{AKK2025_modelling}.
    \end{proof}

    \subsection{Poincar\'e inequalities in insulating layers}
    
    \hspace{5mm}In the forthcoming analysis, we will  resort to the following point-wise Poincar\'e inequality for Sobolev functions  defined in the insulating layer
    $\Sigma_{I}^{\varepsilon}$.

    \begin{lemma}\label{lem:poincare}
        Let $\varepsilon\hspace{-0.1em}\in\hspace{-0.1em} (0,\varepsilon_0)$ and $v_\varepsilon\hspace{-0.1em}\in\hspace{-0.1em} H^1(\Sigma_{I}^{\varepsilon})$. Then, for a.e.\ $s\hspace{-0.1em}\in\hspace{-0.1em} \Gamma_I$ and $t,\widetilde{t}\hspace{-0.1em}\in\hspace{-0.1em} [0,\varepsilon \mathtt{d}(s)]$~with~$t\hspace{-0.1em}\ge\hspace{-0.1em} \widetilde{t}$, there holds
        \begin{align}\label{lem:poincare.0}
            \vert v_\varepsilon(s+tk(s))-v_\varepsilon(s+\widetilde{t}k(s))\vert^2\leq (t-\widetilde{t})
            \int_{\widetilde{t}}^t{\vert \nabla v_\varepsilon (s+\lambda k(s))\vert^2\,\mathrm{d}\lambda}\,.
        \end{align} 
    \end{lemma}

    \begin{proof}
        Resorting to the Newton--Leibniz formula and Jensen's inequality,
        for a.e.\ $s\in \Gamma_I$ and every $\widetilde{t},t\in  [0,\varepsilon \mathtt{d}(s)]$ with $t\ge \widetilde{t}$, we find that
        \begin{align*}
            \vert v_\varepsilon (s+t k(s))-v_\varepsilon(s+\widetilde{t}k(s))\vert^2 
            &
            =\left\vert\int_{\widetilde{t}}^t{\nabla v_\varepsilon (s+\lambda k(s))\cdot k(s)\,\mathrm{d}\lambda}\right\vert^2
            \\&\leq (t-\widetilde{t})
            \int_{\widetilde{t}}^t{\vert \nabla v_\varepsilon (s+\lambda k(s))\cdot k(s)\vert^2\,\mathrm{d}\lambda}\,,
        \end{align*}
        which, using that $\vert k\vert=1$ a.e.\ on $\Gamma_I$, yields the claimed point-wise Poincar\'e inequality \eqref{lem:poincare.0}.
    \end{proof}

    By means of the point-wise Poincar\'e inequality (\textit{cf}.\ Lemma \ref{lem:poincare}) and the approximative~transfor\-mation formula (\textit{cf}.\ Lemma \ref{lem:approx_trans_formula}), we obtain the following Poincar\'e inequalities.\enlargethispage{2.5mm}

    \begin{corollary}\label{lem:poincare_int}
        Let $\mathtt{d}\in C^{0,1}(\Gamma_I)$, $\varepsilon\in (0,\varepsilon_0)$, and $v_\varepsilon\in H^1(\Sigma_{I}^{\varepsilon})$. Then,  there holds
        \begin{align}\label{lem:poincare_1}
       \hspace{-2.5mm} \|\mathtt{d}^{\smash{-\frac{1}{2}}}\{v_\varepsilon(\cdot\hspace{-0.125em}+\hspace{-0.125em}\varepsilon \mathtt{d}k)\hspace{-0.125em}-\hspace{-0.125em}v_\varepsilon\}\|_{\Gamma_{I}}^2&\hspace{-0.125em}\leq \hspace{-0.125em}
            \tfrac{\varepsilon}{\kappa-\varepsilon\|\mathtt{d}\|_{\infty,\Gamma_I}\|R_\varepsilon\|_{\infty,\smash{D_{I}^{\varepsilon}}}}\| \nabla v_\varepsilon\|_{\Sigma_{I}^{\varepsilon}}^2\,,\\
            \|v_\varepsilon\|_{\Gamma_{I}^{\varepsilon}}^2&\hspace{-0.125em}\leq\hspace{-0.125em} 2\{1\hspace{-0.125em}+\hspace{-0.125em}\smash{\varepsilon^{\smash{\frac{1}{2}}}}\| r_\varepsilon\|_{\infty,\Gamma_I}\}\big\{
            \tfrac{\varepsilon\mathtt{d}_{\min}}{\kappa-\varepsilon\|\mathtt{d}\|_{\infty,\Gamma_I}\hspace{-0.125em}\|R_\varepsilon\|_{\infty,\smash{D_{I}^{\varepsilon}}}}\| \nabla v_\varepsilon\|_{\Sigma_{I}^{\varepsilon}}^2\hspace{-0.125em}+\hspace{-0.125em}\|v_{\varepsilon}\|_{\Gamma_{I}}^2\big\}\,,\hspace{-0.5mm}\label{lem:poincare_2}
        \end{align} 
        where the remainders $R_\varepsilon\hspace{-0.05em}\in  \hspace{-0.05em}L^\infty(D_{I}^{\varepsilon})$, $\varepsilon\hspace{-0.05em}\in\hspace{-0.05em} (0,\varepsilon_0)$, and  $r_\varepsilon\hspace{-0.05em}\in\hspace{-0.05em}  L^\infty(\Gamma_{I})$, $\varepsilon\hspace{-0.05em}\in \hspace{-0.05em}(0,\varepsilon_0)$, are~as~in~Lemma~\ref{lem:approx_trans_formula} and in Lemma \ref{lem:approx_trans_formula2}, respectively.
    \end{corollary}

    \begin{proof}
        \emph{ad \eqref{lem:poincare_1}.} Using the point-wise Poincar\'e inequality (\textit{cf}.\ Lemma \ref{lem:poincare} with $t=\varepsilon \mathtt{d}(s)$ and $\widetilde{t}=0$  for a.e.\ $s\in \Gamma_I$) and that $k(s)\cdot n(s)+t R_\varepsilon(s,t)\hspace{-0.15em}\ge\hspace{-0.15em} \kappa- \varepsilon\|\mathtt{d}\|_{\infty,\Gamma_I}\|R_\varepsilon\|_{\infty,\smash{D_{I}^{\varepsilon}}}$ for a.e.\ $\smash{(t,s)^\top}\hspace{-0.15em}\in\hspace{-0.15em} D_I^{\varepsilon}$ together with the approximative transformation~formula (\textit{cf}.\ Lemma \ref{lem:approx_trans_formula}), we obtain
        \begin{align*} 
                \|\mathtt{d}^{\smash{-\frac{1}{2}}}\{v_\varepsilon(\cdot+\varepsilon \mathtt{d}n)-v_\varepsilon\}\|_{\Gamma_{I}}^2&\leq \int_{\Gamma_I}{\varepsilon
            \int_0^{\varepsilon \mathtt{d}(s)}{\vert \nabla v_\varepsilon (s+t k(s))\vert^2\,\mathrm{d}t}\,\mathrm{d}s}
            \\&\leq \varepsilon\int_{\Gamma_I}{
            \int_0^{\varepsilon \mathtt{d}(s)}{\vert \nabla v_\varepsilon (s+t k(s))\vert^2\tfrac{k(s)\cdot n(s)+tR_\varepsilon(s,t)}{\kappa-\varepsilon \|\mathtt{d}\|_{\infty,\Gamma_I}\|R_\varepsilon\|_{\infty,\smash{D_{I}^{\varepsilon}}}} \,\mathrm{d}t}\,\mathrm{d}s}\notag 
            \\&= \smash{\tfrac{\varepsilon}{\kappa-\varepsilon \|\mathtt{d}\|_{\infty,\Gamma_I}\|R_\varepsilon\|_{\infty,\smash{D_{I}^{\varepsilon}}}}}\|\nabla v_\varepsilon\|_{\Sigma_{I}^{\varepsilon}}^2\,,
        \end{align*}
        which is the claimed Poincar\'e inequality \eqref{lem:poincare_1}.

        \emph{ad \eqref{lem:poincare_2}.} We combine Corollary \ref{lem:equiv} with \eqref{lem:poincare_1}.
    \end{proof}\newpage

    \subsection{Equi-coercivity}

    \hspace{5mm}The \hspace{-0.15mm}family \hspace{-0.15mm}of \hspace{-0.15mm}heat \hspace{-0.15mm}loss \hspace{-0.15mm}functionals \hspace{-0.15mm}$\smash{E}_\varepsilon^{\mathtt{d}}\colon \hspace{-0.15em}H^1(\Omega_{I}^{\varepsilon})\hspace{-0.15em}\to\hspace{-0.15em}\mathbb{R}\cup\{+\infty\}$, \hspace{-0.15mm}$\varepsilon\hspace{-0.15em}\in\hspace{-0.15em} (0,\varepsilon_0)$,~\hspace{-0.15mm}(\textit{cf}.~\hspace{-0.15mm}\eqref{eq:Evarh})~\hspace{-0.15mm}is~\hspace{-0.15mm}\mbox{\emph{equi-coercive}}.

    \begin{lemma}\label{lem:equicoercive}
 Let $\mathtt{d}\in C^{0,1}(\Gamma_I)$. Then, for a sequence $v_\varepsilon\in H^1(\Omega_{I}^{\varepsilon})$, $\varepsilon\in (0,\varepsilon_0)$, from\vspace{-0.5mm}
   \begin{align}\label{lem:equicoercive.0.1}
      \sup_{\varepsilon\in (0,\varepsilon_0)}{\big\{\smash{E}_\varepsilon^{\mathtt{d}}(v_\varepsilon)\big\}}<+\infty  \,,
   \end{align}
   it~follows~that\vspace{-0.5mm}
        \begin{align}\label{lem:equicoercive.0.2}
         \sup_{\varepsilon\in (0,\varepsilon_0)}{\big\{\|v_\varepsilon\|_{\Omega}^2+\|\nabla v_\varepsilon\|_{\Omega}^2+\|v_\varepsilon\|_{\Gamma_{I}^{\varepsilon}}^2+\varepsilon\|\nabla v_\varepsilon\|_{\Sigma_{I}^{\varepsilon}}^2\big\}}<+\infty \,.
        \end{align} 
    \end{lemma}
 
    \begin{proof}  
        To begin with,  from \eqref{lem:equicoercive.0.1},~we infer that
        $v_\varepsilon=u_D$ a.e.\ on $\Gamma_D$ and, due to~Young's~inequality, for every $\delta>0$, we find that
        \begin{align}\label{lem:equicoercive.1}
            \begin{aligned} 
            \tfrac{\lambda}{2}\|\nabla v_\varepsilon\|_{\Omega}^2+\tfrac{\varepsilon}{2}\|\nabla v_\varepsilon\|_{\Sigma_{I}^{\varepsilon}}^2+\tfrac{\beta}{2}\|v_\varepsilon -u_{\infty}\|_{\Gamma_{I}^{\varepsilon}}^2\leq \smash{E}_\varepsilon^{\mathtt{d}}(v_\varepsilon)&+\tfrac{1}{2\delta}\{\|f\|_{\Omega}^2+\| g\|_{\smash{(H^{\smash{\frac{1}{2}}}(\Gamma_N))^*}}^2\}\\&\quad+\tfrac{\delta}{2}\{\|v_\varepsilon\|_{\Omega}^2+\| v_\varepsilon\|_{\smash{H^{\smash{\frac{1}{2}}}(\Gamma_N)}}^2\}\,.
            \end{aligned}
        \end{align}
        By \hspace{-0.1mm}Friedrich's \hspace{-0.1mm}inequality \hspace{-0.1mm}\eqref{lem:poin_cont} \hspace{-0.1mm}and \hspace{-0.1mm}the \hspace{-0.1mm}trace \hspace{-0.1mm}theorem \hspace{-0.1mm}(\textit{cf}.~\hspace{-0.1mm}\cite[Thm.~\hspace{-0.1mm}II.4.3]{Galdi}), \hspace{-0.1mm}respectively,~\hspace{-0.1mm}there~\hspace{-0.1mm}holds\vspace{-4.5mm}
         \begin{subequations}\label{lem:equicoercive.4}
        \begin{align}\label{lem:equicoercive.4.1}
            \|v_\varepsilon\|_{\Omega}^2&\leq c_{\textrm{F}}\,\{\|\nabla v_\varepsilon \|_{\Omega}^2+\| v_\varepsilon\|_{\Gamma_I}^2\}\,,\\\label{lem:equicoercive.4.2}
             \| v_\varepsilon \|_{\smash{H^{\smash{\frac{1}{2}}}(\Gamma_N)}}^2&\leq c_{\textrm{Tr}}\,\{\|\nabla v_\varepsilon\|_{\Omega}^2+
            \|v_\varepsilon\|_{\Omega}^2\}\,.
        \end{align}
        \end{subequations}
        On the other hand, resorting to Corollary \ref{lem:poincare_int}\eqref{lem:poincare_1} and Corollary \ref{lem:equiv}, we observe that
        \begin{align}\label{lem:equicoercive.2}
        \begin{aligned} 
                 \|v_\varepsilon-u_\infty\|_{\Gamma_I}^2&\leq 2\{\|\{v_\varepsilon-u_\infty\}(\cdot +\varepsilon \mathtt{d}k)-\{v_\varepsilon-u_\infty\}\|_{\Gamma_I}^2+\|\{v_\varepsilon-u_\infty\}(\cdot +\varepsilon \mathtt{d}k)\|_{\Gamma_I}^2\}
                 \\&\leq 2\big\{\smash{\tfrac{\varepsilon\mathtt{d}_{\min}}{\kappa-\varepsilon \|\mathtt{d}\|_{\infty,\Gamma_I}\|R_\varepsilon\|_{\infty,\smash{D_{I}^{\varepsilon}}}}}
            \|\nabla \{v_\varepsilon -u_\infty\}\|_{\Sigma_{I}^{\varepsilon}}^2+ \tfrac{1}{1-\varepsilon^{1/2}\|r_\varepsilon\|_{\infty,\Gamma_I}}\|v_\varepsilon-u_\infty\|_{\Gamma_I^{\varepsilon}}^2\big\}
                 \,. 
                 \end{aligned}
        \end{align} 
        In summary, using \eqref{lem:equicoercive.0.1}, \eqref{lem:equicoercive.2} together with $\|\nabla u_\infty\|_{\Sigma_{I}^{\varepsilon}}\leq \|\nabla u_\infty\|_{\mathbb{R}^d\setminus\overline{\Omega}}$ (since $\Sigma_{I}^{\varepsilon}\subseteq \mathbb{R}^d\setminus\overline{\Omega}$), and \eqref{lem:equicoercive.4}
         in \eqref{lem:equicoercive.1}, for $\delta>0$ sufficiently small, we deduce that
        \begin{align}\label{lem:equicoercive.5}
           \sup_{\varepsilon\in (0,\varepsilon_0)}{\big\{\|\nabla v_\varepsilon\|_{\Omega}^2+\varepsilon\|\nabla v_\varepsilon\|_{\Sigma_{I}^{\varepsilon}}^2+\|v_\varepsilon-u_{\infty}\|_{\Gamma_{I}^{\varepsilon}}^2\big\}}<+\infty \,.
        \end{align}
        Eventually, from $\limsup_{\varepsilon\to 0^+}{\{\|u_{\infty}\|_{\Gamma_{I}^{\varepsilon}}\}} \leq  2\|u_{\infty}\|_{\Gamma_{I}}$ (\textit{cf}.\ Corollary \ref{lem:poincare_int}\eqref{lem:poincare_2}) and \eqref{lem:equicoercive.2},\eqref{lem:equicoercive.5} in \eqref{lem:equicoercive.4.1}, we conclude that  \eqref{lem:equicoercive.0.2} and, thus, the claimed equi-coercivity property applies.
    \end{proof}

    \subsection{Transversal distance function}

    \hspace{5mm}In order to establish the $\limsup$-estimate in the later $\Gamma$-convergence analysis (\textit{cf}.\ Lemma~\ref{lem:limsup_case2}), it is central to measure the distance of points in the insulating layer $\Sigma_{I}^{\varepsilon}$ to the insulated boundary $\Gamma_I$ with respect to the Lipschitz continuous (globally) transversal vector field $k\in (C^{0,1}(\partial\Omega))^d$.
    The latter is provided by the transversal distance function, the definition and most important properties of which can be found in the following lemma.
    
    \begin{lemma}\label{lem:transversal_distance_function}
        For each $\varepsilon\in (0,\varepsilon_0)$, let the \emph{transversal distance function} $\psi_\varepsilon\colon \Sigma_{I}^{\varepsilon}\to [0,\varepsilon \|\mathtt{d}\|_{\infty,\Gamma_I})$, for every $x=s+tk(s)\in \Sigma_{I}^{\varepsilon}$, where $s\in \Gamma_I$ and $t\in [0,\varepsilon \mathtt{d}(s))$, be defined by\enlargethispage{2mm}
        \begin{align*}
            \psi_\varepsilon(x)\coloneqq t\,.
        \end{align*}
        Then, we have that $\psi_\varepsilon\in H^{1,\infty}(\Sigma_{I}^{\varepsilon})$ with  $\psi_\varepsilon=0$ a.e.\ on $\Gamma_I$ and 
        \begin{subequations} \label{lem:transversal_distance_function.0}
        \begin{align}\label{lem:transversal_distance_function.1}
        \|\psi_\varepsilon \|_{\infty,\Sigma_{I}^{\varepsilon}}&\leq \varepsilon\|\mathtt{d}\|_{\infty,\Gamma_I}\,,\\
            \nabla \psi_\varepsilon(x)&= \tfrac{1}{k(s)\cdot n(s)}n(s)+tR_\varepsilon(x)\quad \text{ for a.e.\ }x=s+tk(s)\in \Sigma_{I}^{\varepsilon}\,,\label{lem:transversal_distance_function.2}
        \end{align}
         \end{subequations}
        where the remainders $R_\varepsilon\hspace{-0.1em}\in \hspace{-0.1em}(L^\infty(\Sigma_{I}^{\varepsilon}))^d$, $\varepsilon\hspace{-0.1em}\in \hspace{-0.1em}(0,\varepsilon_0)$, depend only on the Lipschitz characteristics of $\Gamma_I$ and satisfy $\sup_{\varepsilon\in (0,\varepsilon_0)}{\{\|R_\varepsilon\|_{\infty,\smash{\Sigma_{I}^{\varepsilon}}}\}}<+\infty$.
    \end{lemma}

    \begin{proof}
        See \cite[Lem.\ 4.5]{AKK2025_modelling}.
    \end{proof}

    \newpage
    \section{$\Gamma$-convergence result}\label{sec:gamma_convergence}

    \hspace{5mm}In this section, we establish the main result of the paper, \textit{i.e.}, the stated $\Gamma(L^2(\mathbb{R}^d))$-convergence (as $\varepsilon\to  \smash{0^+}$) of the family of extended heat loss functionals $\smash{\overline{E}}_\varepsilon^{\mathtt{d}}\colon  \smash{L^2(\mathbb{R}^d)}\to \mathbb{R}\cup\{+\infty\}$, $\varepsilon\in (0,\varepsilon_0)$, for every $v_\varepsilon\in \smash{L^2(\mathbb{R}^d)}$ defined by
    \begin{align}\label{def:E_eps_extended}
       \smash{\overline{E}}_\varepsilon^{\mathtt{d}}(v_\varepsilon)\coloneqq \begin{cases}
            \smash{E}_\varepsilon^{\mathtt{d}}(v_\varepsilon)&\text{ if }v_\varepsilon\in H^1(\Omega_{I}^{\varepsilon})\,,\\[-0.5mm]
            +\infty & \text{ else}\,,
        \end{cases}
    \end{align}
    to the extended limit functional 
    $\smash{\overline{E}}^{\mathtt{d}}\colon \smash{L^2(\mathbb{R}^d)}\to \mathbb{R}\cup\{+\infty\}$,  for every $v\in \smash{L^2(\mathbb{R}^d)}$ defined by 
    \begin{align}\label{def:E_extended}
        \smash{\overline{E}}^{\mathtt{d}}(v)\coloneqq \begin{cases}
            \smash{E}^{\mathtt{d}}(v)& \text{ if }v\in H^1(\Omega)\,,\\[-0.5mm]
            +\infty& \text{ else}\,.
        \end{cases}
    \end{align}

    \begin{theorem}\label{thm:main}
       Let either of the following sufficient cases be satisfied:
       \begin{itemize}[noitemsep,topsep=2pt,leftmargin=!,labelwidth=\widthof{(Case 2)}]
           \item[(Case 1)]\hypertarget{Case 1}{}  $\Gamma_I$ is $C^{1,1}$-regular and $k=n\in (C^{0,1}(\Gamma_I))^d$;
           \item[(Case 2)]\hypertarget{Case 2}{} $\Gamma_I$  is  piece-wise flat   (\textit{i.e.},  there  exist $L\in  \mathbb{N}$  boundary parts  $\smash{\Gamma_{I}^{\ell}}\subseteq\Gamma_I$, $\ell=1,\ldots,L$, with constant outward normal vectors $n_{\ell}\in \mathbb{S}^{d-1}$ such that $\bigcup_{\ell=1}^L{\Gamma_{I}^{\ell}}= \Gamma_I$) and $d\leq 4$.
       \end{itemize}
        Then, if $\mathtt{d}\in C^{0,1}(\Gamma_I)$ with $\mathtt{d}\ge \mathtt{d}_{\textup{min}}$ in $\Gamma_I$, for a constant $\mathtt{d}_{\textup{min}}>0$, there holds \vspace{-0.75mm}
        \begin{align*}
            \Gamma(L^2(\mathbb{R}^d))\text{-}\lim_{\varepsilon\to \smash{0^+}}{\big\{\smash{\overline{E}}_\varepsilon^{\mathtt{d}}\big\}}= \smash{\overline{E}}^{\mathtt{d}}\,,
        \end{align*}
        \textit{i.e.}, the following two statements apply:
        \begin{itemize}[noitemsep,topsep=2pt,leftmargin=!,labelwidth=\widthof{$\bullet$}] 
        \item[$\bullet$] \textit{$\liminf$-estimate.} For every sequence $\smash{(v_\varepsilon)_{\varepsilon\in (0,\varepsilon_0)}}\subseteq \smash{L^2(\mathbb{R}^d)}$ and $v\in \smash{L^2(\mathbb{R}^d)}$, from $v_\varepsilon\to v$ in $\smash{L^2(\mathbb{R}^d)}$ $(\varepsilon\to 0^{+})$,
        it follows that\vspace{-0.75mm}
        \begin{align*}
            \liminf_{\varepsilon\to \smash{0^+}}{\big\{\smash{\overline{E}}_\varepsilon^{\mathtt{d}}(v_\varepsilon)\big\}}\ge \overline{E}(v)\,;
        \end{align*}

        \item[$\bullet$] \textit{$\limsup$-estimate.} For every $v\in \smash{L^2(\mathbb{R}^d)}$, there exists a recovery sequence $\smash{(v_\varepsilon)_{\varepsilon\in (0,\varepsilon_0)}}\subseteq \smash{L^2(\mathbb{R}^d)}$ such that  $v_\varepsilon\to v$ in $\smash{L^2(\mathbb{R}^d)}$ $(\varepsilon\to 0^{+})$ and\vspace{-0.5mm}
        \begin{align*}
         \limsup_{\smash{\varepsilon\to \smash{0^+}}}{\big\{\smash{\overline{E}}_\varepsilon^{\mathtt{d}}(v_\varepsilon)\big\}}\leq \overline{E}(v)\,.
        \end{align*}
    \end{itemize}
    \end{theorem} 

    \begin{remark}[comments on Theorem \ref{thm:main}]
        \begin{itemize}[noitemsep,topsep=2pt,leftmargin=!,labelwidth=\widthof{(iii)}]
            \item[(i)] Case \hyperlink{Case 1}{1} and Case \hyperlink{Case 2}{2} can be considered~together:  A  combination of the proofs of Lemmas \ref{lem:liminf_case1} and  \ref{lem:limsup_case1} with those of Lemmas \ref{lem:liminf_case2} and  \ref{lem:limsup_case2} yields the assertion of Theorem \ref{thm:main} in the mixed case, where the isolated boundary part $\Gamma_I$ consists of $C^{1,1}$-regular and flat segments;
            \item[(ii)] 
            A combination of the proofs of Theorem \ref{thm:main}, \cite[Thm.\ II.2]{AcerbiButtazzo1986}, and \cite[Thm.\ 5.1]{AKK2025_modelling},~should~make it possible to consider insulated boundary parts $\Gamma_I$, which split into relatively open boundary parts $\Gamma_I^{\textup{cond}}\subseteq \Gamma_I$ with conductive heat transfer and $\Gamma_I^{\textup{conv}}\subseteq \Gamma_I$ with convective heat transfer.
        \end{itemize}
    \end{remark}

   Our argument is organized in two parts: first, we establish the $\liminf$-estimate, considering  Case \hyperlink{Case 1}{1} and Case \hyperlink{Case 2}{2} separately; second, we establish the $\limsup$-estimate, again, distinguishing between these two cases.

    \subsection{$\liminf$-estimate}\vspace{-0.5mm}

    \hspace{5mm}In this subsection, we establish the stated  $\liminf$-estimate  in Theorem \ref{thm:main} for  Case~\hyperlink{Case 1}{1}~and~Case~\hyperlink{Case 2}{2}. To begin with, we consider Case \hyperlink{Case 1}{1}, which in the case of pure insulation (\textit{i.e.}, $\Gamma_I=\partial\Omega$) and trivial ambient temperature (\textit{i.e.}, $u_{\infty}=0$) has already been studied in \cite[Thm.\ 3.1]{PietraNitschScalaTrombetti2021}.


    \begin{lemma}[$\liminf$-estimate; Case \hyperlink{Case 1}{1}]\label{lem:liminf_case1}
        Let Case \hyperlink{Case 1}{1} be satisfied. 
        Then, if $\mathtt{d}\in C^{0,1}(\Gamma_I)$ with $\mathtt{d}\ge \mathtt{d}_{\textup{min}}$ in $\Gamma_I$, for a constant $\mathtt{d}_{\textup{min}}>0$, for every sequence $\smash{(v_\varepsilon)_{\varepsilon\in (0,\varepsilon_0)}}\subseteq L^2(\mathbb{R}^d)$ and $v\in L^2(\mathbb{R}^d)$, from $v_\varepsilon\to v$ in $L^2(\mathbb{R}^d)$ $(\varepsilon\to \smash{0^+})$, 
        it follows that\vspace{-0.5mm}
        \begin{align*}
            \liminf_{\varepsilon\to \smash{0^+}}{\big\{\smash{\overline{E}}_\varepsilon^{\mathtt{d}}(v_\varepsilon)\big\}}\ge \overline{E}(v)\,.
        \end{align*}
    \end{lemma}\newpage

    \begin{proof} Let $\smash{(v_\varepsilon)_{\varepsilon\in (0,\varepsilon_0)}}\subseteq L^2(\mathbb{R}^d)$ be a sequence such that $v_\varepsilon\to v$ in $L^2(\mathbb{R}^d)$ $(\varepsilon\to 0^{+})$. Then,\linebreak 
        without loss of generality, we may assume that $\smash{\liminf_{\varepsilon\to \smash{0^+}}{\{\smash{\overline{E}}_\varepsilon^{\mathtt{d}}(v_\varepsilon)\}}<+\infty}$~(\mbox{otherwise},~\mbox{trivially} we have that $\liminf_{\varepsilon\to \smash{0^+}}{\{\smash{\overline{E}}_\varepsilon^{\mathtt{d}}(v_\varepsilon)\}}=+\infty\ge \smash{\overline{E}}^{\mathtt{d}}(v)$).
        As a consequence, we can find  a subsequence $(v_{\varepsilon'})_{\varepsilon'\in (0,\varepsilon_0)}\subseteq L^2(\mathbb{R}^d)$ with $\smash{v_{\varepsilon'}|_{\smash{\Omega^{\varepsilon'}_{I}}}\in H^1(\Omega^{\varepsilon'}_{I})}$ and $v_{\varepsilon'}=u_D$ a.e.\ on $\Gamma_D$ for all $\varepsilon'\in (0,\varepsilon_0)$~such~that\enlargethispage{1mm}\vspace{-0.5mm}
        \begin{align}\label{eq:liminf_case1.1}
            E_{\varepsilon'}^{\mathtt{d}}(v_{\varepsilon'})\to \liminf_{\varepsilon\to \smash{0^+}}{\big\{\smash{\overline{E}}_\varepsilon^{\mathtt{d}}(v_\varepsilon)\big\}}\quad (\varepsilon'\to 0^{+})\,.
        \end{align}
        Due to the equi-coercivity of $\smash{\overline{E}}_\varepsilon^{\mathtt{d}}\colon  L^2(\mathbb{R}^d)\to  \mathbb{R}\cup\{+\infty\}$,  $\varepsilon\in (0,\varepsilon_0)$, (\textit{cf}.\ Lemma \ref{lem:equicoercive}), from \eqref{eq:liminf_case1.1}, it follows that\vspace{-0.5mm}
        \begin{align}\label{eq:liminf_case1.0.1}
            \sup_{\varepsilon'\in (0,\varepsilon_0)}\smash{\big\{\|v_{\varepsilon'}\|_{\Omega}^2+\|\nabla v_{\varepsilon'}\|_{\Omega}^2+\|v_{\varepsilon'}\|_{\Gamma^{\smash{\varepsilon'}}_{I}}^2+\varepsilon'\|\nabla v_{\varepsilon'}\|_{\Sigma^{\smash{\varepsilon'}}_{I}}^2\big\}}<+\infty\,.
        \end{align}
        From \eqref{eq:liminf_case1.0.1}, 
        using the weak continuity of the trace operator from $\smash{H^1(\Omega)}$ to $\smash{H^{\smash{\frac{1}{2}}}(\partial\Omega)}$ (\textit{cf}.~\cite[Thm.\ II.4.3]{Galdi}) and the compact embedding $H^{\smash{\frac{1}{2}}}(\partial\Omega)\hookrightarrow\hookrightarrow L^2(\partial\Omega)$, we deduce that $v|_{\Omega}\in H^1(\Omega)$~and\vspace{-0.5mm}
        \begin{subequations}\label{eq:liminf_case1.2.0}
        \begin{alignat}{3}\label{eq:liminf_case1.2}
          v_{\varepsilon'}&\rightharpoonup  v&&\quad \text{ in }H^1(\Omega)&&\quad(\varepsilon'\to0^{+})\,,\\[-0.35mm]\label{eq:liminf_case1.2.1}v_{\varepsilon'}&\rightharpoonup  v&&\quad \text{ in }H^{\smash{\frac{1}{2}}}(\partial\Omega)&&\quad(\varepsilon'\to0^{+})\,,\\[-0.25mm]\label{eq:liminf_case1.3}
            v_{\varepsilon'}&\to v&&\quad \text{ in }L^2(\partial\Omega)&&\quad(\varepsilon'\to0^{+})\,.
        \end{alignat}
        \end{subequations}
        Since $v_{\varepsilon'}=u_D$ a.e.\ on $\Gamma_D$ for all $\varepsilon'\in (0,\varepsilon_0)$, from  \eqref{eq:liminf_case1.3}, we infer that $v=u_D$~a.e.~on~$\Gamma_D$.
        Moreover, from \eqref{eq:liminf_case1.2.0}, we infer that
        \begin{align}\label{eq:liminf_case1.4}
            \liminf_{\varepsilon'\to \smash{0^+}}{\big\{\tfrac{\lambda}{2}\|\nabla v_{\varepsilon'}\|_{\Omega}^2-(f,v_{\varepsilon'})_{\Omega}-\langle g,v_{\varepsilon'}\rangle_{\smash{H^{\smash{\frac{1}{2}}}(\Gamma_N)}}\big\}}\ge \tfrac{\lambda}{2}\|\nabla v\|_{\Omega}^2-(f,v)_{\Omega}-\langle g,v\rangle_{\smash{H^{\smash{\frac{1}{2}}}(\Gamma_N)}}\,.
        \end{align}
        Using Corollary \ref{lem:poincare_int}\eqref{lem:poincare_1}, Corollary \ref{lem:equiv}, the binomial formula, and the point-wise $\delta$-Young inequality 
        with~$\delta(s)\coloneqq 1+ \beta\mathtt{d}(s)$ for a.e.\ $s\in \Gamma_{I}$,
        we observe that\enlargethispage{1mm}
        \begin{align}\label{eq:liminf_case1.5}
            \begin{aligned}
            &\liminf_{\varepsilon'\to \smash{0^+}}{\big\{\tfrac{\varepsilon'}{2}\|\nabla v_{\varepsilon'}\|_{\Sigma_{I}^{\smash{\varepsilon'}}}^2+\tfrac{\beta}{2}\|v_{\varepsilon'}-u_{\infty}\|_{\smash{\Gamma^{\varepsilon'}_{I}}}^2\big\}}
            \\[-1mm]&\qquad\overset{\eqref{lem:poincare_1}}{\ge } \liminf_{\varepsilon'\to \smash{0^+}}\big\{\tfrac{1}{2}\|\mathtt{d}^{-\smash{\frac{1}{2}}}\{v_{\varepsilon'}(\cdot+\varepsilon' \mathtt{d} n )-v_{\varepsilon'}\}\|_{\Gamma_{I}}^2+\tfrac{\beta}{2}\|\{v_{\varepsilon'}-u_{\infty}\}(\cdot+\varepsilon' \mathtt{d} n )\|_{\Gamma_{I}}^2\big\}
            \\&\qquad\,\overset{\eqref{lem:approx_trans_formula.0}}{=}\liminf_{\varepsilon'\to \smash{0^+}}\big\{\tfrac{1}{2}\|\mathtt{d}^{-\smash{\frac{1}{2}}}\{v_{\varepsilon'}-u_{\infty}\}(\cdot+\varepsilon' \mathtt{d} n )-\mathtt{d}^{-\smash{\frac{1}{2}}}\{v_{\varepsilon'}-u_{\infty}(\cdot+\varepsilon' \mathtt{d} n )\}\|_{\Gamma_{I}}^2\\[-0.75mm]&\qquad\qquad\qquad\quad+\tfrac{1}{2}\|(\beta\mathtt{d})^{\smash{\frac{1}{2}}}\mathtt{d}^{\smash{-\frac{1}{2}}}\{v_{\varepsilon'}-u_{\infty}\}(\cdot+\varepsilon'\mathtt{d}n)\|_{\Gamma_{I}}^2\big\}
            \\&\qquad\;\;= \liminf_{\varepsilon'\to \smash{0^+}}\big\{\tfrac{1}{2}\|(1
            +\beta\mathtt{d})^{\smash{\frac{1}{2}}}\mathtt{d}^{\smash{-\frac{1}{2}}}\{v_{\varepsilon'}-u_{\infty}\}(\cdot+\varepsilon' \mathtt{d} n )\|_{\Gamma_{I}}^2\\[-0.75mm]&\qquad\qquad\qquad\quad
            -(\mathtt{d}^{-1}\{v_{\varepsilon'}-u_{\infty}\}(\cdot+\varepsilon' \mathtt{d} n ),v_{\varepsilon'}-u_{\infty}(\cdot+\varepsilon'\mathtt{d} n ))_{\smash{\Gamma_{I}}}
            \\[-0.25mm]&\qquad\qquad\qquad\quad+\tfrac{1}{2}\|\mathtt{d}^{\smash{-\frac{1}{2}}}\{v_{\varepsilon'}-u_{\infty}(\cdot+\varepsilon' \mathtt{d} n )\}\|_{\Gamma_{I}}^2\big\}
            \\&\qquad\;\;\ge \liminf_{\varepsilon'\to \smash{0^+}}\big\{\tfrac{1}{2}\big\|\{1-\delta+\beta\mathtt{d}\}\mathtt{d}^{\smash{-\frac{1}{2}}}\{v_{\varepsilon'}-u_{\infty}\}(\cdot+\varepsilon' \mathtt{d} n )\|_{\Gamma_{I}}^2
            \\[-0.75mm]&\qquad\qquad\qquad\quad+\tfrac{1}{2}\|\{1-\tfrac{1}{\delta}\}\mathtt{d}^{\smash{-\frac{1}{2}}}\{v_{\varepsilon'}-u_{\infty}(\cdot+\varepsilon' \mathtt{d} n )\}\|_{\Gamma_{I}}^2\big\}
            \\&\qquad\;\;=\liminf_{\varepsilon'\to \smash{0^+}}{\big\{\tfrac{\beta}{2}\|(1+\beta\mathtt{d})^{-\smash{\frac{1}{2}}}\{v_{\varepsilon'}-u_{\infty}(\cdot+\varepsilon' \mathtt{d} n)\}\|_{\Gamma_{I}}^2\big\}}\,.
            \end{aligned}
        \end{align} 
        Next, using \eqref{eq:liminf_case1.3} and that\vspace{-0.5mm}
        \begin{align}\label{eq:liminf_case1.6}
            \smash{u_{\infty}(\cdot+\varepsilon' \mathtt{d} n)\to u_{\infty}\quad \text{ in }L^2(\Gamma_{I})\quad (\varepsilon\to \smash{0^+})}\,,
        \end{align}
        which,  due to Corollary \ref{lem:poincare_int}\eqref{lem:poincare_1}, follows from
        \begin{align*}
            \begin{aligned} 
            \|u_{\infty}(\cdot+\varepsilon' \mathtt{d} n)-u_{\infty}\|_{\Gamma_{I}}^2\leq \smash{\tfrac{\varepsilon' \mathtt{d}_{\min}}{1-\varepsilon' \|\mathtt{d}\|_{\infty,\Gamma_I}\|R_{\varepsilon'}\|_{\smash{\infty,D_{I}^{\varepsilon'}}}}}\|\nabla u_{\infty}\|_{\smash{\Sigma_{I}^{\varepsilon'}}}^2\to 0\quad (\varepsilon\to \smash{0^+})\,,
            \end{aligned}
        \end{align*}
        where the remainders $R_{\varepsilon'}\in  L^\infty(D_{I}^{\varepsilon'})$, $\varepsilon'\in  (0,\varepsilon_0)$,  are as in Lemma \ref{lem:approx_trans_formula}, 
        from \eqref{eq:liminf_case1.5},~we~\mbox{infer}~that
        \begin{align}\label{eq:liminf_case1.7}
            \liminf_{\varepsilon'\to \smash{0^+}}{\big\{\tfrac{\varepsilon'}{2}\|\nabla v_{\varepsilon'}\|_{\Sigma_{I}^{\smash{\varepsilon'}}}^2+\tfrac{\beta}{2}\|v_{\varepsilon'}-u_{\infty}\|_{\smash{\Gamma^{\varepsilon'}_{I}}}^2\big\}}\ge \tfrac{\beta}{2}\|(1+\beta\mathtt{d})^{-\smash{\frac{1}{2}}}\{v-u_{\infty}\}\|_{\Gamma_{I}}^2\,.
        \end{align}
        In summary, from  \eqref{eq:liminf_case1.4} and \eqref{eq:liminf_case1.7},
        we conclude the claimed $\liminf$-estimate for the Case \hyperlink{Case 1}{1}.
    \end{proof}

    Next, let us consider Case \hyperlink{Case 2}{2}.
        
    \begin{lemma}[$\liminf$-estimate; Case \hyperlink{Case 2}{2}]\label{lem:liminf_case2}
        Let Case \hyperlink{Case 2}{2} be satisfied. 
        Then, if $\mathtt{d}\in C^{0,1}(\Gamma_I)$ with $\mathtt{d}\ge \mathtt{d}_{\textup{min}}$ in $\Gamma_I$, for a constant $\mathtt{d}_{\textup{min}}>0$, for every sequence $\smash{(v_\varepsilon)_{\varepsilon\in (0,\varepsilon_0)}}\subseteq L^2(\mathbb{R}^d)$ and $v\in L^2(\mathbb{R}^d)$, from $v_\varepsilon\to v$ in $L^2(\mathbb{R}^d)$ $(\varepsilon\to 0^+)$, 
        it follows that
        \begin{align*}
            \liminf_{\varepsilon\to \smash{0^+}}{\big\{\smash{\overline{E}}_\varepsilon^{\mathtt{d}}(v_\varepsilon)\big\}}\ge \overline{E}(v)\,.
        \end{align*}
    \end{lemma}

    \begin{proof}
        Let $\smash{(v_\varepsilon)_{\varepsilon\in (0,\varepsilon_0)}}\subseteq L^2(\mathbb{R}^d)$ be a sequence such that $v_\varepsilon\to v$ in $L^2(\mathbb{R}^d)$~$(\varepsilon\to 0^{+})$, which,  without loss of generality, satisfies $\smash{\liminf_{\varepsilon\to \smash{0^+}}{\{\smash{\overline{E}}_\varepsilon^{\mathtt{d}}(v_\varepsilon)\}}<+\infty}$. 
        Let $(v_{\varepsilon'})_{\varepsilon'\in (0,\varepsilon_0)}\subseteq L^2(\mathbb{R}^d)$ be a subsequence  with $\smash{v_{\varepsilon'}|_{\smash{\Omega^{\varepsilon'}_{I}}}\in H^1(\Omega^{\varepsilon'}_{I})}$ and $v_{\varepsilon'}=u_D$ a.e.~on~$\Gamma_D$ for all $\varepsilon'\in (0,\varepsilon_0)$ such that 
        \begin{align}\label{eq:liminf_case2.1}
            E_{\varepsilon'}^{\mathtt{d}}(v_{\varepsilon'})\to \liminf_{\varepsilon\to \smash{0^+}}{\big\{\smash{\overline{E}}_\varepsilon^{\mathtt{d}}(v_\varepsilon)\big\}}\quad (\varepsilon'\to 0^{+})\,.
        \end{align}
        Due to the equi-coercivity of $\smash{\overline{E}}_\varepsilon^{\mathtt{d}}\colon  L^2(\mathbb{R}^d)\to  \mathbb{R}\cup\{+\infty\}$,  $\varepsilon\in (0,\varepsilon_0)$, (\textit{cf}.\ Lemma \ref{lem:equicoercive}), from \eqref{eq:liminf_case2.1}, it follows that
        \begin{align}\label{eq:liminf_case2.2}
            \sup_{\varepsilon'\in (0,\varepsilon_0)}\smash{\big\{\|v_{\varepsilon'}\|_{\Omega}^2+\|\nabla v_{\varepsilon'}\|_{\Omega}^2+\|v_{\varepsilon'}\|_{\Gamma^{\smash{\varepsilon'}}_{I}}^2+\varepsilon'\|\nabla v_{\varepsilon'}\|_{\Sigma^{\smash{\varepsilon'}}_{I}}^2\big\}}<+\infty\,.\\[-6mm]\notag
        \end{align}
        From \eqref{eq:liminf_case2.2}, using the weak continuity of the trace operator from $\smash{H^1(\Omega)}$ to $\smash{H^{\smash{\frac{1}{2}}}(\partial\Omega)}$ (\textit{cf}.~\cite[Thm.\ II.4.3]{Galdi}) and the compact embedding $H^{\smash{\frac{1}{2}}}(\partial\Omega)\hookrightarrow \hookrightarrow L^2(\partial\Omega)$, we deduce that $v|_{\Omega}\in H^1(\Omega)$~and  
        \begin{subequations}\label{eq:liminf_case2.3}
        \begin{alignat}{3}\label{eq:liminf_case2.3.1}
          v_{\varepsilon'}&\rightharpoonup  v&&\quad \text{ in }H^1(\Omega)&&\quad(\varepsilon'\to0^{+})\,,\\[-0.35mm]\label{eq:liminf_case2.3.2}v_{\varepsilon'}&\rightharpoonup  v&&\quad \text{ in }H^{\smash{\frac{1}{2}}}(\partial\Omega)&&\quad(\varepsilon'\to0^{+})\,,\\[-0.25mm]\label{eq:liminf_case2.3.3}
            v_{\varepsilon'}&\to v&&\quad \text{ in }L^2(\partial\Omega)&&\quad(\varepsilon'\to0^{+})\,.
        \end{alignat}
        \end{subequations}
        Since $v_{\varepsilon'}=u_D$ a.e.\ on $\Gamma_D$ for all $\varepsilon'\in (0,\varepsilon_0)$, from  \eqref{eq:liminf_case2.3.3}, we infer that $v=u_D$ a.e.\ on $\Gamma_D$.
        Moreover, from \eqref{eq:liminf_case2.3}, we infer that
        \begin{align}\label{eq:liminf_case2.4}
            \liminf_{\varepsilon'\to \smash{0^+}}{\big\{\tfrac{\lambda}{2}\|\nabla v_{\varepsilon'}\|_{\Omega}^2-(f,v_{\varepsilon'})_{\Omega}-\langle g,v_{\varepsilon'}\rangle_{\smash{H^{\smash{\frac{1}{2}}}(\Gamma_N)}}\big\}}\ge \tfrac{\lambda}{2}\|\nabla v\|_{\Omega}^2-(f,v)_{\Omega}-\langle g,v\rangle_{\smash{H^{\smash{\frac{1}{2}}}(\Gamma_N)}}\,.
        \end{align}
        Since $\Gamma_I$ is piece-wise flat, there exists flat boundary parts $\Gamma_{I}^{\ell}\subseteq \Gamma_I$, $\ell=1,\ldots,L$, with constant outward unit normal vectors $n_{\ell}\in \mathbb{S}^{d-1}$ such that $\bigcup_{\ell=1}^L{\Gamma_{I}^{\ell}}= \Gamma_I$.
        Then, for every $\ell=1,\ldots,L$, we introduce the transformation mapping $\phi_{\varepsilon'}^{\ell}\colon \Gamma_{I}^{\ell}\to \mathbb{R}^d$, for every $s\in \Gamma_{I}^{\ell}$ defined~by~(\textit{cf}.~Figure~\ref{fig:phi_ell})
        \begin{align}\label{eq:liminf_case2.5}
        \begin{aligned} 
            \phi_{\varepsilon'}^{\ell}(s)&\coloneqq 
            s+\varepsilon' \mathtt{d}(s)\smash{\{k(s)-(k(s)\cdot n_{\ell})n_{\ell}\}} \,,
            \end{aligned}
        \end{align}
        which, by construction, for every $\ell=1,\ldots,L$ and $\varepsilon'\in (0,\widetilde{\varepsilon}_0)$, where $\widetilde{\varepsilon}_0>0$ is sufficiently~small~and fixed,  is bi-Lipschitz continuous and satisfies\enlargethispage{5mm}
        \begin{align}\label{eq:liminf_case2.6} \smash{\|\mathrm{id}_{\mathbb{R}^d}-\smash{\phi_{\varepsilon'}^{\ell}\|_{\infty,\smash{\Gamma_{I}^{\ell}}}\leq 2\|\mathtt{d}\|_{\infty,\smash{\Gamma_{I}^{\ell}}}\varepsilon'\,.}}
        \end{align}
        
        \begin{figure}[H]
            \centering

  
\tikzset {_wslvnmkzp/.code = {\pgfsetadditionalshadetransform{ \pgftransformshift{\pgfpoint{0 bp } { 0 bp }  }  \pgftransformrotate{-90 }  \pgftransformscale{2 }  }}}
\pgfdeclarehorizontalshading{_mvaa4mi6d}{150bp}{rgb(0bp)=(0.89,0.89,0.89);
rgb(37.5bp)=(0.89,0.89,0.89);
rgb(37.5bp)=(0.86,0.86,0.86);
rgb(37.5bp)=(0.82,0.82,0.82);
rgb(62.5bp)=(1,1,1);
rgb(100bp)=(1,1,1)}

  
\tikzset {_x4dxa1qze/.code = {\pgfsetadditionalshadetransform{ \pgftransformshift{\pgfpoint{0 bp } { 0 bp }  }  \pgftransformrotate{-90 }  \pgftransformscale{2 }  }}}
\pgfdeclarehorizontalshading{_vj61ft4ip}{150bp}{rgb(0bp)=(0.89,0.89,0.89);
rgb(37.5bp)=(0.89,0.89,0.89);
rgb(37.5bp)=(0.86,0.86,0.86);
rgb(37.5bp)=(0.82,0.82,0.82);
rgb(62.5bp)=(1,1,1);
rgb(100bp)=(1,1,1)}
\tikzset{every picture/.style={line width=0.75pt}} 


            \caption{\hspace{-0.15mm}Schematic \hspace{-0.15mm}diagram \hspace{-0.15mm}of \hspace{-0.15mm}the~\hspace{-0.15mm}transformation~\hspace{-0.15mm}\mbox{mapping}~\hspace{-0.15mm}${\phi_{\varepsilon'}^{\ell}\colon\hspace{-0.15em}  \Gamma_I^{\ell}\hspace{-0.15em}\to\hspace{-0.15em} \mathbb{R}^d}$,~\hspace{-0.15mm}${\ell\hspace{-0.15em}=\hspace{-0.15em}1,\ldots,L}$,~\hspace{-0.15mm}(\textit{cf}.~\hspace{-0.15mm}\eqref{eq:liminf_case2.5}).}
            \label{fig:phi_ell}
        \end{figure}
        
        \noindent Due to \eqref{eq:liminf_case2.6},  
        for the \textit{local insulated boundary parts}  
        \begin{align}\label{eq:liminf_case2.7}
            \smash{\Gamma_{I}^{\smash{\varepsilon',\ell}} \coloneqq \Gamma_{I}^{\ell}\cap \phi_{\varepsilon'}^{\ell}(\Gamma_{I}^{\ell})\,,\;\ell=1,\ldots,L\,,}
        \end{align}
         there holds $\smash{\sup_{\varepsilon'\in (0,\varepsilon_0')}{\{\frac{1}{\varepsilon'}\vert \Gamma_{I}^{\ell}\setminus \Gamma_{I}^{\smash{\varepsilon',\ell}}\vert\}}<+\infty}$,  $ \ell=1,\ldots,L $,
        and, thus, up to  
        a subsequence 
        \begin{align}\label{eq:liminf_case2.8}
    \smash{\chi_{\smash{\Gamma_{I}^{\smash{\varepsilon',\ell}}}}\to  1\quad\text{ a.e.\ in }\Gamma_{I}^{\ell}\quad (\varepsilon'\to 0^{+})\,.}
        \end{align} 
        On the other hand, from \eqref{eq:liminf_case2.6}, in turn, for every $\varepsilon'\in (0,\varepsilon_0)$ and $\ell=1,\ldots,L$, we infer that
        \begin{align*}
            \smash{\|\mathrm{id}_{\mathbb{R}^d}-(\phi_{\varepsilon'}^{\ell})^{-1}\|_{\infty,\smash{\phi_{\varepsilon'}^{\ell}(\Gamma_{I}^{\ell})}}=\|\phi_{\varepsilon'}^{\ell}-\mathrm{id}_{\mathbb{R}^d}\|_{\infty,\smash{\Gamma_{I}^{\ell}}}\leq 2\|\mathtt{d}\|_{\infty,\smash{\Gamma_{I}^{\ell}}}\varepsilon'\,,}
        \end{align*}
        which, exploiting that $\widetilde{\mathtt{d}}\in H^{1,\infty}(\Gamma_I^{\ell})$ (since $\mathtt{d}\in H^{1,\infty}(\Gamma_I^{\ell})$, $k\in (H^{1,\infty}(\Gamma_I^{\ell}))^d$, and $n=n_{\ell}$~in~$\Gamma_I^{\ell}$) for all $\ell=1,\ldots,L$ and \eqref{eq:relation_h_h_tilde},  for every $\ell=1,\ldots,L$, abbreviating $\widetilde{\mathtt{d}}_{\varepsilon'}^{\ell}\coloneqq \widetilde{\mathtt{d}}\circ (\phi_{\varepsilon'}^{\ell})^{-1}$, implies that
        \begin{align}\label{eq:liminf_case2.9}
            \smash{\|\widetilde{\mathtt{d}}_{\varepsilon'}^{\ell}- \widetilde{\mathtt{d}}\|_{\infty,\smash{\phi_{\varepsilon'}^{\ell}(\Gamma_{I}^{\ell})}}\leq 2\|\nabla\widetilde{\mathtt{d}}\|_{\infty,\smash{\Gamma_I^{\ell}}}\|\mathtt{d}\|_{\infty,\smash{\Gamma_{I}^{\ell}}}\varepsilon'\,.}
        \end{align}
        Next, for every $\ell=1,\ldots,L$, 
        we define the \textit{local insulating layer}  and \textit{local interacting insulation boundary part} (each in direction of $n_\ell$), respectively, (\textit{cf}.\ Figure~\ref{fig:construction}) 
        \begin{subequations}
        \begin{align}\label{def:sigma_tilde} \widetilde{\Sigma}_{I}^{\smash{\varepsilon',\ell}}&\coloneqq \big\{ \widetilde{s}+t n_{\ell}\mid \widetilde{s}\in \Gamma_{I}^{\smash{\varepsilon',\ell}}\,,\;t\in [0,\varepsilon'\widetilde{\mathtt{d}}_{\varepsilon'}^{\ell}(\widetilde{s}))\big\} \subseteq \Sigma_{I}^{\varepsilon'}\,,\\ 
          \label{def:gamma_tilde}
            \widetilde{\Gamma}_{I}^{\smash{\varepsilon',\ell}}&\coloneqq \big\{ \widetilde{s}+\varepsilon'\widetilde{\mathtt{d}}_{\varepsilon'}^{\ell}(\widetilde{s})n_{\ell}\mid \widetilde{s}\in \Gamma_{I}^{\smash{\varepsilon',\ell}}\big\}\subseteq \Gamma_I^{\varepsilon'}\,,
        \end{align}
        \end{subequations} 
        where the inclusion in \eqref{def:sigma_tilde} results from the bijectivity of the transformation~mappings~\eqref{eq:liminf_case2.5}: if on the contrary $\widetilde{\Sigma}_{I}^{\smash{\varepsilon',\ell}}\not\subseteq \Sigma_{I}^{\varepsilon'}$, 
        there would exist $\widetilde{s}\in \Gamma_{I}^{\varepsilon',\ell}$ such that~the~line~\mbox{segment} $\widetilde{s}+[0,\varepsilon'\widetilde{\mathtt{d}}_{\varepsilon'}^{\ell}(\widetilde{s}))n_{\ell}$ passes (at least) twice through $\Gamma_I^{\varepsilon}$. Then, however, there would exist distinct $\widetilde{s}_i\in \Gamma_{I}^{\varepsilon',\ell}$,~${i=1,2}$, such that $\widetilde{s}\hspace{-0.15em}=\hspace{-0.15em}\phi_{\varepsilon'}^{\ell}(\widetilde{s}_1)\hspace{-0.15em}=\hspace{-0.15em}\phi_{\varepsilon'}^{\ell}(\widetilde{s}_2)$, contradicting the bijectivity~of~the~transformation~\mbox{mappings}~\eqref{eq:liminf_case2.5}.\vspace{-1mm}\enlargethispage{3mm}
        \begin{figure}[H]
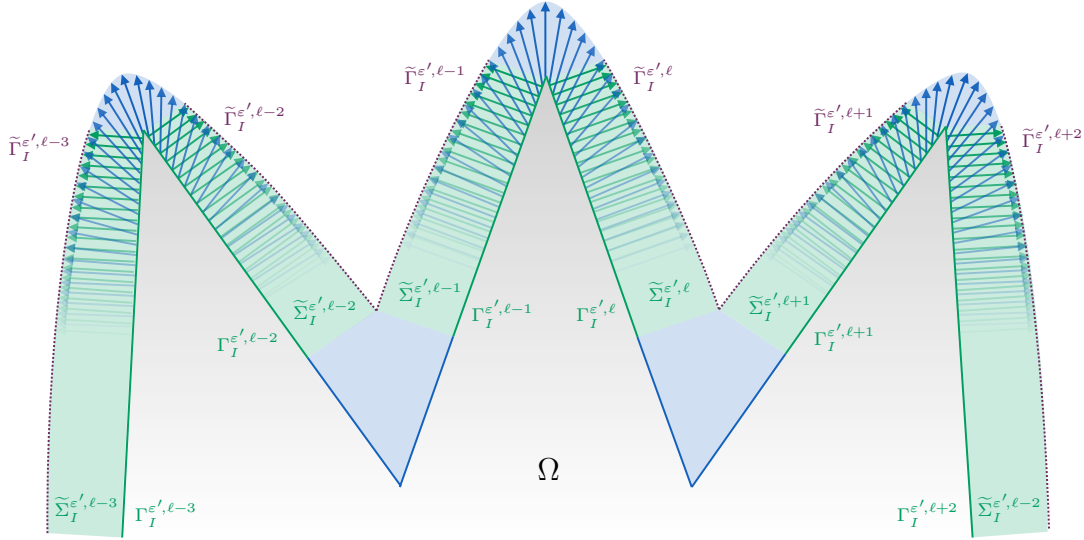

            \centering

  
\tikzset {_a2qpj8h0l/.code = {\pgfsetadditionalshadetransform{ \pgftransformshift{\pgfpoint{0 bp } { 0 bp }  }  \pgftransformrotate{-90 }  \pgftransformscale{2 }  }}}
\pgfdeclarehorizontalshading{_nflpc24eh}{150bp}{rgb(0bp)=(0.89,0.89,0.89);
rgb(37.5bp)=(0.89,0.89,0.89);
rgb(37.5bp)=(0.86,0.86,0.86);
rgb(37.5bp)=(0.82,0.82,0.82);
rgb(62.5bp)=(1,1,1);
rgb(100bp)=(1,1,1)}
\tikzset{every picture/.style={line width=0.75pt}} 


            \caption{Schematic diagram of the construction in the proof of Lemma \ref{lem:liminf_case2}: \textit{(a)} local insulated boundary parts $\Gamma_I^{\varepsilon',\ell}$, $\ell=1,\ldots,L$, (\textit{cf}.\ \eqref{eq:liminf_case2.7}) \textit{(b)}  local insulating layers $\widetilde{\Sigma}_{I}^{\varepsilon',\ell}$, $\ell=1,\ldots,L$,  (\textit{cf}.\ \eqref{def:sigma_tilde});  \textit{(c)} local interacting boundary parts $\widetilde{\Gamma}_I^{\varepsilon',\ell}$, $\ell=1,\ldots,L$, (\textit{cf}.\ \eqref{def:gamma_tilde}).}
            \label{fig:construction}
        \end{figure}
        
        \noindent Resorting to Corollary \ref{lem:poincare_int}\eqref{lem:poincare_1} (with $\Sigma_{I}^{\varepsilon}=\smash{\widetilde{\Sigma}_{I}^{\smash{\varepsilon',\ell}}}$, \textit{i.e.},
        $\Gamma_I=\smash{\Gamma_{I}^{\smash{\varepsilon',\ell}}}$, $\Gamma_I^{\varepsilon'}=\widetilde{\Gamma}_{I}^{\smash{\varepsilon',\ell}}$,  $k=n_{\ell}$, $\mathtt{d}=\widetilde{\mathtt{d}}_{\varepsilon'}^{\ell}$, and ${\varepsilon=\varepsilon'}$), for every  $\ell=1,\ldots,L$, we find that
        \begin{align}\label{eq:liminf_case2.10}
            \begin{aligned} 
           \|(\widetilde{\mathtt{d}}_{\varepsilon'}^{\ell})^{-\smash{\frac{1}{2}}}\{v_{\varepsilon'}(\cdot +\varepsilon' \widetilde{\mathtt{d}}_{\varepsilon'}^{\ell} n_{\ell})-v_{\varepsilon'}\}\|_{\smash{\Gamma_{I}^{\smash{\varepsilon',\ell}}}}^2\leq
           \tfrac{\varepsilon'}{1-\varepsilon'\|\widetilde{\mathtt{d}}\|_{\infty,\Gamma_I}\|\widetilde{R}_{\smash{\varepsilon'}}^{\ell}\|_{\infty,\smash{\widetilde{D}_{I}^{\smash{\varepsilon',\ell}}}}}\|\nabla v_{\varepsilon'}\|_{\widetilde{\Sigma}_{I}^{\smash{\varepsilon',\ell}}}^2\,,
            \end{aligned}
        \end{align}
        where $\widetilde{R}_{\smash{\varepsilon'}}^{\ell}\in L^\infty(\widetilde{D}_{I}^{\smash{\varepsilon',\ell}})$, $\widetilde{D}_{I}^{\smash{\varepsilon',\ell}}\coloneqq \bigcup_{\widetilde{s}\in \smash{\Gamma_{I}^{\smash{\varepsilon',\ell}}}}{\{\widetilde{s}\}\times [0,\varepsilon' \widetilde{\mathtt{d}}_{\varepsilon'}^{\ell}(\widetilde{s}))}$, $\varepsilon'\in (0,\widetilde{\varepsilon}_0)$,  are~as~in~Lemma~\ref{lem:approx_trans_formula}. 
        From Corollary \ref{lem:equiv}, for every $\ell=1,\ldots,L$, we obtain
        \begin{align}\label{eq:liminf_case2.12}
           \|v_{\varepsilon'}-u_\infty\|_{\widetilde{\Gamma}_{I}^{\smash{\varepsilon',\ell}}}^2\ge \{1-(\varepsilon')^{\smash{\frac{1}{2}}}\|\smash{\widetilde{r}_{\smash{\varepsilon'}}^{\ell}}\|_{\infty,\Gamma_I}\}^{-1}\|\{v_{\varepsilon'}-u_\infty\}(\cdot+\varepsilon'\widetilde{\mathtt{d}}_{\varepsilon'}^{\ell}n_{\ell})\|_{\smash{\Gamma_{I}^{\smash{\varepsilon',\ell}}}}^2\,,
        \end{align}
        where the remainders $\smash{\widetilde{r}_{\smash{\varepsilon'}}^{\ell}\in L^\infty(\Gamma_{I}^{\smash{\varepsilon',\ell}})}$, $\varepsilon'\in (0,\widetilde{\varepsilon}_0)$, are as in Lemma \ref{lem:approx_trans_formula2}.\newline
        \hphantom{.}\hspace{5mm}In summary, from \eqref{eq:liminf_case2.10} and \eqref{eq:liminf_case2.12}, we deduce that
        \begin{align}\label{eq:liminf_case2.13}
            \begin{aligned}
            &\liminf_{\varepsilon'\to \smash{0^+}}{\big\{\tfrac{\varepsilon'}{2}\|\nabla v_{\varepsilon'}\|_{\Sigma_{I}^{\smash{\varepsilon'}}}^2+\tfrac{\beta}{2}\|v_{\varepsilon'}-u_{\infty}\|_{\Gamma^{\varepsilon}_{I}}^2\big\}}\\& \ge \liminf_{\varepsilon'\to \smash{0^+}}{\bigg\{\sum_{\ell=1}^{L}{\big\{\tfrac{\varepsilon'}{2}\|\nabla v_{\varepsilon'}\|_{\widetilde{\Sigma}_{I}^{\smash{\varepsilon',\ell}}}^2+\tfrac{\beta}{2}\|v_{\varepsilon'}-u_{\infty}\|_{\widetilde{\Gamma}^{\smash{\varepsilon',\ell}}_{I}}^2\big\}}\bigg\}}
            \\&\ge  \liminf_{\varepsilon'\to \smash{0^+}}\bigg\{\sum_{\ell=1}^{L}\big\{\tfrac{1}{2}\|(\widetilde{\mathtt{d}}_{\varepsilon'}^{\ell})^{\smash{-\frac{1}{2}}}\{ v_{\varepsilon'}(\cdot+\varepsilon'\widetilde{\mathtt{d}}_{\varepsilon'}^{\ell} n_{\ell} )-v_{\varepsilon'}\}\|_{\smash{\Gamma_{I}^{\smash{\varepsilon',\ell}}}}^2\\&\qquad\qquad\qquad\quad+\tfrac{\beta}{2}\|\{v_{\varepsilon'}-u_{\infty}\}(\cdot+\varepsilon'\widetilde{\mathtt{d}}_{\varepsilon'}^{\ell} n_{\ell} )\|_{\smash{\Gamma_{I}^{\smash{\varepsilon',\ell}}}}^2\big\}\bigg\}\,.
            \end{aligned}
        \end{align}
           Similar to \eqref{eq:liminf_case1.5}, for every $\ell=1,\ldots,L$, applying the binomial formula and point-wise Young's inequality with $\delta_{\varepsilon'}^{\ell}(s)\coloneqq 1+\beta \widetilde{\mathtt{d}}_{\varepsilon'}^{\ell}(s)$ for a.e.\ $s\in \Gamma_{I}^{\varepsilon
           ',\ell}$, we find that
        \begin{align}\label{eq:liminf_case2.14}
       \begin{aligned} 
           &\tfrac{1}{2}\|(\widetilde{\mathtt{d}}_{\varepsilon'}^{\ell})^{\smash{-\frac{1}{2}}}\{ v_{\varepsilon'}(\cdot+\varepsilon'\widetilde{\mathtt{d}}_{\varepsilon'}^{\ell} n_{\ell} )-v_{\varepsilon'}\}\|_{\smash{\Gamma_{I}^{\smash{\varepsilon',\ell}}}}^2+\tfrac{\beta}{2}\|\{v_{\varepsilon'}-u_{\infty}\}(\cdot+\varepsilon'\widetilde{\mathtt{d}}_{\varepsilon'}^{\ell}n_{\ell} )\|_{\smash{\Gamma_{I}^{\smash{\varepsilon',\ell}}}}^2
            \\&=\tfrac{1}{2}\|(\widetilde{\mathtt{d}}_{\varepsilon'}^{\ell})^{\smash{-\frac{1}{2}}}\{v_{\varepsilon'}-u_{\infty}\}(\cdot+\varepsilon'\widetilde{\mathtt{d}}_{\varepsilon'}^{\ell} n_{\ell} )-(\widetilde{\mathtt{d}}_{\varepsilon'}^{\ell})^{\smash{-\frac{1}{2}}}\{v_{\varepsilon'}-u_{\infty}(\cdot+\varepsilon'\widetilde{\mathtt{d}}_{\varepsilon'}^{\ell} n_{\ell} )\}\|_{\smash{\Gamma_{I}^{\smash{\varepsilon',\ell}}}}^2\\&\quad+\tfrac{1}{2}\|(\beta\widetilde{\mathtt{d}}_{\varepsilon'}^{\ell})^{\smash{\frac{1}{2}}}(\widetilde{\mathtt{d}}_{\varepsilon'}^{\ell})^{\smash{-\frac{1}{2}}}\{v_{\varepsilon'}-u_{\infty}\}(\cdot+\varepsilon'\widetilde{\mathtt{d}}_{\varepsilon'}^{\ell}n_{\ell})\|_{\smash{\Gamma_{I}^{\smash{\varepsilon',\ell}}}}^2
            \\&\ge\tfrac{1}{2}\|(1
            +\beta\widetilde{\mathtt{d}}_{\varepsilon'}^{\ell})^{\smash{\frac{1}{2}}}(\widetilde{\mathtt{d}}_{\varepsilon'}^{\ell})^{\smash{-\frac{1}{2}}}
            \{v_{\varepsilon'}-u_{\infty}\}(\cdot+\varepsilon'\widetilde{\mathtt{d}}_{\varepsilon'}^{\ell}n_{\ell} )\|_{\smash{\Gamma_{I}^{\smash{\varepsilon',\ell}}}}^2\\&\quad
            -((\widetilde{\mathtt{d}}_{\varepsilon'}^{\ell})^{\smash{-1}}\{v_{\varepsilon'}-u_{\infty}\}(\cdot+\varepsilon'\widetilde{\mathtt{d}}_{\varepsilon'}^{\ell}n_{\ell} ),v_{\varepsilon'}-u_{\infty}(\cdot+\varepsilon'\widetilde{\mathtt{d}}_{\varepsilon'}^{\ell}n_{\ell} ))_{\smash{\Gamma_{I}^{\smash{\varepsilon',\ell}}}}
            \\&\quad+\tfrac{1}{2}\|(\widetilde{\mathtt{d}}_{\varepsilon'}^{\ell})^{\smash{-\frac{1}{2}}}\{v_{\varepsilon'}-u_{\infty}(\cdot+\varepsilon'\widetilde{\mathtt{d}}_{\varepsilon'}^{\ell}n_{\ell}) \}\|_{\smash{\Gamma_{I}^{\smash{\varepsilon',\ell}}}}^2
            \\&\ge \tfrac{1}{2}\|\{1-\delta_{\varepsilon'}^{\ell}+\beta\widetilde{\mathtt{d}}_{\varepsilon'}^{\ell}\}(\widetilde{\mathtt{d}}_{\varepsilon'}^{\ell})^{\smash{-\frac{1}{2}}}\{v_{\varepsilon'}-u_{\infty}\}(\cdot+\varepsilon'\widetilde{\mathtt{d}}_{\varepsilon'}^{\ell}n_{\ell} )\|_{\smash{\Gamma_{I}^{\smash{\varepsilon',\ell}}}}^2
            \\&\quad+\tfrac{1}{2}\|\{1-\tfrac{1}{\delta_{\varepsilon'}^{\ell}}\}(\widetilde{\mathtt{d}}_{\varepsilon'}^{\ell})^{\smash{-\frac{1}{2}}}\{v_{\varepsilon'}-u_{\infty}(\cdot+\varepsilon'\widetilde{\mathtt{d}}_{\varepsilon'}^{\ell}n_{\ell} )\}\|_{\smash{\Gamma_{I}^{\smash{\varepsilon',\ell}}}}^2
            \\&=\tfrac{\beta}{2}\|(1+\beta\widetilde{\mathtt{d}}_{\varepsilon'}^{\ell})^{-\smash{\frac{1}{2}}}\{v_{\varepsilon'}-u_{\infty}(\cdot+\varepsilon'\widetilde{\mathtt{d}}_{\varepsilon'}^{\ell}n_{\ell})\}\|_{\smash{\Gamma_{I}^{\smash{\varepsilon',\ell}}}}^2\,.
            \end{aligned}
        \end{align}
        Next, using that, by \eqref{eq:liminf_case2.3.3} and \eqref{eq:liminf_case2.8}, for every $\ell=1,\ldots,L$, we have that
        \begin{subequations} 
        \begin{alignat*}{3}
            v_{\varepsilon} \chi_{\Gamma_{I}^{\smash{\varepsilon',\ell}}}&\to v &&\quad \text{ in }L^2( \Gamma_{I}^{\ell})&&\quad (\varepsilon'\to 0^+)\,,\\
            u_\infty(\cdot+ \varepsilon'\widetilde{\mathtt{d}}_{\varepsilon'}^{\ell} n_{\ell}) \chi_{\Gamma_{I}^{\smash{\varepsilon',\ell}}}&\to u_\infty &&\quad \text{ in }L^2( \Gamma_{I}^{\ell})&&\quad (\varepsilon'\to 0^+)\,,
        \end{alignat*}
        \end{subequations}
        which, using Corollary \ref{lem:poincare_int}\eqref{lem:poincare_1} and that $\mathtt{d}_{\min}\leq \widetilde{\mathtt{d}}_{\varepsilon'}^{\ell}\leq \|\mathtt{d}\|_{\infty,\smash{\Gamma_{I}}}$ a.e.\ on $\Gamma_{I}^{\smash{\varepsilon',\ell}}$, for every $\ell=1,\ldots,L$, follows from
        \begin{align*}
            \|u_\infty(\cdot+ \varepsilon'\widetilde{\mathtt{d}}_{\varepsilon'}^{\ell} n_{\ell})-u_\infty\|_{\smash{\Gamma_{I}^{\smash{\varepsilon',\ell}}}}^2\leq  \tfrac{\varepsilon'\mathtt{d}_{\min}}{1-\varepsilon'\|\mathtt{d}\|_{\infty,\smash{\Gamma_{I}}}\|\widetilde{R}_{\smash{\varepsilon'}}^{\ell}\|_{\infty,\smash{\widetilde{D}_{I}^{\smash{\varepsilon',\ell}}}}}\|\nabla u_\infty\|_{\smash{\widetilde{\Sigma}_{I}^{\smash{\varepsilon',\ell}}}}\to 0\quad (\varepsilon'\to \smash{0^+})\,,
        \end{align*}
        for every $\ell=1,\ldots,L$, together with \eqref{eq:liminf_case2.9}, we deduce that
        \begin{align}
            \begin{aligned}
                 &\tfrac{\beta}{2}\|(1+\beta\widetilde{\mathtt{d}}_{\varepsilon'}^{\ell})^{-\smash{\frac{1}{2}}}\{v_{\varepsilon'}-u_{\infty}(\cdot+\varepsilon'\widetilde{\mathtt{d}}_{\varepsilon'}^{\ell} n_\ell)\}\|_{\smash{\Gamma_{I}^{\smash{\varepsilon',\ell}}}}^2
            \\&\to \tfrac{\beta}{2}\|(1+\beta\widetilde{\mathtt{d}})^{-\smash{\frac{1}{2}}}\{v_{\varepsilon'}-u_{\infty}\}\|_{\smash{\Gamma_{I}^{\ell}}}^2\quad (\varepsilon'\to \smash{0^+})\,.
            \end{aligned}\label{eq:liminf_case2.15}
        \end{align}
        Using \eqref{eq:liminf_case2.14} together with \eqref{eq:liminf_case2.15} in \eqref{eq:liminf_case2.13}, we find that\enlargethispage{1.5mm}
        \begin{align}\label{eq:liminf_case2.16}
            \liminf_{\varepsilon'\to \smash{0^+}}{\big\{\tfrac{\varepsilon'}{2}\|\nabla v_{\varepsilon'}\|_{\Sigma_{I}^{\smash{\varepsilon'}}}^2+\tfrac{\beta}{2}\|v_{\varepsilon'}-u_{\infty}\|_{\Gamma^{\varepsilon}_{I}}^2\big\}}&\ge
            \tfrac{\beta}{2}\|(1+\beta\widetilde{\mathtt{d}})^{-\smash{\frac{1}{2}}}\{v_{\varepsilon'}-u_{\infty}\}\|_{\smash{\Gamma_{I}}}^2\,.
        \end{align}  
        In summary, from \eqref{eq:liminf_case2.4}  and \eqref{eq:liminf_case2.16}, we conclude the claimed $\liminf$-estimate in the Case \hyperlink{Case 2}{2}.
    \end{proof}

    \subsection{$\limsup$-estimate}

    \hspace{5mm}In this subsection, similar to the previous subsection, we establish the stated $\limsup$-estimate in Theorem \ref{thm:main} for Case~\hyperlink{Case 1}{1} and Case \hyperlink{Case 2}{2}. To begin with, we consider Case \hyperlink{Case 1}{1}, which in~the~case~of~pure insulation (\textit{i.e.}, $\Gamma_I=\partial\Omega$) and trivial ambient temperature (\textit{i.e.}, $u_{\infty}=0$) has already been studied in \cite[Thm.\ 3.1]{PietraNitschScalaTrombetti2021}.\enlargethispage{1mm}
    
    \begin{lemma}[$\limsup$-estimate; Case \hyperlink{Case 1}{1}]\label{lem:limsup_case1}
        Let Case \hyperlink{Case 1}{1} be satisfied. Then, if   $\mathtt{d}\in  C^{0,1}(\Gamma_I)$ is such that $\mathtt{d}\ge  \mathtt{d}_{\min}$ in $\Gamma_I$, for a constant $\mathtt{d}_{\min}>0$, then for every $v\in  L^2(\mathbb{R}^d)$, there exists a recovery sequence $\smash{(v_\varepsilon)_{\varepsilon\in (0,\varepsilon_0)}}\subseteq  L^2(\mathbb{R}^d)$ such that $v_\varepsilon\to v$ in $L^2(\mathbb{R}^d)$ $(\varepsilon\to 0^{+})$~and
        \begin{align*}
            \limsup_{\smash{\varepsilon\to \smash{0^+}}}{\smash{\big\{\smash{\overline{E}}_\varepsilon^{\mathtt{d}}(v_\varepsilon)\big\}}}\leq \smash{\overline{E}}^{\mathtt{d}}(v)\,.
        \end{align*}
    \end{lemma}

    \begin{proof}
         Let $v\in L^2(\mathbb{R}^d)$ be fixed, but arbitrary. Without loss of generality,~we~may~assume~that $v|_{\Omega}\in H^1(\Omega)$ with $v=u_D$ a.e.\ on $\Gamma_D$. Otherwise, we can~choose~${v_\varepsilon=v\in L^2(\mathbb{R}^d)}$ for~all~${\varepsilon\in (0,\varepsilon_0)}$, which satisfies $\limsup_{\varepsilon\to \smash{0^+}}{\{\smash{\overline{E}}_{\varepsilon}^{\mathtt{d}}(v_{\varepsilon})\}}=+\infty
=
    \smash{\overline{E}}^{\mathtt{d}}(v)$. As a consequence, there exists an extension $\overline{v}\in H^1(\mathbb{R}^d)$ of the restriction $v|_{\Omega}\in H^1(\Omega)$, \textit{i.e.},  we have that $\overline{v}|_{\Omega}=v|_{\Omega}$ a.e.\ in $\Omega$. Next,~let~$\varepsilon\in (0,\varepsilon_0)$ be fixed, but arbitrary. In order to construct the desired recovery sequence, we modify the extension $\overline{v}\in H^1(\mathbb{R}^d)$ by means of  the cut-off function $\varphi_\varepsilon\colon \mathbb{R}^d\to [0,1]$, for every $x\in \mathbb{R}^d$~defined~by
        \begin{align}\label{eq:limsup_case1.1}
            \varphi_\varepsilon(x)\coloneqq\begin{cases}
                1-\smash{\frac{\beta \textup{dist}(x,\partial\Omega)}{\varepsilon(1+\beta \mathtt{d}(x))}}&\text{ if }x\in \Sigma^{\varepsilon}_{I}\,,\\
                1&\text{ if }x\in \overline{\Omega}\,,\\
                0&\text{ else}\,,
            \end{cases}
        \end{align}
        where $\mathtt{d}\colon\hspace{-0.1em} \Sigma^{\varepsilon_0}_{I}\hspace{-0.1em}\to\hspace{-0.1em}  (0,+\infty)$ is a not relabelled extension of 
        $\mathtt{d}\in  C^{0,1}(\Gamma_I)$, for every ${x\hspace{-0.1em}=\hspace{-0.1em} s+t n(s)\hspace{-0.1em}\in  \hspace{-0.1em}\Sigma^{\varepsilon_0}_{I}}$, where $s\hspace{-0.1em}\in\hspace{-0.1em} \Gamma_I$ and $t\hspace{-0.1em}\in \hspace{-0.1em}[0,\varepsilon_0 \mathtt{d}(s))$, defined by $\mathtt{d}(x)\hspace{-0.1em}\coloneqq\hspace{-0.1em} \mathtt{d}(s)$, which, in turn, also satisfies~${\mathtt{d}\hspace{-0.1em}\in\hspace{-0.1em} C^{0,1}(\Sigma^{\varepsilon_0}_{I})}$. 
        By construction, 
        the cut-off function  \eqref{eq:limsup_case1.1}
        satisfies $\varphi_\varepsilon|_{\Omega^{\varepsilon}_{I}}\in H^{1,\infty}(\Omega^{\varepsilon}_{I})$ with
        \begin{subequations}\label{eq:limsup_case1.2}
        \begin{alignat}{2}\label{eq:limsup_case1.2.1}
            0\leq \varphi_\varepsilon&\leq 1 &&\quad\text{ in }\mathbb{R}^d\,,\\
            \varphi_\varepsilon&=1&&\quad\text{ in }\overline{\Omega}\,,\label{eq:limsup_case1.2.2}\\
       \smash{\varphi_{\varepsilon}} &= \smash{\tfrac{1}{1+\beta \mathtt{d}}}&&\quad\text{ on }\Gamma_{I}^{\varepsilon}\,.\label{eq:limsup_case1.2.3}
        \end{alignat}
        \end{subequations}
        Moreover, using that $\nabla \textup{dist}(\cdot,\partial\Omega)=n\circ \pi_{\partial\Omega}$  in $\mathbb{R}^d\setminus (\mathrm{Med}(\partial\Omega)\cup\overline{\Omega})$ (\textit{cf}.\ \eqref{eq:grad_dist}) and $\vert \mathrm{Med}(\partial\Omega)\vert=0$,  we have that
        \begin{align}\label{eq:limsup_case1.3}
          \left. \begin{aligned}
                \nabla \varphi_{\varepsilon}&=-\tfrac{\beta}{\varepsilon(1+\beta\mathtt{d})^2}\{(1+\beta\mathtt{d})n\circ \pi_{\partial\Omega}-\textup{dist}(\cdot,\partial\Omega)\beta\nabla \mathtt{d}\}
                \\[-1mm]&=-\tfrac{\beta}{\varepsilon(1+\beta\mathtt{d})}n\circ\pi_{\partial\Omega}+\tfrac{\beta^2\textup{dist}(\cdot,\partial\Omega)}{\varepsilon(1+\beta\mathtt{d})^2}\nabla \mathtt{d}
           \end{aligned}\quad\right\} \quad\text{ a.e.\ in }\Sigma_{I}^{\varepsilon}\,,
        \end{align}
        so that, due to $\smash{\textup{dist}(\cdot,\partial\Omega)\leq \varepsilon\|\mathtt{d}\|_{\infty,\Gamma_I}}$ in $\smash{\Sigma_{I}^{\varepsilon}}$,
        \begin{align*} 
              \smash{ \vert \nabla \varphi_{\varepsilon}\vert\leq\tfrac{\beta}{\varepsilon(1+\beta\mathtt{d})}+\tfrac{\beta^2\|\mathtt{d}\|_{\infty,\Gamma_{I}}}{(1+\beta \mathtt{d}_{\min})^2}\|\nabla \mathtt{d}\|_{\infty,\Gamma_I}
           \quad\text{ a.e.\ in }\Sigma_{I}^{\varepsilon}\,,}
        \end{align*}
        and, thus, by the convexity of the function $(t\mapsto t^2)\colon \mathbb{R}\to \mathbb{R}$, for fixed, but arbitrary $\delta\in (0,1)$,
        \begin{align}\label{eq:limsup_case1.4}
               \smash{\vert \nabla \varphi_{\varepsilon}\vert^2\leq\tfrac{1}{\delta}\tfrac{\beta^2}{\varepsilon^2(1+\beta\mathtt{d})^2}+\tfrac{1}{1-\delta}\beta^4\|\mathtt{d}\|_{\infty,\Gamma_{I}}^2\|\nabla \mathtt{d}\|_{\infty,\Gamma_I}^2
           \quad\text{ a.e.\ in }\Sigma_{I}^{\varepsilon}\,.}
        \end{align}
       Then, let the  desired recovery sequence 
        $v_\varepsilon\in L^2(\mathbb{R}^d)$,  for a.e.\ $x\in \mathbb{R}^d$,  be defined by
        \begin{align*}
             v_\varepsilon(x)
             \coloneqq\begin{cases}
                \overline{v}(x)\varphi_\varepsilon(x)+u_{\infty}(x) (1-\varphi_\varepsilon(x))&\text{ if }x\in \Omega^{\varepsilon}_{I}\,,\\
                v(x)&\text{ else}\,,
            \end{cases}
        \end{align*}
        which, by construction and $\varphi_\varepsilon|_{\Omega^{\varepsilon}_{I}}\in H^{1,\infty}(\Omega^{\varepsilon}_{I})$ with \eqref{eq:limsup_case1.2.2},\eqref{eq:limsup_case1.2.3}, satisfies $v_\varepsilon|_{\Omega^{\varepsilon}_{I}}\in H^1(\Omega^{\varepsilon}_{I})$~with\vspace{-4.5mm}
        \begin{subequations}\label{eq:limsup_case1.6}
        \begin{alignat}{2}  
            v_\varepsilon&=v&&\quad\text{ a.e.\ in }\mathbb{R}^d\setminus \Sigma_{I}^{\varepsilon} \,,\label{eq:limsup_case1.6.1}\\ 
            v_\varepsilon&=u_D&&\quad\text{ a.e.\ on }\smash{\Gamma_D} \,,\label{eq:limsup_case1.6.2}\\
             v_\varepsilon-u_{\infty}&=\smash{\tfrac{1}{1+\beta\mathtt{d}}}\{\overline{v}-u_{\infty}\}&&\quad\text{ a.e.\ on }\smash{\Gamma_{I}^{\varepsilon} }\,.\label{eq:limsup_case1.6.3}
        \end{alignat}
        \end{subequations}
        Moreover,  using \eqref{eq:limsup_case1.4} and  the convexity of $(t\mapsto t^2)\colon \mathbb{R}\to \mathbb{R}$, for fixed, but arbitrary $\delta\in (0,1)$, we have that
        \begin{align}\label{eq:limsup_case1.7}
        \left.
        \begin{aligned} 
            \vert \nabla v_{\varepsilon}\vert^2&\leq \tfrac{1}{\delta}\vert \nabla \varphi_{\varepsilon}(\overline{v}-u_{\infty})\vert^2+\tfrac{1}{1-\delta}\vert \varphi_{\varepsilon}\nabla \overline{v}+(1-\varphi_{\varepsilon})\nabla u_{\infty}\vert^2
            \\&\leq \tfrac{1}{\delta}\big\{\tfrac{1}{\delta}\tfrac{\beta^2}{\varepsilon^2(1+\beta\mathtt{d})^2}+\tfrac{1}{1-\delta}\beta^4\|\mathtt{d}\|_{\infty,\Gamma_{I}}^2\|\nabla \mathtt{d}\|^2_{\infty,\Gamma_I}\big\}\vert\overline{v}-u_{\infty}\vert^2\\&\quad+\tfrac{1}{1-\delta}\{\vert\nabla \overline{v}\vert+\vert\nabla u_{\infty}\vert\}^2
            \end{aligned}\quad\right\}\quad\text{ a.e.\ in }\Sigma_{I}^{\varepsilon}\,.
        \end{align}
        In particular, due to \eqref{eq:limsup_case1.6.1}, $\vert \Sigma_{I}^{\varepsilon}\vert\to 0$ $(\varepsilon\to 0^+)$, and $\vert v_{\varepsilon}\vert \leq \vert v\vert +\vert u_{\infty}\vert$ a.e.\ in $\mathbb{R}^d$ (due to \eqref{eq:limsup_case1.2.1}), Lebesgue's dominated convergence theorem yields that 
        \begin{align*}
            v_\varepsilon\to v\quad \text{ in }L^2(\mathbb{R}^d)\quad (\varepsilon\to\smash{0^+})\,.
        \end{align*}
        In addition, as a direct consequence of \eqref{eq:limsup_case1.6.1},\eqref{eq:limsup_case1.6.2}, we obtain
        \begin{align}\label{eq:limsup_case1.8}
             \smash{\overline{E}}^{\mathtt{d}}_\varepsilon(v_\varepsilon)= \tfrac{\lambda}{2}\|\nabla v\|_{\Omega}^2-(f,v)_{\Omega}-\langle g,v\rangle_{\smash{H^{\smash{\frac{1}{2}}}(\Gamma_N)}}+\tfrac{\varepsilon}{2}\|\nabla v_{\varepsilon}\|_{\smash{\Sigma^{\varepsilon}_{I}}}^2+\tfrac{\beta}{2}\|v_\varepsilon-u_{\infty}\|_{\Gamma^{\varepsilon}_{I}}^2\,,
        \end{align}
        so that it is left to treat the limit superior of the last two terms on the right-hand side~of~\eqref{eq:limsup_case1.8}. For the latter, it is sufficient establish that  
        \begin{subequations}\label{eq:limsup_case1.9}
        \begin{align}\label{eq:limsup_case1.9.1}
             \limsup_{\varepsilon\to \smash{\smash{0^+}}}{\big\{\tfrac{\varepsilon}{2}\|\nabla v_{\varepsilon}\|_{\smash{\Sigma^{\varepsilon}_{I}}}^2\big\}}&\leq \tfrac{\beta}{2}\|(\beta \mathtt{d})^{\smash{\frac{1}{2}}}(1+\beta \mathtt{d})^{-1}\{v-u_{\infty}\}\|_{\Gamma_{I}}^2\,,\\
             \limsup_{\varepsilon\to \smash{0^+}}{\big\{\tfrac{\beta}{2}\|v_{\varepsilon}-u_{\infty}\|_{\Gamma^{\varepsilon}_{I}}^2\big\}}&\leq \tfrac{\beta}{2}\|(1+\beta \mathtt{d})^{-1}\{v-u_{\infty}\}\|_{\Gamma_{I}}^2\,,\label{eq:limsup_case1.9.2}
        \end{align}
        \end{subequations} 
        which jointly imply that
        \begin{align}\label{eq:limsup_case1.10}
             \limsup_{\varepsilon\to \smash{0^+}}{\big\{\tfrac{\varepsilon}{2}\|\nabla v_{\varepsilon}\|_{\smash{\Sigma^{\varepsilon}_{I}}}^2+\tfrac{\beta}{2}\|v_{\varepsilon}-u_{\infty}\|_{\Gamma^{\varepsilon}_{I}}^2\big\}}\leq \tfrac{\beta}{2}\|(1+\beta \mathtt{d})^{-\smash{\frac{1}{2}}}\{v-u_{\infty}\}\|_{\Gamma_{I}}^2\,.
        \end{align}

        Therefore, let us next establish the $\limsup$-estimates \eqref{eq:limsup_case1.9.1} and \eqref{eq:limsup_case1.9.2} separately:

        \emph{ad \eqref{eq:limsup_case1.9.1}.} Resorting to \eqref{eq:limsup_case1.7}, 
            Lemma \ref{lem:Lebesgue_boundary_limit}\eqref{lem:Lebesgue_boundary_limit.0} (with $k=n$ and, thus, $\widetilde{\mathtt{d}}=\mathtt{d}$), and \eqref{eq:limsup_case1.6.1}, because $\delta\in (0,1)$ was chosen arbitrarily,  we find that
            \begin{align*}
            \begin{aligned} 
             \limsup_{\varepsilon\to \smash{0^+}}{\big\{\tfrac{\varepsilon}{2}\|\nabla v_{\varepsilon}\|_{\smash{\Sigma^{\varepsilon}_{I}}}^2\big\}}&\smash{\overset{\eqref{eq:limsup_case1.7}}{\leq}}  \limsup_{\varepsilon\to \smash{0^+}}{\big\{\tfrac{1}{2\varepsilon}\tfrac{1}{\delta^2}\|\beta(1+\beta\mathtt{d})^{-1}\{\overline{v}-u_{\infty}\}\|_{\smash{\Sigma^{\varepsilon}_{I}}}^2\big\}}\\&\quad\hspace*{1.5mm}+
             \limsup_{\varepsilon\to \smash{0^+}}{\big\{\tfrac{\varepsilon}{2}\tfrac{1}{\delta(1-\delta)}\beta^4\|\mathtt{d}\|_{\infty,\Gamma_{I}}^2\|\nabla \mathtt{d}\|_{\infty,\Gamma_{I}}^2\|\overline{v}-u_{\infty}\|_{\smash{\Sigma^{\varepsilon}_{I}}}^2\big\}}
             \\&\quad\hspace*{1.5mm}+\limsup_{\varepsilon\to \smash{0^+}}{\big\{\tfrac{\varepsilon}{2}\tfrac{1}{1-\delta}\{\|\nabla \overline{v} \|_{\smash{\Sigma^{\varepsilon}_{I}}}+\|\nabla u_\infty\|_{\smash{\Sigma^{\varepsilon}_{I}}}\}^2\big\}} 
             \\&\hspace*{1.5mm}\leq \limsup_{\varepsilon\to \smash{0^+}}{\big\{\tfrac{1}{\delta^2}\tfrac{1}{2\varepsilon}\|\beta(1+\beta \mathtt{d})^{-1}\{ \overline{v}-u_{\infty}\}\|_{\smash{\Sigma^{\varepsilon}_{I}}}^2\big\}}
             \\&\hspace*{1.25mm}\smash{\overset{\eqref{lem:Lebesgue_boundary_limit.0}}{=}}\tfrac{1}{\delta^2}\tfrac{\beta}{2}\|(\beta\mathtt{d})^{\smash{\frac{1}{2}}}(1+\beta \mathtt{d})^{-1}\{v-u_{\infty}\}\|_{\Gamma_{I}}^2
             \\&\hspace*{1.5mm}\to \tfrac{\beta}{2}\|(\beta \mathtt{d})^{\smash{\frac{1}{2}}}(1+\beta \mathtt{d})^{-1}\{v-u_{\infty}\}\|_{\Gamma_{I}}^2\quad (\delta \to 1^-)\,.
             \end{aligned} 
        \end{align*} 

        \emph{ad \eqref{eq:limsup_case1.9.2}.} Using \eqref{eq:limsup_case1.6.3}, the approximative transformation formula (\textit{cf}.\ Lemma \ref{lem:approx_trans_formula2}),~and~that
         \begin{align*}
             \{\overline{v}-u_{\infty}\}(\cdot+\varepsilon \mathtt{d} n)\to \overline{v}-u_{\infty}=v-u_{\infty}\quad \text{ in }L^2(\Gamma_I)\quad (\varepsilon\to \smash{0^+})\,,
         \end{align*}
         which, similar to \eqref{eq:liminf_case1.6}, using Corollary \ref{lem:poincare_int}\eqref{lem:poincare_1}, follows from
        \begin{align*}
            \|\{\overline{v}-u_{\infty}\}(\cdot+\varepsilon \mathtt{d} n)-\{v-u_{\infty}\}\|_{\smash{\Gamma_{I}}}^2\leq  \tfrac{\varepsilon\mathtt{d}_{\min}}{1-\varepsilon\|\mathtt{d}\|_{\infty,\smash{\Gamma_{I}}}\|R_{\varepsilon}\|_{\infty,\smash{D_{I}^{\varepsilon}}}}\|\nabla\{ \overline{v}-u_{\infty}\}\|_{\smash{\Sigma_{I}^{\varepsilon}}}^2\to 0\quad (\varepsilon\to \smash{0^+})\,,
        \end{align*}
        we find that
        \begin{align*}
            \begin{aligned} 
            \lim_{\varepsilon\to \smash{0^+}}{\big\{\tfrac{\beta}{2}\|v_{\varepsilon}-u_{\infty}\|_{\Gamma^{\varepsilon}_{I}}^2\big\}}&= \lim_{\varepsilon\to \smash{0^+}}{\big\{\tfrac{\beta}{2}\|(1+\beta \mathtt{d})^{-1}\{\overline{v}-u_{\infty}\}(\cdot+\varepsilon \mathtt{d} n)\}\|_{\Gamma_{I}}^2\big\}}
            \\&=\tfrac{\beta}{2}\|(1+\beta \mathtt{d})^{-1}\{v-u_{\infty}\}\|_{\Gamma_{I}}^2\,.
            \end{aligned}
        \end{align*} 
        In summary, from \eqref{eq:limsup_case1.9.1} and \eqref{eq:limsup_case1.9.2}, it follows  \eqref{eq:limsup_case1.10}, which together with \eqref{eq:limsup_case1.8} confirms the claimed $\limsup$-estimate for the Case \hyperlink{Case 1}{1}. 
    \end{proof}

    Next, let us consider Case \hyperlink{Case 2}{2}.

    \begin{lemma}[$\limsup$-estimate; Case \hyperlink{Case 2}{2}]\label{lem:limsup_case2}
        Let Case \hyperlink{Case 2}{2} be satisfied. Then, if   $\mathtt{d}\in  C^{0,1}(\Gamma_I)$ is such~that $\mathtt{d}\ge  \mathtt{d}_{\min}$, for a constant  $\mathtt{d}_{\min}>0$, then for every $v\in  L^2(\mathbb{R}^d)$, there exists a recovery sequence $\smash{(v_\varepsilon)_{\varepsilon\in (0,\varepsilon_0)}}\subseteq L^2(\mathbb{R}^d)$ such that $v_\varepsilon\to v$ in $L^2(\mathbb{R}^d)$ $(\varepsilon\to 0^{+})$ and
        \begin{align*}
            \limsup_{\smash{\varepsilon\to \smash{0^+}}}{\smash{\big\{\smash{\overline{E}}_\varepsilon^{\mathtt{d}}(v_\varepsilon)\big\}}}\leq \smash{\overline{E}}^{\mathtt{d}}(v)\,.
        \end{align*}
    \end{lemma}

    \begin{proof}
         Let \hspace{-0.1mm}$v\hspace{-0.15em}\in \hspace{-0.15em} L^2(\mathbb{R}^d)$ \hspace{-0.1mm}be \hspace{-0.1mm}fixed, \hspace{-0.1mm}but \hspace{-0.1mm}arbitrary. \hspace{-0.1mm}Again, \hspace{-0.1mm}without \hspace{-0.1mm}loss \hspace{-0.1mm}of \hspace{-0.1mm}generality, \hspace{-0.1mm}we~\hspace{-0.1mm}may~\hspace{-0.1mm}\mbox{assume}~\hspace{-0.1mm}that $v|_{\Omega}\hspace{-0.1em}\in\hspace{-0.1em} H^1(\Omega)$ with $v\hspace{-0.1em}=\hspace{-0.1em}u_D$ a.e.\ on $\Gamma_D$, so that 
          there exists an extension $\overline{v}\hspace{-0.1em}\in \hspace{-0.1em}H^1(\mathbb{R}^d)$~of~the~\mbox{restriction} $v|_{\Omega}\in H^1(\Omega)$, \textit{i.e.}, we have that $\overline{v}|_{\Omega}=v|_{\Omega}$ a.e.\ in $\Omega$. Next,~let~$\varepsilon\in (0,\varepsilon_0)$ be fixed, but arbitrary.\linebreak The construction of the desired recovery sequence, again, relies on the construction of an appropriate cut-off function $\varphi_{\varepsilon}\colon\hspace{-0.1em} \mathbb{R}^d\hspace{-0.1em}\to\hspace{-0.1em} [0,1]$, which, in this case, is more delicate than in Case \hyperlink{Case 1}{1} and requires the smooth approximation of the piece-wise constant outward~unit~\mbox{normal}~\mbox{vector}~field~${n\colon\hspace{-0.15em} \Gamma_I\hspace{-0.15em}\to\hspace{-0.15em} \mathbb{S}^{d-1}}$.\linebreak As the latter is not defined in all of $\mathbb{R}^d$, motivated by 
          $\nabla \smash{\widehat{\mathrm{dist}}}(\cdot,\partial\Omega)=n\circ \pi_{\partial\Omega}$ in $\mathbb{R}^d\setminus (\mathrm{Med}(\partial\Omega)\cup \partial\Omega)$ (\textit{cf}.~\eqref{eq:grad_dist_hat}), we construct a smooth approximation by taking the gradient of the  mollified signed distance function \eqref{def:dist_hat}.\enlargethispage{2.5mm}
          
          More precisely, let the \emph{mollified outward unit normal vector field} $n_\varepsilon\colon \hspace{-0.1em}\mathbb{R}^d\hspace{-0.1em}\to\hspace{-0.1em} \mathbb{R}^d$,~for~every~$x\hspace{-0.1em}\in\hspace{-0.1em} \mathbb{R}^d$, be defined by
    \begin{align}\label{eq:limsup_case2.1}
        n_\varepsilon(x)\coloneqq \nabla (\omega_\varepsilon\ast \widehat{\textrm{dist}}(\cdot,\partial\Omega))(x)\coloneqq\int_{B_\varepsilon^d(x)}{\omega_\varepsilon(x-y)\nabla \widehat{\textrm{dist}}(y,\partial\Omega)\,\mathrm{d}y}\,,
    \end{align}
    where 
	$(\omega_\varepsilon)_{\varepsilon\in (0,\varepsilon_0)}\subseteq
	C^\infty_0(\mathbb{R}^d)$ is a family of Friedrichs mollifiers, for every
	$\varepsilon\in  (0,\varepsilon_0)$~and~${x\in  \mathbb{R}^d}$, defined by
	$\omega_\varepsilon(x)\coloneqq \varepsilon^{-d}\omega(\varepsilon^{-1}x)$, where $\omega\in C^\infty_c(\mathbb{R}^d)$ is a radially symmetric mollification kernel  such that $\omega\ge 0$ in $\mathbb{R}^d$, $\textup{supp}\,\omega\subseteq B_1^d(0)$, and $\|\omega\|_{1,\mathbb{R}^d} =1$.\enlargethispage{3mm}

    By means of the mollified outward unit normal vector field \eqref{eq:limsup_case2.1}, 
    denoting by $k\in (C^{0,1}(\Sigma^{\varepsilon_0}_{I}))^d$ and $\mathtt{d}\in C^{0,1}(\Sigma^{\varepsilon_0}_{I})$ the not relabelled extensions of $k\in (C^{0,1}(\Gamma_{I}))^d$ and $\mathtt{d}\in C^{0,1}(\Gamma_{I})$, respectively,  for every $x\hspace{-0.15em}=\hspace{-0.15em}  s+t n(s)\hspace{-0.175em}\in \hspace{-0.175em} \Sigma^{\varepsilon_0}_{I}$, where $s\hspace{-0.175em}\in\hspace{-0.175em}  \Gamma_I$ and $t\hspace{-0.175em}\in\hspace{-0.175em} [0,\varepsilon_0 \mathtt{d}(s))$, defined~by~${k(x)\hspace{-0.175em}\coloneqq\hspace{-0.175em} k(s)}$~and~${\mathtt{d}(x)\hspace{-0.175em}\coloneqq\hspace{-0.175em} \mathtt{d}(s)}$, 
    we next introduce the \textit{mollified distribution function (in  direction of $n$)} 
    \begin{align}\label{eq:limsup_case2.2}
        \smash{\widetilde{\mathtt{d}}_{\varepsilon}\coloneqq \max\{0,k\cdot n_\varepsilon\}\mathtt{d}\in C^{0,1}(\Sigma_{I}^{\varepsilon})\,,}
    \end{align} 
    which  
    satisfies  
    \begin{align*}
         \nabla \widetilde{\mathtt{d}}_{\varepsilon}=\mathtt{d}\{n_{\varepsilon}\nabla k+\nabla n_{\varepsilon}k\}\chi_{\{k\cdot n_{\varepsilon}\ge 0\}}+ \max\{0,k\cdot n_\varepsilon\}\nabla \mathtt{d}\quad \text{ a.e.\ in }\Sigma_{I}^{\varepsilon}\,,
    \end{align*}
    so that, due to $\vert n_{\varepsilon}\vert,\vert k\vert,\varepsilon\vert \nabla n_{\varepsilon}\vert\leq 1$ a.e.\ in $\mathbb{R}^d$, there holds
    \begin{align*}
   \begin{aligned}
        \vert \widetilde{\mathtt{d}}_{\varepsilon}\vert&\leq \vert\mathtt{d}\vert &&\quad \text{ a.e.\ in }\Sigma_{I}^{\varepsilon}\,,\\
        \vert \nabla \widetilde{\mathtt{d}}_{\varepsilon}\vert
        &\leq  \vert\mathtt{d}\vert \{\vert \nabla k\vert+\smash{\tfrac{1}{\varepsilon}}\}+\vert \nabla \mathtt{d}\vert
       &&\quad  \text{ a.e.\ in }\Sigma_{I}^{\varepsilon}\,,
         \end{aligned}
    \end{align*}
    and, thus,  there exists a constant $c_{n}>0$, independent of $\varepsilon\in (0,\varepsilon_0)$, such that 
    \begin{subequations}\label{eq:limsup_case2.4}
    \begin{align}\label{eq:limsup_case2.4.1}
        \|\widetilde{\mathtt{d}}_{\varepsilon}\|_{\infty,\Sigma_{I}^{\varepsilon}}&\leq \|\mathtt{d}\|_{\infty,\Gamma_{I}}\,, \\
    \label{eq:limsup_case2.4.2}
        \|\nabla \widetilde{\mathtt{d}}_{\varepsilon}\|_{\infty,\smash{\Sigma_{I}^{\varepsilon}}}&\leq \smash{\tfrac{c_{n}}{\varepsilon}}\,.
    \end{align}
    \end{subequations}

     Next, let the cut-off function $\varphi_\varepsilon\colon \mathbb{R}^d\to [0,1]$, for every $x\in \mathbb{R}^d$, be defined by
        \begin{align}\label{eq:limsup_case2.5}
            \varphi_\varepsilon(x)\coloneqq\begin{cases}
                1-\smash{\frac{\beta \widetilde{\mathtt{d}}_{\varepsilon}(x)\psi_\varepsilon(x)}{\varepsilon(1+\beta \widetilde{\mathtt{d}}_{\varepsilon}(x))\mathtt{d}(x)}}&\text{ if }x\in \Sigma^{\varepsilon}_{I}\,,\\[-0.25mm]
                1&\text{ if }x\in \overline{\Omega}\,,\\[-0.25mm]
                0&\text{ else}\,.
            \end{cases}
        \end{align} 
       By construction, the cut-off function \eqref{eq:limsup_case2.5} satisfies $\varphi_\varepsilon|_{\Omega^{\varepsilon}_{I}}\in H^{1,\infty}(\Omega^{\varepsilon}_{I})$ with
        \begin{subequations}\label{eq:limsup_case2.6}
        \begin{alignat}{2}\label{eq:limsup_case2.6.1}
            0\leq \varphi_\varepsilon&\leq 1 &&\quad\text{ in }\mathbb{R}^d\,,\\[-0.5mm]
            \varphi_\varepsilon&=1&&\quad\text{ in }\overline{\Omega}\,,\label{eq:limsup_case2.6.2}\\[-0.5mm]
            \varphi_\varepsilon&=\smash{\tfrac{1}{1+\beta \widetilde{\mathtt{d}}_{\varepsilon}}}&&\quad\text{ on }\Gamma_{I}^{\varepsilon}\,.\label{eq:limsup_case2.6.3}
        \end{alignat}
        \end{subequations}
        Moreover, we have that\vspace{-0.5mm}
        \begin{align*}
        \left.\begin{aligned}
            \nabla \varphi_{\varepsilon}&=-\smash{\tfrac{\beta}{\varepsilon(1+\beta \widetilde{\mathtt{d}}_{\varepsilon})^2\mathtt{d}^2}}\big\{\{\psi_{\varepsilon}\nabla \widetilde{\mathtt{d}}_{\varepsilon}+\widetilde{\mathtt{d}}_{\varepsilon}\nabla \psi_{\varepsilon}\}(1+\beta \widetilde{\mathtt{d}}_{\varepsilon})\mathtt{d}\\[-1mm]&\qquad\qquad\qquad\quad- \widetilde{\mathtt{d}}_{\varepsilon}\psi_{\varepsilon}\{\beta \mathtt{d}\nabla\widetilde{\mathtt{d}}_{\varepsilon}+(1+\beta \widetilde{\mathtt{d}}_{\varepsilon})\nabla\mathtt{d}\}\big\}
            \\[-1mm]&=-\smash{\tfrac{\beta}{\varepsilon(1+\beta \widetilde{\mathtt{d}}_{\varepsilon})\mathtt{d}}}\widetilde{\mathtt{d}}_{\varepsilon}\nabla \psi_{\varepsilon}-\smash{\tfrac{\beta}{\varepsilon(1+\beta \widetilde{\mathtt{d}}_{\varepsilon})^2\mathtt{d}^2}}\psi_{\varepsilon}\{ \mathtt{d}\nabla\widetilde{\mathtt{d}}_{\varepsilon}-\widetilde{\mathtt{d}}_{\varepsilon}(1+\beta \widetilde{\mathtt{d}}_{\varepsilon})\nabla\mathtt{d}\}
            \end{aligned}\quad\right\}\quad\textup{ a.e.\ in }\Sigma_{I}^{\varepsilon}\,,
        \end{align*}
        so that, using \eqref{lem:transversal_distance_function.1} with remainders $R_\varepsilon\in (L^\infty(\Sigma_{I}^{\varepsilon}))^d$, $\varepsilon\in (0,\varepsilon_0)$, as in Lemma \ref{lem:transversal_distance_function},\vspace{-0.5mm}
        \begin{align*}
            \left.\begin{aligned} 
            \vert\nabla \varphi_{\varepsilon}\vert&\leq \smash{\tfrac{\beta}{\varepsilon(1+\beta \widetilde{\mathtt{d}}_{\varepsilon})}}\smash{\tfrac{\widetilde{\mathtt{d}}_{\varepsilon}}{\widetilde{\mathtt{d}}}}\{1+\varepsilon\|R_\varepsilon\|_{\infty,\Sigma_{I}^{\varepsilon}}\}+\smash{\tfrac{\beta \|\mathtt{d}\|_{\infty,\Gamma_I}}{(1+\beta \widetilde{\mathtt{d}}_{\varepsilon})\mathtt{d}}}\{\vert \nabla\widetilde{\mathtt{d}}_{\varepsilon}\vert+\tfrac{\widetilde{\mathtt{d}}_{\varepsilon}}{\mathtt{d}}\vert\nabla\mathtt{d}\vert\}
            \\&\leq \tfrac{\beta}{\varepsilon(1+\beta \widetilde{\mathtt{d}}_{\varepsilon})}\tfrac{\widetilde{\mathtt{d}}_{\varepsilon}}{\widetilde{\mathtt{d}}}\{1+\varepsilon\|R_\varepsilon\|_{\infty,\Sigma_{I}^{\varepsilon}}\}+\tfrac{\beta \|\mathtt{d}\|_{\infty,\Gamma_I}^2}{\mathtt{d}_{\min}^2}\{\vert \nabla\widetilde{\mathtt{d}}_{\varepsilon}\vert+\vert\nabla\mathtt{d}\vert\}
            \end{aligned}\quad\right\}\quad\textup{ a.e.\ in }\Sigma_{I}^{\varepsilon}\,,
        \end{align*}
        and, thus, by the convexity of the function $(t\mapsto t^2)\colon \mathbb{R}\to \mathbb{R}$, for fixed, but arbitrary $\delta\in (0,1)$,\vspace{-0.5mm}
        \begin{align}\label{eq:limsup_case2.7}
            \vert\nabla \varphi_{\varepsilon}\vert^2\leq \tfrac{1}{\delta}\tfrac{\beta^2}{\varepsilon^2(1+\beta \widetilde{\mathtt{d}}_{\varepsilon})^2}\tfrac{\widetilde{\mathtt{d}}_{\varepsilon}^2}{\widetilde{\mathtt{d}}^2}\{1+\varepsilon\|R_\varepsilon\|_{\infty,\Sigma_{I}^{\varepsilon}}\}^2+\tfrac{1}{1-\delta}\tfrac{\beta^2 \|\mathtt{d}\|_{\infty,\Gamma_I}^4}{\mathtt{d}_{\min}^4}\{\vert\nabla\widetilde{\mathtt{d}}_{\varepsilon}\vert+\vert\nabla\mathtt{d}\vert\}^2\,.
        \end{align}
 Recall that since   $\Gamma_I$ is piece-wise flat, one can~find $\delta\in C^{0,1}(\Gamma_{I})$  such~that~for~every~${\ell=1,\ldots,L}$, one has that  $\delta>0$ in $\Gamma_{I}^{\ell} $, $\delta=0$ on $\partial\Gamma_{I}^{\ell} $, $\mathcal{N}_\delta(\Gamma_I^{\ell})\cap \mathrm{Med}(\partial\Omega)=\emptyset$, and $\mathcal{N}_\delta(\Gamma_I^{\ell})\cap \mathcal{N}_\delta(\Gamma_I^{\ell'})=\emptyset$~if~${\ell\neq \ell'}$. On the basis of the latter, for possibly smaller (but not relabelled) $\varepsilon_0>0$,   for every $\ell=1,\ldots,L$, one can find 
       a subset $\widetilde{\Gamma}_I^{\varepsilon,\ell}\subseteq \Gamma_I^{\ell}$ such that (\textit{cf}.\ Figure \ref{fig:limsup_case2})\vspace{-0.5mm}
       \begin{subequations}\label{eq:limsup_case2.8}
        \begin{align}\label{eq:limsup_case2.8.1}
           \sup_{\varepsilon\in (0,\varepsilon_0)}{\big\{\tfrac{1}{\varepsilon}\vert  \Gamma_I^{\ell}\setminus \widetilde{\Gamma}_I^{\varepsilon,\ell}\vert\big\}}<+\infty\,,&\\[-1mm]\label{eq:limsup_case2.8.2}
           \widetilde{\Sigma}_{I}^{\varepsilon,\ell}+B_\varepsilon^d(0)\subseteq \mathcal{N}_\delta(\Gamma_I^{\ell})\,,&\quad\text{where}\quad
            \widetilde{\Sigma}_{I}^{\varepsilon,\ell}\coloneqq \big\{s+t k(s)\mid s\in \widetilde{\Gamma}_I^{\varepsilon,\ell}\,,\; t\in [0,\varepsilon \mathtt{d}(s))\big\}\,. 
        \end{align} 
        \end{subequations}
      Then, due to \eqref{eq:limsup_case2.8.2} and \eqref{eq:grad_dist_hat}, we have that $\smash{\nabla\widehat{\textrm{dist}}(\cdot,\partial\Omega)}=n_{\ell}$~in~$\smash{\widetilde{\Sigma}_{I}^{\varepsilon,\ell}}+B_\varepsilon^d(0)$~for~all~$\ell=1,\ldots ,L$, so that $n_\varepsilon=n_\ell$ in $\smash{\widetilde{\Sigma}_{I}^{\varepsilon,\ell}}$~for~all~$\ell=1,\ldots ,L$, which implies that\enlargethispage{5mm}\vspace{-1mm}
        \begin{align}\widetilde{\mathtt{d}}_{\varepsilon}=\widetilde{\mathtt{d}}\quad \text{ in }\widetilde{\Sigma}_{I}^{\varepsilon}\coloneqq\bigcup_{\ell=1}^L{\widetilde{\Sigma}_{I}^{\varepsilon,\ell}}\,.\label{eq:limsup_case2.10}
        \end{align}\vspace{-5mm}

        \begin{figure}[H]
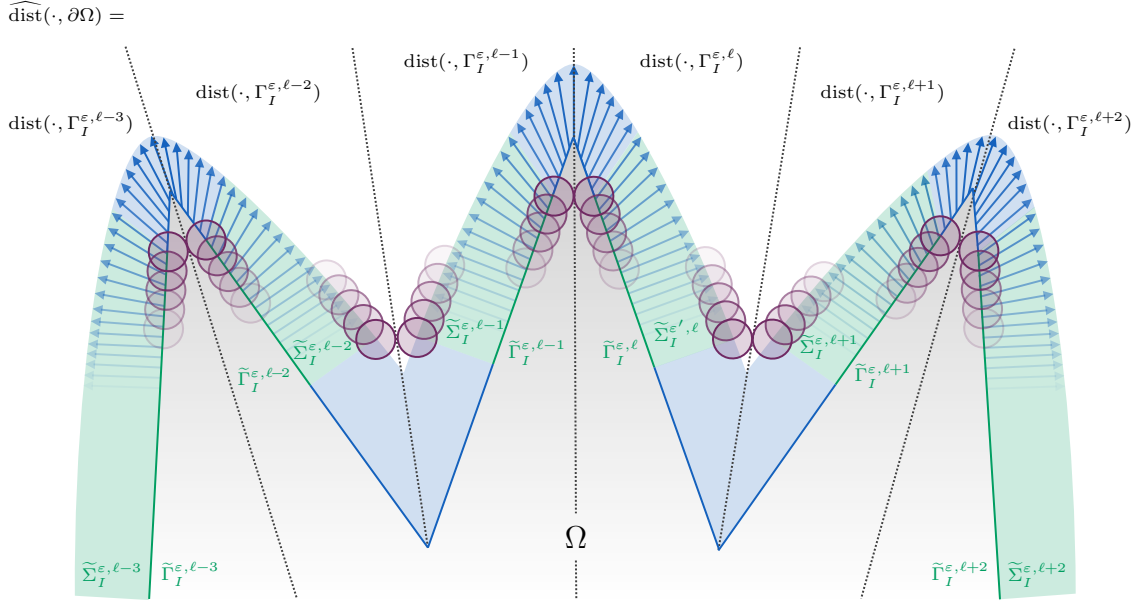

            \centering

  
\tikzset {_gvbm46y6a/.code = {\pgfsetadditionalshadetransform{ \pgftransformshift{\pgfpoint{0 bp } { 0 bp }  }  \pgftransformrotate{-90 }  \pgftransformscale{2 }  }}}
\pgfdeclarehorizontalshading{_wua27sj6k}{150bp}{rgb(0bp)=(0.89,0.89,0.89);
rgb(37.5bp)=(0.89,0.89,0.89);
rgb(37.5bp)=(0.86,0.86,0.86);
rgb(37.5bp)=(0.82,0.82,0.82);
rgb(62.5bp)=(1,1,1);
rgb(100bp)=(1,1,1)}
\tikzset{every picture/.style={line width=0.75pt}} 

\hspace*{-2.5mm}

            \caption{Schematic diagram of the construction in the proof of Lemma \ref{lem:limsup_case2}: \textit{(a)} local boundary parts $\widetilde{\Gamma}_I^{\varepsilon',\ell}$, $\ell=1,\ldots,L$, (\textcolor{shamrockgreen}{green lines}) (\textit{cf}.\ \eqref{eq:limsup_case2.8.1}) \textit{(b)}  local insulating layers $\widetilde{\Sigma}_{I}^{\varepsilon',\ell}$, $\ell=1,\ldots,L$, (\textcolor{shamrockgreen!50!white}{light green areas}) (\textit{cf}.\ \eqref{eq:limsup_case2.8.2}); \textit{(c)} medial axis $\textup{Med}(\partial\Omega)$ (\textcolor{darkgray}{dashed dark gray lines}); \textit{(d)} translations of the ball $B_\varepsilon^d(0)$ (\textcolor{byzantium}{purple discs}).}
            \label{fig:limsup_case2}
        \end{figure}
        \newpage

        \noindent Since, due to \eqref{eq:limsup_case2.8.1}, the truncated $\Sigma_I^{\varepsilon}\setminus \widetilde{\Sigma}_{I}^{\varepsilon}$ insulating layer has thickness proportional to $\varepsilon$ in two directions,
        there exists a constant $c_{\Omega}>0$, which depends only on the Lipschitz regularity~of~$\Gamma_I$, such that
        \begin{align}\label{eq:limsup_case2.11}
            \vert \Sigma_I^{\varepsilon}\setminus \widetilde{\Sigma}_{I}^{\varepsilon}\vert\leq c_{\Omega}\varepsilon^2\,.
        \end{align}
       Next, we establish that for $d\leq 4$, there holds
       \begin{align}\label{eq:limsup_case2.12}
            \tfrac{1}{\varepsilon}\|\overline{v}-u_\infty\|_{\smash{\Sigma_{I}^{\varepsilon}\setminus\widetilde{\Sigma}_{I}^{\varepsilon}}}^2\to 0\quad (\varepsilon\to 0^+)\,.
       \end{align}
       To this end, we distinguish the cases $d\in \{3,4\}$ and $d=2$:

       $\bullet$  \emph{Case $d\in \{3,4\}$.} In this case, 
       due to $H^1(\Omega)\hookrightarrow L^{\smash{\frac{2d}{d-2}}}(\Omega)$ with  $\frac{1}{2}=\frac{d-2}{2d}+\frac{1}{d}$ and \eqref{eq:limsup_case2.11}, we find that
       \begin{align*}
           \tfrac{1}{\varepsilon}\|\overline{v}-u_\infty\|_{\smash{\Sigma_{I}^{\varepsilon}\setminus\widetilde{\Sigma}_{I}^{\varepsilon}}}^2
           &\leq \tfrac{1}{\varepsilon}\|\overline{v}-u_\infty\|_{\frac{2d}{d-2},\smash{\Sigma_{I}^{\varepsilon}\setminus\widetilde{\Sigma}_{I}^{\varepsilon}}}^2\vert \Sigma_{I}^{\varepsilon}\setminus\widetilde{\Sigma}_{I}^{\varepsilon}\vert^{\frac{2}{d}}
            \\&\leq c_{\Omega}^2\varepsilon^{\smash{\frac{4}{d}-1}}\|\overline{v}-u_\infty\|_{\frac{2d}{d-2},\smash{\Sigma_{I}^{\varepsilon}\setminus\widetilde{\Sigma}_{I}^{\varepsilon}}}^2
            \to 0\quad(\varepsilon\to0^+)\,.
        \end{align*}

      $\bullet$  \emph{Case $d=2$.} In this case, 
       due to $H^1(\Omega)\hookrightarrow L^s(\Omega)$ with $\frac{1}{2}=\frac{1}{s}+\frac{s-2}{s2}$ for all $s>2$ and \eqref{eq:limsup_case2.11}, for  $s\ge 4$, due to $\frac{2s-4}{s}\ge 1$, we find that
       \begin{align*}
           \tfrac{1}{\varepsilon}\|\overline{v}-u_\infty\|_{\smash{\Sigma_{I}^{\varepsilon}\setminus\widetilde{\Sigma}_{I}^{\varepsilon}}}^2
           &\leq \tfrac{1}{\varepsilon}\|\overline{v}-u_\infty\|_{s,\smash{\Sigma_{I}^{\varepsilon}\setminus\widetilde{\Sigma}_{I}^{\varepsilon}}}^2\vert \Sigma_{I}^{\varepsilon}\setminus\widetilde{\Sigma}_{I}^{\varepsilon}\vert^{\frac{s-2}{s}}
            \\&\leq c_{\Omega}^2\varepsilon^{\smash{\frac{2s-4}{s}-1}}\|\overline{v}-u_\infty\|_{s,\smash{\Sigma_{I}^{\varepsilon}\setminus\widetilde{\Sigma}_{I}^{\varepsilon}}}^2
            \to 0\quad (\varepsilon\to 0^+)\,.
        \end{align*}
        From \eqref{eq:limsup_case2.12} together with \eqref{eq:limsup_case2.4.2}, in turn,  we infer that  
        \begin{align}\label{eq:limsup_case2.13}
        \begin{aligned}
            \varepsilon\|\nabla \widetilde{\mathtt{d}}_{\varepsilon}(\overline{v}-u_\infty)\|_{\Sigma_{I}^{\varepsilon}}^2&=\varepsilon\|\nabla \widetilde{\mathtt{d}}(\overline{v}-u_\infty)\|_{\widetilde{\Sigma}_{I}^{\varepsilon}}^2+\varepsilon\|\nabla \widetilde{\mathtt{d}}_{\varepsilon}(\overline{v}-u_\infty)\|_{\smash{\Sigma_{I}^{\varepsilon}\setminus\widetilde{\Sigma}_{I}^{\varepsilon}}}^2
            \\&\leq \varepsilon\|\nabla \widetilde{\mathtt{d}}(\overline{v}-u_\infty)\|_{\widetilde{\Sigma}_{I}^{\varepsilon}}^2+\smash{\tfrac{c_n^2}{\varepsilon}}\|\overline{v}-u_\infty\|_{\smash{\Sigma_{I}^{\varepsilon}\setminus\widetilde{\Sigma}_{I}^{\varepsilon}}}^2
            \to 0\quad (\varepsilon\to 0^+)\,.
            \end{aligned}
        \end{align}  
       Next, let the desired recovery sequence
        $v_\varepsilon\in L^2(\mathbb{R}^d)$,  for a.e.\ $x\in \Omega$, be defined by
        \begin{align*}
             v_\varepsilon(x)
             \coloneqq\begin{cases}
                \overline{v}(x)\varphi_\varepsilon(x)+u_{\infty}(x) (1-\varphi_\varepsilon(x))&\text{ if }x\in \Omega^{\varepsilon}_{I}\,,\\
                v(x)&\text{ else}\,.
            \end{cases}
        \end{align*}
        which,  by construction and $\varphi_\varepsilon|_{\Omega^{\varepsilon}_{I}}\in H^{1,\infty}(\Omega^{\varepsilon}_{I})$ with \eqref{eq:limsup_case2.6.2},\eqref{eq:limsup_case2.6.3}, satisfies $v_\varepsilon|_{\Omega^{\varepsilon}_{I}}\in H^1(\Omega^{\varepsilon}_{I})$~with\vspace{-3.5mm}
        \begin{subequations} \label{eq:limsup_case2.14}
        \begin{alignat}{2}  
            v_\varepsilon&=v&&\quad\text{ a.e.\ in }\mathbb{R}^d\setminus \Sigma_{I}^{\varepsilon} \,,\label{eq:limsup_case2.14.1}\\ 
            v_\varepsilon&=u_D&&\quad\text{ a.e.\ on }\smash{\Gamma_D} \,,\label{eq:limsup_case2.14.2}\\
            v_{\varepsilon}-u_{\infty}&=\tfrac{1}{1+\beta \widetilde{\mathtt{d}}_{\varepsilon}}\{\overline{v}-u_{\infty}\}&&\quad\text{ a.e.\ on }\smash{\Gamma_{I}^{\varepsilon}} \,.\label{eq:limsup_case2.14.3}
        \end{alignat}
        \end{subequations}
        Moreover, by the convexity of $(t\mapsto t^2)\colon \mathbb{R}\to \mathbb{R}$ and \eqref{eq:limsup_case2.7}, for fixed, but arbitrary $\delta\in (0,1)$, we have that
        \begin{align}\label{eq:limsup_case2.15}
        \begin{aligned} 
            \vert \nabla v_{\varepsilon}\vert^2&\leq \tfrac{1}{\delta}\vert \nabla \varphi_{\varepsilon}(\overline{v}-u_{\infty})\vert^2+\tfrac{1}{1-\delta}\vert \varphi_{\varepsilon}\nabla \overline{v}+(1-\varphi_{\varepsilon})\nabla u_{\infty}\vert^2
            \\&\leq \tfrac{1}{\delta}\big\{\tfrac{1}{\delta}\tfrac{\beta^2}{\varepsilon^2(1+\beta\widetilde{\mathtt{d}}_{\varepsilon})^2}\tfrac{\widetilde{\mathtt{d}}_{\varepsilon}^2}{\widetilde{\mathtt{d}}^2}\{1+\varepsilon\|R_\varepsilon\|_{\infty,\Sigma_{I}^{\varepsilon}}\}^2+\tfrac{1}{1-\delta}\tfrac{\beta^2\|\mathtt{d}\|_{\infty,\Gamma_{I}}^4}{\mathtt{d}_{\min}^4}\{\vert\nabla \widetilde{\mathtt{d}}_{\varepsilon}\vert+\vert\nabla \mathtt{d}\vert\}^2\big\}\vert\overline{v}-u_{\infty}\vert^2\\&\quad+\smash{\tfrac{1}{1-\delta}\{\vert\nabla \overline{v}\vert+\vert\nabla u_{\infty}\vert\}^2}\,.
            \end{aligned}
        \end{align}
        In particular, due to \eqref{eq:limsup_case2.14.1}, $\vert \Sigma_{I}^{\varepsilon}\vert\to 0$ $(\varepsilon\to0^+)$, and $\vert v_{\varepsilon}\vert \leq \vert v\vert +\vert u_{\infty}\vert$ a.e.\ in $\mathbb{R}^d$ (\textit{cf}.\ \eqref{eq:limsup_case2.6.1}), Lebesgue's dominated convergence theorem yields that
        \begin{align*}
            v_\varepsilon\to v\quad \text{ in }L^2(\mathbb{R}^d)\quad (\varepsilon\to\smash{0^+})\,.
        \end{align*}
        In addition, as a direct consequence of \eqref{eq:limsup_case2.14.1},\eqref{eq:limsup_case2.14.2}, we obtain
        \begin{align}\label{eq:limsup_case2.16}
             \smash{\overline{E}}^{\mathtt{d}}_\varepsilon(v_\varepsilon)= \tfrac{\lambda}{2}\|\nabla v\|_{\Omega}^2-(f,v)_{\Omega}-\langle g,v\rangle_{\smash{H^{\smash{\frac{1}{2}}}(\Gamma_N)}}+\tfrac{\varepsilon}{2}\|\nabla v_{\varepsilon}\|_{\smash{\Sigma^{\varepsilon}_{I}}}^2+\tfrac{\beta}{2}\|v_\varepsilon-u_{\infty}\|_{\Gamma^{\varepsilon}_{I}}^2\,,
        \end{align}
        so that it is left to treat the limit superior of the last two terms on the right-hand side~of~\eqref{eq:limsup_case2.16}. For the latter, it is sufficient establish that  
        \begin{subequations}\label{eq:limsup_case2.17}
        \begin{align}\label{eq:limsup_case2.17.1}
             \limsup_{\varepsilon\to \smash{\smash{0^+}}}{\big\{\tfrac{\varepsilon}{2}\|\nabla v_{\varepsilon}\|_{\smash{\Sigma^{\varepsilon}_{I}}}^2\big\}}&\leq \tfrac{\beta}{2}\|(\beta \widetilde{\mathtt{d}})^{\smash{\frac{1}{2}}}(1+\beta \widetilde{\mathtt{d}})^{-1}\{v-u_{\infty}\}\|_{\Gamma_{I}}^2\,,\\
             \limsup_{\varepsilon\to \smash{0^+}}{\big\{\tfrac{\beta}{2}\|v_{\varepsilon}-u_{\infty}\|_{\Gamma^{\varepsilon}_{I}}^2\big\}}&\leq \tfrac{\beta}{2}\|(1+\beta \widetilde{\mathtt{d}})^{-1}\{v-u_{\infty}\}\|_{\Gamma_{I}}^2\,,\label{eq:limsup_case2.17.2}
        \end{align}
        \end{subequations} 
        which jointly imply that
        \begin{align}\label{eq:limsup_case2.18}
             \limsup_{\varepsilon\to \smash{0^+}}{\big\{\tfrac{\varepsilon}{2}\|\nabla v_{\varepsilon}\|_{\smash{\Sigma^{\varepsilon}_{I}}}^2+\tfrac{\beta}{2}\|v_{\varepsilon}-u_{\infty}\|_{\Gamma^{\varepsilon}_{I}}^2\big\}}\leq \tfrac{\beta}{2}\|(1+\beta \widetilde{\mathtt{d}})^{-\smash{\frac{1}{2}}}\{v-u_{\infty}\}\|_{\Gamma_{I}}^2\,.
        \end{align}
 
        Therefore, let us next establish the $\limsup$-estimates \eqref{eq:limsup_case2.17.1} and \eqref{eq:limsup_case2.17.2} separately:

        \emph{ad \eqref{eq:limsup_case2.17.1}.}  Using \eqref{eq:limsup_case2.15},
            Lemma \ref{lem:Lebesgue_boundary_limit}\eqref{lem:Lebesgue_boundary_limit.0}, and \eqref{eq:limsup_case2.14.1}, because $\delta\in (0,1)$ was~chosen~arbitrarily, 
            we find that
        \begin{align*}
            \begin{aligned} 
             \limsup_{\varepsilon\to \smash{0^+}}{\big\{\tfrac{\varepsilon}{2}\|\nabla v_{\varepsilon}\|_{\smash{\Sigma^{\varepsilon}_{I}}}^2\big\}}&\smash{\overset{\eqref{eq:limsup_case2.15}}{\leq}}  \limsup_{\varepsilon\to \smash{0^+}}{\big\{\tfrac{1}{2\varepsilon}\tfrac{1}{\delta^2}\|\beta(1+\beta\widetilde{\mathtt{d}}_{\varepsilon})^{-1}\widetilde{\mathtt{d}}_{\varepsilon}\widetilde{\mathtt{d}}^{-1}\{\overline{v}-u_{\infty}\}\|_{\smash{\Sigma^{\varepsilon}_{I}}}^2\big\}}\\[-1mm]&\quad\hspace*{1.5mm}+
             \limsup_{\varepsilon\to \smash{0^+}}{\big\{\tfrac{\varepsilon}{2}\tfrac{1}{\delta(1-\delta)}\tfrac{\beta^2\|\mathtt{d}\|_{\infty,\Gamma_{I}}^4}{\mathtt{d}_{\min}^4}\|\{\vert \nabla \widetilde{\mathtt{d}}_{\varepsilon}\vert+\vert\nabla \mathtt{d}\vert\}\{\overline{v}-u_{\infty}\}\|_{\smash{\Sigma^{\varepsilon}_{I}}}^2\big\}}
             \\&\quad\hspace*{1.5mm}+\limsup_{\varepsilon\to \smash{0^+}}{\big\{\tfrac{\varepsilon}{2}\tfrac{1}{1-\delta}\{\|\nabla \overline{v} \|_{\smash{\Sigma^{\varepsilon}_{I}}}+\|\nabla u_\infty\|_{\smash{\Sigma^{\varepsilon}_{I}}}\}^2\big\}} 
              \\&\hspace*{1.5mm}\leq \limsup_{\varepsilon\to \smash{0^+}}{\big\{\tfrac{1}{\delta^2}\tfrac{\beta}{2\varepsilon}
             \|\beta^{\smash{\frac{1}{2}}}(1+\beta \widetilde{\mathtt{d}})^{-1}\{ \overline{v}-u_{\infty}\}\|_{\smash{\Sigma^{\varepsilon}_{I}}}^2\big\}}
             \\&\hspace*{1mm}\smash{\overset{\eqref{lem:Lebesgue_boundary_limit.0}}{=}}\tfrac{1}{\delta^2}\tfrac{\beta}{2}\|(\beta\widetilde{\mathtt{d}})^{\smash{\frac{1}{2}}}(1+\beta \widetilde{\mathtt{d}})^{-1}\{v-u_{\infty}\}\|_{\Gamma_{I}}^2
             \\&\hspace*{1.5mm}\to \tfrac{\beta}{2}\|(\beta\widetilde{\mathtt{d}})^{\smash{\frac{1}{2}}}(1+\beta \widetilde{\mathtt{d}})^{-1}\{v-u_{\infty}\}\|_{\Gamma_{I}}^2\quad (\delta \to 1^-)\,,
             \end{aligned} 
        \end{align*} 
        where we used in the second inequality that, due to \eqref{eq:limsup_case2.10}, we have that
        \begin{align*}
             \|(1+\beta \widetilde{\mathtt{d}}_\varepsilon)^{-1}\widetilde{\mathtt{d}}_\varepsilon\widetilde{\mathtt{d}}^{-1}\{ \overline{v}-u_{\infty}\}\|_{\smash{\Sigma^{\varepsilon}_{I}}}&\leq  \|(1+\beta \widetilde{\mathtt{d}})^{-1}\{ \overline{v}-u_{\infty}\}\|_{\widetilde{\Sigma}^{\varepsilon}_{I}}\\&\quad+ \|(1+\beta \widetilde{\mathtt{d}}_\varepsilon)^{-1}\widetilde{\mathtt{d}}_\varepsilon\widetilde{\mathtt{d}}^{-1}\{ \overline{v}-u_{\infty}\}\|_{\smash{\Sigma^{\varepsilon}_{I}\setminus\widetilde{\Sigma}^{\varepsilon}_{I}}}\
             \\&\leq \|(1+\beta \widetilde{\mathtt{d}})^{-1}\{ \overline{v}-u_{\infty}\}\|_{\smash{\Sigma^{\varepsilon}_{I}}}+ \tfrac{\|\mathtt{d}\|_{\infty,\Gamma_{I}}}{\mathtt{d}_{\min}}\|\overline{v}-u_{\infty}\|_{\smash{\Sigma^{\varepsilon}_{I}\setminus\widetilde{\Sigma}^{\varepsilon}_{I}}}\,,
        \end{align*}
        together with \eqref{eq:limsup_case2.12}.

         \emph{ad \eqref{eq:limsup_case2.17.2}.} Using \eqref{eq:limsup_case2.14.3}, the approximative transformation formula (\textit{cf}.\ Lemma \ref{lem:approx_trans_formula2}),~and~that
         \begin{align*}
             \{\overline{v}-u_{\infty}\}(\cdot+\varepsilon\mathtt{d}k)\to \overline{v}-u_{\infty}=v-u_{\infty}\quad \text{ in }L^2(\Gamma_I)\quad (\varepsilon\to \smash{0^+})\,,
         \end{align*}
         which, similar to \eqref{eq:liminf_case1.6}, using Corollary \ref{lem:poincare_int}\eqref{lem:poincare_1}, follows from\enlargethispage{6mm}
        \begin{align*}
            \|\{\overline{v}-u_{\infty}\}(\cdot+\varepsilon \mathtt{d} k)-\{\overline{v}-u_{\infty}\}\|_{\smash{\Gamma_{I}}}^2\leq  \tfrac{\varepsilon\mathtt{d}_{\min}}{\kappa-\varepsilon\|\mathtt{d}\|_{\infty,\smash{\Gamma_{I}}}\|R_{\varepsilon}\|_{\infty,\smash{D_{I}^{\varepsilon}}}}\|\nabla\{ \overline{v}-u_{\infty}\}\|_{\smash{\Sigma_{I}^{\varepsilon}}}^2\to 0\quad (\varepsilon\to \smash{0^+})\,,
        \end{align*}
       setting $\smash{\widetilde{\Gamma}_{I}^{\varepsilon}\coloneqq\bigcup_{\ell=1}^L{\widetilde{\Gamma}_{I}^{\ell,\varepsilon}}}$ and using \eqref{eq:limsup_case2.11}, we find that
        \begin{align*}
            \begin{aligned} 
            \limsup_{\varepsilon\to \smash{0^+}}{\big\{\tfrac{\beta}{2}\|v_{\varepsilon}-u_{\infty}\|_{\Gamma^{\varepsilon}_{I}}^2\big\}}&= \limsup_{\varepsilon\to \smash{0^+}}{\big\{\tfrac{\beta}{2}\|(1+\beta \widetilde{\mathtt{d}}_{\varepsilon})^{-1}\{\overline{v}-u_{\infty}\}(\cdot+\varepsilon \mathtt{d} k)\|_{\Gamma_{I}}^2\big\}}
            \\&\leq \limsup_{\varepsilon\to \smash{0^+}}{\big\{\tfrac{\beta}{2}\|(1+\beta \widetilde{\mathtt{d}})^{-1}\{\overline{v}-u_{\infty}\}(\cdot+\varepsilon \mathtt{d} k)\|_{\widetilde{\Gamma}_{I}^{\varepsilon}}^2\big\}}
            \\[-0.5mm]&\quad+\limsup_{\varepsilon\to \smash{0^+}}{\big\{\tfrac{\beta}{2}\|\{\overline{v}-u_{\infty}\}(\cdot+\varepsilon \mathtt{d} k)\|_{\Gamma_{I}\setminus\widetilde{\Gamma}_{I}^{\varepsilon}}^2\big\}}
            \\&\leq \limsup_{\varepsilon\to \smash{0^+}}{\big\{\tfrac{\beta}{2}\|(1+\beta \widetilde{\mathtt{d}})^{-1}\{\overline{v}-u_{\infty}\}(\cdot+\varepsilon \mathtt{d} k)\|_{\Gamma_{I}}^2\big\}}
            \\[-0.5mm]&\quad+\limsup_{\varepsilon\to \smash{0^+}}{\big\{\tfrac{\beta}{2}\|\{\overline{v}-u_{\infty}\}(\cdot+\varepsilon \mathtt{d} k)-\{v-u_{\infty}\}\|_{\Gamma_{I}}^2\big\}}
            \\[-0.5mm]&\quad+\limsup_{\varepsilon\to \smash{0^+}}{\big\{\beta\|v-u_{\infty}\|_{\Gamma_{I}\setminus\widetilde{\Gamma}_{I}^{\varepsilon}}^2\big\}}
            \\&\leq \tfrac{\beta}{2}\big\|(1+\beta \widetilde{\mathtt{d}})^{-1}\{v-u_{\infty}\}\|_{\Gamma_{I}}^2\,.
            \end{aligned}
        \end{align*} 
        In summary, from \eqref{eq:limsup_case2.17.1} and \eqref{eq:limsup_case2.17.2}, it follows  \eqref{eq:limsup_case2.18}, which together~with~\eqref{eq:limsup_case2.16} confirms the claimed $\limsup$-estimate for the Case \hyperlink{Case 1}{1}. 
    \end{proof}\newpage
  
	{\setlength{\bibsep}{0pt plus 0.0ex}\small
		
		\bibliographystyle{aomplain}
		\bibliography{references}
		
	}
	
\end{document}